\documentclass[11pt, a4paper]{article}

\usepackage{typearea}
\typearea{12}

\usepackage{amsmath}
\usepackage{amsthm}
\usepackage{amssymb}
\usepackage{color}
\usepackage{url}
\usepackage{graphicx}
\usepackage{float}
\usepackage{comment}
\usepackage{ifthen}
\usepackage{mathdots}
\usepackage{algorithm}
\usepackage{algorithmic}

\theoremstyle{definition}
\newtheorem{thm}{Theorem}
\newtheorem{prop}[thm]{Proposition}
\newtheorem{lem}[thm]{Lemma}

\newtheorem{ex}{Example}
\newtheorem{rem}{Remark}

\allowdisplaybreaks[4]

\numberwithin{equation}{section}
\numberwithin{thm}{section}
\numberwithin{rem}{section}

\title{
Generation of point sets by convex optimization 
for interpolation in reproducing kernel Hilbert spaces}
\author{
Ken'ichiro Tanaka%
\footnote{
Department of Mathematical Informatics,
Graduate School of Information Science and Technology,
University of Tokyo. 
7-3-1 Hongo, Bunkyo-ku, Tokyo, 113-8656, Japan. 
\texttt{kenichiro@mist.i.u-tokyo.ac.jp}}
}
\date{April 30, 2019}

\begin{document}

\maketitle

\begin{abstract}
We propose algorithms to take point sets 
for kernel-based interpolation of functions in reproducing kernel Hilbert spaces (RKHSs)
by convex optimization. 
We consider the case of kernels with the Mercer expansion 
and propose an algorithm by deriving a second-order cone programming (SOCP) problem 
that yields $n$ points at one sitting for a given integer $n$. 
In addition, 
by modifying the SOCP problem slightly, 
we propose another sequential algorithm that adds an arbitrary number of new points in each step.
Numerical experiments show that in several cases the proposed algorithms
compete with the $P$-greedy algorithm, 
which is known to provide nearly optimal points.

\end{abstract}


\smallskip


\section{Introduction}	

We propose algorithms to take point sets 
for kernel-based interpolation of functions in reproducing kernel Hilbert spaces (RKHSs).
The algorithms are based on a convex optimization problem
derived by the Mercer series of the kernel. 

Kernel-based methods are useful in various fields of scientific computing and engineering. 
Their details and applications can be found in the references 
\cite{bib:Buhmann_RBF_2000,bib:Fass_Meshfree2007,bib:FassMc_Kernel_2015,bib:SchabackWendland_ActaNumerica_2006,bib:Wendland_SDA_2005}. 
Many of the applications include approximation of multivariate functions by scattered data 
as a fundamental component. 
Therefore we focus on kernel interpolation in this paper. 
Let $d$ be a positive integer and let $\Omega \subset \mathbf{R}^{d}$ be a region.  
For a finite set $\{ x_{1}, \ldots, x_{n} \} \subset \Omega$ of $n$ pairwise distinct points 
and a set $\{ y_{1}, \ldots , y_{n} \} \subset \mathbf{R}$, 
we aim to find a function $s: \Omega \to \mathbf{R}$ 
with the interpolation condition $s(x_{i}) = y_{i}\ (i=1,\ldots, n)$. 
To this end, 
we consider a positive definite kernel $K: \Omega \times \Omega \to \mathbf{R}$
and the function $s$ given by 
\begin{align}
s(x) = \sum_{j=1}^{n} c_{j} \, K(x, x_{j}), 
\label{eq:K_interp}
\end{align}
where $c_{j}$ are determined by the interpolation condition. 

For a positive definite kernel $K$, 
there exists a unique Hilbert space $\mathcal{H}_{K}(\Omega)$ 
of functions from $\Omega$ to $\mathbf{R}$ 
where $K$ is a reproducing kernel, that is, 
\begin{align}
& \text{(a)} \  \forall x \in \Omega, \ K(\cdot, x) \in \mathcal{H}_{K}(\Omega),\\
& \text{(b)} \  \forall x \in \Omega, \ \forall f \in \mathcal{H}_{K}(\Omega), \ \langle f, K(\cdot, x) \rangle_{\mathcal{H}_{K}(\Omega)} = f(x).
\end{align}
The reproducing kernel Hilbert space (RKHS) $\mathcal{H}_{K}(\Omega)$ is called the native space of $K$. 
As the kernel interpolation, we consider the case that $y_{j} = f(x_{j})$ for $f \in \mathcal{H}_{K}(\Omega)$. 
Then, we encounter a problem of approximating functions in $\mathcal{H}_{K}(\Omega)$ by the kernel interpolation. 

Such approximation is widely used, for example, 
for the numerical solution of partial differential equations via collocation 
\cite{bib:Nguyen_etal_Meshfree_survey_2008}. 
In such situations, 
we can choose the point set $\{ x_{1}, \ldots, x_{n} \}$ to achieve accurate computation. 
The choice is critical for the kernel interpolation. 
Actually, 
there are several fundamental results about the approximation power of the kernel interpolation 
for well-distributed points \cite[Chapter 11]{bib:Wendland_SDA_2005}. 
Based on these, 
many methods for choosing well-distributed points have been proposed: 
geometric greedy algorithm \cite{bib:DeMarchi2003,bib:DeMarchi2009,bib:DeMarchi_etal_2005}, 
$P$-greedy algorithm \cite{bib:DeMarchi2003,bib:DeMarchi_etal_2005,bib:SantinHaasdonk2017}, 
$f$-greedy algorithm \cite{bib:SchabackWendland_fgreedy_2000}, and 
$f/P$-greedy algorithm~\cite{bib:Muller_PhD_2009}. 
Furthermore, 
there are methods based on optimal designs and similar concepts to them. 
In the paper \cite{bib:Briani_etal_Fekete_2012}, 
the authors compute (approximate) Fekete points, 
the points maximizing the Vandermonde determinants with respect to a polynomial basis. 
For the computation, they use the routines provided in Matlab for general nonlinear optimization. 
For other methods, see~\cite[Appendix B]{bib:FassMc_Kernel_2015} and the references therein. 

In this paper, 
we focus on such methods that depend just on the space $\mathcal{H}_{K}(\Omega)$
and provide universal interpolation on it. 
The $P$-greedy algorithm and the algorithms for the Fekete points 
are in this category. 
The $P$-greedy algorithm is a sequential one adding a new point one by one in each step
as described in Section~\ref{sec:prelim}. 
It is easy to implement and known to be nearly optimal~\cite{bib:SantinHaasdonk2017}. 
On the other hand, 
the others are non-sequential algorithms providing a point set at one sitting. 
The Fekete points for a polynomial basis are known to be well-distributed, 
although they are given by non-convex optimization. 
Therefore we expect that algorithms of this type for a general kernel $K$ 
would provide good point sets for kernel interpolation in $\mathcal{H}_{K}(\Omega)$. 
To investigate whether this expectation is the case or not, 
we propose an algorithm for generating point sets by 
approximately maximizing the determinant of the kernel matrix $\mathcal{K} = (K(x_{i}, x_{j}))_{ij}$.
More precisely, 
we show that they can be computed 
by \textit{convex optimization} in the case of kernels with the Mercer expansion. 
Actually, we can formulate a second-order cone programming (SOCP) problem 
yielding point sets. 
As a by-product of this formulation, 
we provide a variant sequential algorithm that adds an arbitrary number of new points in each step.
By some numerical experiments, 
we observe that 
the point sets provided by the proposed algorithms are well-distributed 
and compete with the points provided by the $P$-greedy algorithm in several cases. 

The rest of this paper is organized as follows. 
In Section~\ref{sec:prelim}, 
we summarize the fundamental facts of positive definite kernels and reproducing kernel Hilbert spaces. 
In Section~\ref{sec:opt_prob_rkhs}, 
we derive a relaxed $D$-optimal design problem to generate a point set by using the Mercer expansion of a kernel. 
Then, according to the paper \cite{bib:Sagnol_etal_OptDesign_SOCP_2015}, 
we describe the equivalent formulation of the design problem as a SOCP problem in Section~\ref{sec:reform_by_SOCP}. 
We show the results of numerical experiments in Section~\ref{sec:num_expr} 
and conclude this paper by Section~\ref{sec:concl}.

\section{Interpolation in reproducing kernel Hilbert spaces}
\label{sec:prelim}

\subsection{Positive definite kernels}

Let $\Omega \subset \mathbf{R}^{d}$ be a region and
let $K: \Omega \times \Omega \to \mathbf{R}$ be a continuous, symmetric, and positive definite  kernel. 
That is, for any set $\mathcal{X}_{n} = \{ x_{1}, \ldots , x_{n} \} \subset \Omega$ of distinct points, 
the kernel matrix 
\begin{align}
\mathcal{K} = 
\begin{bmatrix}
K(x_{1}, x_{1}) & \cdots & K(x_{n}, x_{1}) \\ 
\vdots & \ddots & \vdots \\
K(x_{1}, x_{n}) & \cdots & K(x_{n}, x_{n})
\end{bmatrix}
\end{align}
is positive definite. 
Throughout this paper, 
we assume that $K$ has the Mercer series expansion given by
\begin{align}
K(x, y) = \sum_{\ell = 1}^{\infty} \lambda_{\ell} \, \varphi_{\ell}(x) \, \varphi_{\ell}(y), 
\label{eq:MercerExpansion}
\end{align}
where $\lambda_{1} \geq \lambda_{2} \geq \cdots > 0$, and 
$\{ \varphi_{\ell} \}$ is an orthonormal system in $L^{2}(\Omega, \rho)$, 
the Lebesgue space with a density function $\rho$. 
This expansion is known to be possible on appropriate conditions \cite[Theorem 2.2]{bib:FassMc_Kernel_2015}.

\begin{ex}[{Brownian motion kernel on $\Omega = [0,1] \subset \mathbf{R}$ \cite[Section A.1.2]{bib:FassMc_Kernel_2015}}]
\label{ex:kernel_Brownian}
The kernel 
\begin{align}
K(x, y) = \min\{x, y \} \qquad (x, y \in [0,1])
\notag
\end{align}
is the Brownian motion kernel on the domain $\Omega = [0,1]$. 
Its Mercer series is given by letting
\begin{align}
\lambda_{\ell} = \frac{4}{(2\ell-1)^{2} \pi^{2}}, \quad 
\varphi_{\ell}(x) = \sqrt{2} \sin \left( (2\ell-1)\frac{\pi x}{2} \right)
\notag
\end{align}
in \eqref{eq:MercerExpansion}.
\end{ex}

\begin{ex}[Spherical inverse multiquadric kernel on $\Omega = S^{2}$ {\cite[Section A.9.1]{bib:FassMc_Kernel_2015}}]
\label{ex:kernel_Spherical}
Let $\gamma$ be a real number with $0 < \gamma < 1$. 
The kernel 
\begin{align}
K(x,y) = \frac{1}{\sqrt{1+\gamma^{2} - 2\gamma \, x^{T} y}}
\qquad
(x,y \in S^{2})
\notag
\end{align}
is the spherical inverse multiquadric kernel on 
the sphere $S^{2} = \{ \| x \| = 1 \mid x \in \mathbf{R}^{3} \}$. 
Its Mercer series is given by
\begin{align}
K(x,y) = \sum_{n=0}^{\infty} \frac{4\pi \gamma^{n}}{2n+1} \sum_{\ell = 1}^{2n+1} Y_{n, \ell}(x) Y_{n, \ell}(y), 
\label{eq:sphere_Mercer}
\end{align}
where $Y_{n, \ell}$ are the spherical harmonics \cite{bib:AtkinsonHan_SpHarm_2012, bib:DaiXu_SpApprox_2013}. 
\end{ex}

\begin{ex}[Gaussian kernel on $\Omega \subset \mathbf{R}^{d}$ {\cite[Section 12.2.1]{bib:FassMc_Kernel_2015}}]
\label{ex:kernel_Gaussian}
Let $\varepsilon$ be a positive real number. 
The kernel
\begin{align}
K(x,y) = \exp(- \varepsilon^{2} \| x - y \|^{2})
\notag
\end{align}
is the multivariate Gaussian kernel. 
By letting $\beta = (1+(2\varepsilon/\alpha)^{2})^{1/4}$ and $\delta = \alpha^{2}(\beta^{2}-1)/2$ for a positive real number $\alpha$, 
its Mercer series is given by
\begin{align}
K(x,y) = \sum_{\boldsymbol{n} \in \mathbf{N}^{d}} 
\boldsymbol{\lambda}_{\boldsymbol{n}} \, \boldsymbol{\varphi}_{\boldsymbol{n}}(x) \, \boldsymbol{\varphi}_{n}(y), 
\qquad
\boldsymbol{\lambda}_{\boldsymbol{n}} = \prod_{\ell=1}^{d} \lambda_{n_{\ell}}, 
\qquad 
\boldsymbol{\varphi}_{\boldsymbol{n}}(x) = \prod_{\ell=1}^{d} \varphi_{n_{\ell}}(x_{\ell}), 
\label{eq:Gauss_Mercer}
\end{align}
where 
\begin{align}
\lambda_{n} = 
\sqrt{\frac{\alpha^{2}}{\alpha^{2}+\delta^{2}+\varepsilon^{2}}}
\left( \frac{\varepsilon^{2}}{\alpha^{2}+\delta^{2}+\varepsilon^{2}} \right)^{n-1},
\qquad
\varphi_{n}(x) = \sqrt{\frac{\beta}{2^{n-1} \Gamma(n)}} \, \mathrm{e}^{-\delta^{2} x^{2}} H_{n-1}(\alpha \beta x).
\label{eq:Gauss_eigen_val_func}
\end{align}
The functions $H_{n}$ are Hermite polynomials of degree $n$. 
\end{ex}

\subsection{Interpolation of functions in reproducing kernel Hilbert spaces}

Let $\mathcal{H}_{K}(\Omega)$ be 
the native space of $K$ on $\Omega$. 
We denote the inner product of $\mathcal{H}_{K}(\Omega)$ by
$\langle \ , \ \rangle_{\mathcal{H}_{K}(\Omega)}$. 
Then, the kernel $K$ satisfies the reproducing property:
\begin{align}
f \in \mathcal{H}_{K}(\Omega),
\quad
f(x) = \langle f, K(\cdot, x) \rangle_{\mathcal{H}_{K}(\Omega)}. 
\label{eq:R_prop}
\end{align}

For a function $f \in \mathcal{H}_{K}(\Omega)$, 
we consider its interpolant of the form
\begin{align}
s_{f}(x) = \sum_{j = 1}^{n} c_{j} K(x, x_{j}), 
\qquad 
x \in \Omega \subset \mathbf{R}^{d},
\label{eq:kernel_interp}
\end{align}
where $\mathcal{X}_{n} = \{x_{1}, \ldots, x_{n} \} \subset \Omega$ 
is a set of points for interpolation. 
The coefficients $c_{j}$ are determined by the interpolation equations
\begin{align}
s_{f}(x_{i}) = f(x_{i}), \quad i = 1,\ldots, n.
\end{align}
The interpolant in \eqref{eq:kernel_interp} can be rewritten in the form
\begin{align}
s_{f}(x) = \sum_{j = 1}^{n} f(x_{j})\, u_{j}(x), 
\end{align}
where $u_{j}$ are cardinal bases satisfying the Lagrange property
\begin{align}
u_{j}(x_{i}) = \delta_{ij}, 
\qquad
i,j = 1,\ldots, n.
\end{align}
The vector $u(x) = (u_{1}(x), \ldots , u_{n}(x))^{T}$ is determined by the linear equation
\begin{align}
\mathcal{K}\, u(x) = k(x), 
\label{eq:lin_eq_for_cardinal_u}
\end{align}
where 
\(
k(x) = 
(K(x, x_{1}), \ldots , K(x, x_{n}))^{T}
\). 

Using the reproducing property in~\eqref{eq:R_prop}, 
we can derive a well-known error bound of the interpolation as follows:
\begin{align}
| f(x) - s_{f}(x) | 
& = 
\left| f(x) - \sum_{j=1}^{n} f(x_{j}) \, u_{j}(x) \right| 
= 
\left| 
\left \langle f, K(\cdot, x) 
-
\sum_{j=1}^{n} K(\cdot, x_{j}) \, u_{j}(x)
\right \rangle_{\mathcal{H}_{K}(\Omega)}
\right| \notag \\
& =
\left| 
\left \langle f, K(\cdot, x) -
k(\cdot)^{T} \mathcal{K}^{-1} \, k(x)
\right \rangle_{\mathcal{H}_{K}(\Omega)}
\right| \notag \\
& \leq
\| f \|_{\mathcal{H}_{K}(\Omega)}
\left \| K(\cdot, x) - 
k(\cdot)^{T} \mathcal{K}^{-1} \, k(x)
\right \|_{\mathcal{H}_{K}(\Omega)} \notag \\
& = 
\| f \|_{\mathcal{H}_{K}(\Omega)} \, P_{K, \mathcal{X}_{n}}(x), 
\end{align}
where
\begin{align}
P_{K, \mathcal{X}_{n}}(x) = 
\sqrt{K(x, x) - 
k(x)^{T} \mathcal{K}^{-1} \, k(x) }
\label{eq:power_function}
\end{align}
is called the power function. 
The set of its zeros is $\mathcal{X}_{n} = \{ x_{1}, \ldots, x_{n} \}$. 

By using the power function, 
the $P$-greedy algorithm generating points for interpolation is described as follows%
\footnote{If $K$ is translation-invariant, we can choose a starting point $x_{1}$ arbitrarily.}. 
Start with $\mathcal{X}_{1} = \{ x_{1} \}$ for a point $x_{1} \in \Omega$ maximizing $\sqrt{K(x,x)}$, 
and 
\begin{align}
\mathcal{X}_{j} = \mathcal{X}_{j-1} \cup \{ x_{j} \} \quad \text{with} 
\quad 
P_{K, \mathcal{X}_{j-1}}(x_{j}) = \max_{x \in \Omega} P_{K, \mathcal{X}_{j-1}}(x)
\quad 
(j = 2,\ldots, n).
\notag
\end{align}

\section{Optimization problems to generate good point configurations}
\label{sec:opt_prob_rkhs}

\subsection{Upper bound of the power function}

If we can find a minimizer $\mathcal{X}_{n}^{\dagger} = \{ x_{1}^{\dagger}, \ldots, x_{n}^{\dagger} \}$ 
of the worst case error
\begin{align}
\max_{x \in \Omega} P_{K, \mathcal{X}_{n}}(x), 
\label{eq:original_obj_func}
\end{align}
we can take $\mathcal{X}_{n}^{\dagger}$ as one of the best point sets for the interpolation. 
However, this optimization problem is difficult. 
Then, we provide an upper bound of the power function 
and consider its minimization to obtain an approximate minimizer of the value in \eqref{eq:original_obj_func}. 
To this end, 
We start with providing an expression of the power function $P_{K, \mathcal{X}_{n}}(x)$ 
by using determinants of matrices. 

\begin{prop}[{\cite[\S 14.1.1]{bib:FassMc_Kernel_2015}}, \cite{bib:DeMarchi2003}]
\label{prop:pow_func_expr}
The power function $P_{K, \mathcal{X}_{n}}(x)$ in \eqref{eq:power_function} satisfies that
\begin{align}
P_{K, \mathcal{X}_{n}}(x)
= 
\left(
\frac{1}{\det \mathcal{K}} \, 
\det
\begin{bmatrix}
K(x,x) & k(x)^{T} \\
k(x) & \mathcal{K}
\end{bmatrix}
\right)^{1/2}.
\label{eq:pow_func_det_expr}
\end{align}
\end{prop}

\begin{proof}
The assertion is shown as follows:
\begin{align*}
P_{K, \mathcal{X}_{n}}({x})^{2}
& = 
K({x}, {x})
-
{k}({x})^{T}
\mathcal{K}^{-1} \, 
{k}({x}) \\
&=
\frac{1}{\det \mathcal{K}}
\det
\begin{bmatrix}
K({x}, {x})
-
{k}({x})^{T}
\mathcal{K}^{-1} \, 
{k}({x}) & {k}({x})^{T}\\
\boldsymbol{0} & \mathcal{K}
\end{bmatrix} \\
&=
\frac{1}{\det \mathcal{K}}
\det
\left(
\begin{bmatrix}
K({x},{x}) & {k}({x})^{T} \\
{k}({x}) & \mathcal{K}
\end{bmatrix} 
\begin{bmatrix}
1 & \boldsymbol{0} \\
-\mathcal{K}^{-1} {k}({x}) & I
\end{bmatrix} 
\right) \\
&= 
\frac{1}{\det \mathcal{K}} \, 
\det
\begin{bmatrix}
K({x},{x}) & {k}({x})^{T} \\
{k}({x}) & \mathcal{K}
\end{bmatrix}.
\end{align*}
\end{proof}

\noindent
Next, 
for a positive integer $r$ we define $\kappa_{r}$ by 
\begin{align}
\kappa_{r} := \max_{\mathcal{X}_{r} = \{ x_{1}, \ldots, x_{r} \} \subset \Omega} 
\det (K(x_{i}, x_{j}))_{ij}. 
\end{align}
We call a maximizer providing the value of $\kappa_{r}$ a set of Fekete type points, 
because it resembles the set of the Fekete points that maximizes the determinant of a Vandermonde matrix. 
Then, for any $x \in \Omega$ we have
\begin{align}
P_{K, \mathcal{X}_{n}}({x})^{2}
\leq \frac{\kappa_{n+1}}{\det \mathcal{K}}, 
\label{eq:ub_wce_by_Fekete}
\end{align}
in which the square root of the RHS gives an upper bound of the worst case error in~\eqref{eq:original_obj_func}. 
Therefore, we consider the maximization of $\det \mathcal{K}$ with respect to 
$\mathcal{X}_{n} = \{ x_{1}, \ldots, x_{n} \}$. 
However, because this maximization problem is still difficult, we consider
\begin{enumerate}
\item
approximation of the determinant $\det \mathcal{K}$ by using the Mercer expansion in~\eqref{eq:MercerExpansion}, and
\item
approximate reduction of the maximization problem to a convex optimization problem. 
\end{enumerate}
We show these procedures in Sections~\ref{sec:app_det_Mercer} and~\ref{sec:reduce_convex}, 
respectively. 

\subsection{Approximation of the determinant $\det \mathcal{K}$}
\label{sec:app_det_Mercer}

Based on Proposition~\ref{prop:pow_func_expr} 
and the expansion of the kernel $K$ in~\eqref{eq:MercerExpansion}, 
we provide an approximation of $P_{K, \mathcal{X}_{n}}(x)$. 
By truncating the expansion, we have
\begin{align}
\mathcal{K}
& \approx 
\left( \sum_{\ell = 1}^{n} \lambda_{\ell} \, \varphi_{\ell}(x_{j}) \, \varphi_{\ell}(x_{i}) \right)_{ij} \notag \\
& = 
\begin{bmatrix}
\lambda_{1} \varphi_{1}(x_{1}) & \lambda_{2} \varphi_{2}(x_{1}) & \cdots & \lambda_{n} \varphi_{n}(x_{1}) \\
\lambda_{1} \varphi_{1}(x_{2}) & \lambda_{2} \varphi_{2}(x_{2}) & \cdots & \lambda_{n} \varphi_{n}(x_{2}) \\
\vdots & \vdots & \ddots & \vdots \\
\lambda_{1} \varphi_{1}(x_{n}) & \lambda_{2} \varphi_{2}(x_{n}) & \cdots & \lambda_{n} \varphi_{n}(x_{n}) 
\end{bmatrix}
\begin{bmatrix}
\varphi_{1}(x_{1}) & \varphi_{1}(x_{2}) & \cdots & \varphi_{1}(x_{n}) \\
\varphi_{2}(x_{1}) & \varphi_{2}(x_{2}) & \cdots & \varphi_{2}(x_{n}) \\
\vdots & \vdots & \ddots & \vdots \\
\varphi_{n}(x_{1}) & \varphi_{n}(x_{2}) & \cdots & \varphi_{n}(x_{n}) 
\end{bmatrix} \notag \\
& = 
\begin{bmatrix}
\varphi_{1}(x_{1}) & \varphi_{2}(x_{1}) & \cdots & \varphi_{n}(x_{1}) \\
\varphi_{1}(x_{2}) & \varphi_{2}(x_{2}) & \cdots & \varphi_{n}(x_{2}) \\
\vdots & \vdots & \ddots & \vdots \\
\varphi_{1}(x_{n}) & \varphi_{2}(x_{n}) & \cdots & \varphi_{n}(x_{n}) 
\end{bmatrix}
\begin{bmatrix}
\lambda_{1} & & & \\
 & \lambda_{2} & & \\
 & & \ddots & \\
 & & & \lambda_{n} 
\end{bmatrix}
\begin{bmatrix}
\varphi_{1}(x_{1}) & \varphi_{1}(x_{2}) & \cdots & \varphi_{1}(x_{n}) \\
\varphi_{2}(x_{1}) & \varphi_{2}(x_{2}) & \cdots & \varphi_{2}(x_{n}) \\
\vdots & \vdots & \ddots & \vdots \\
\varphi_{n}(x_{1}) & \varphi_{n}(x_{2}) & \cdots & \varphi_{n}(x_{n}) 
\end{bmatrix}. 
\end{align}
Then, letting
\begin{align}
\Phi_{n}(x_{1}, x_{2}, \ldots, x_{n})
= 
\begin{bmatrix}
\varphi_{1}(x_{1}) & \varphi_{1}(x_{2}) & \cdots & \varphi_{1}(x_{n}) \\
\varphi_{2}(x_{1}) & \varphi_{2}(x_{2}) & \cdots & \varphi_{2}(x_{n}) \\
\vdots & \vdots & \ddots & \vdots \\
\varphi_{n}(x_{1}) & \varphi_{n}(x_{2}) & \cdots & \varphi_{n}(x_{n}) 
\end{bmatrix}, 
\label{eq:approx_matrix_Phi}
\end{align}
we have
\begin{align}
\det \mathcal{K}
\approx	
(\lambda_{1} \lambda_{2} \cdots \lambda_{n}) \left( \det \Phi_{n}(x_{1}, x_{2}, \ldots, x_{n}) \right)^{2}. 
\label{eq:approx_denom_n}
\end{align}

\subsection{Approximate reduction to a convex optimization problem}
\label{sec:reduce_convex}

Based on the approximation in~\eqref{eq:approx_denom_n}, 
we consider the optimization problem 
\begin{align}
\text{maximize} 
\quad 
\left( \det \Phi_{n}(x_{1}, \ldots, x_{n}) \right)^{2}
\qquad 
\text{subject to}
\quad 
(x_{1}, \ldots , x_{n}) \in \Omega^{n}. 
\label{eq:max_det_Phi_n}
\end{align}
First, we consider the equivalent form of this problem as shown below. 
\begin{thm}
Problem \eqref{eq:max_det_Phi_n} is equivalent to the problem given by
\begin{align}
\begin{array}{ll}
\text{maximize} 
& \quad
\displaystyle
\det
\left(
\sum_{k = 1}^{n} 
\begin{bmatrix}
\varphi_{1}(x_{k}) \\
\vdots \\
\varphi_{n}(x_{k}) \\
\end{bmatrix}
\begin{bmatrix}
\varphi_{1}(x_{k}) & 
\cdots & 
\varphi_{n}(x_{k}) 
\end{bmatrix}
\right) \\
\text{subject to}
& \quad 
\LARGE \mathstrut
(x_{1}, \ldots , x_{n}) \in \Omega^{n}. 
\end{array}
\label{eq:equiv_det_opt}
\end{align}
\end{thm}

\begin{proof}
The conclusion is derived from the following relations:
\begin{align}
& \left( \det \Phi_{n}(x_{1}, \ldots, x_{n}) \right)^{2} \notag \\
& = \det \left( \Phi_{n}(x_{1}, \ldots, x_{n}) \, \Phi_{n}(x_{1}, \ldots, x_{n})^{T} \right) \notag \\
& = 
\det \left(
\begin{bmatrix}
\varphi_{1}(x_{1}) & \varphi_{1}(x_{2}) & \cdots & \varphi_{1}(x_{n}) \\
\varphi_{2}(x_{1}) & \varphi_{2}(x_{2}) & \cdots & \varphi_{2}(x_{n}) \\
\vdots & \vdots & \ddots & \vdots \\
\varphi_{n}(x_{1}) & \varphi_{n}(x_{2}) & \cdots & \varphi_{n}(x_{n}) 
\end{bmatrix}
\begin{bmatrix}
\varphi_{1}(x_{1}) & \varphi_{2}(x_{1}) & \cdots & \varphi_{n}(x_{1}) \\
\varphi_{1}(x_{2}) & \varphi_{2}(x_{2}) & \cdots & \varphi_{n}(x_{2}) \\
\vdots & \vdots & \ddots & \vdots \\
\varphi_{1}(x_{n}) & \varphi_{2}(x_{n}) & \cdots & \varphi_{n}(x_{n}) 
\end{bmatrix}
\right) \notag \\
& = 
\det \left( \left(
\sum_{k = 1}^{n} \varphi_{i}(x_{k}) \, \varphi_{j}(x_{k})
\right)_{ij} \right). 
\end{align}
\end{proof}

Next, we consider approximation of Problem~\eqref{eq:equiv_det_opt}
because it is not easily tractable. 
We approximate it by preparing a sufficient number of candidate points 
$y_{1}, \ldots, y_{m} \in \Omega \ (n \ll m)$ and choosing $n$ points $y_{i_{1}} ,\ldots y_{i_{n}}$ which maximize
\[
\det
\left(
\sum_{\ell = 1}^{n} 
\begin{bmatrix}
\varphi_{1}(y_{i_{\ell}}) \\
\vdots \\
\varphi_{n}(y_{i_{\ell}}) \\
\end{bmatrix}
\begin{bmatrix}
\varphi_{1}(y_{i_{\ell}}) & 
\cdots & 
\varphi_{n}(y_{i_{\ell}}) 
\end{bmatrix}
\right). 
\]
Then, we can rewrite this problem as follows:
\begin{align}
\begin{array}{ll}
\text{maximize} 
& \quad
\displaystyle
\det
\left(
\sum_{j = 1}^{m}
w_{j} 
\begin{bmatrix}
\varphi_{1}(y_{j}) \\
\vdots \\
\varphi_{n}(y_{j}) \\
\end{bmatrix}
\begin{bmatrix}
\varphi_{1}(y_{j}) & 
\cdots & 
\varphi_{n}(y_{j}) 
\end{bmatrix}
\right) \\
\text{subject to}
& \quad 
\LARGE \mathstrut
w_{j} \in \{0,1\}, \ w_{1} + \cdots + w_{m} = n.
\end{array}
\label{eq:equiv_appr_Dopt}
\end{align}
In fact, 
Problem~\eqref{eq:equiv_appr_Dopt} is well-known as a \textit{$D$-optimal experimental design problem}. 

Finally, we relax the constraint $w_{j} \in \{ 0, 1 \}$ of Problem~\eqref{eq:equiv_appr_Dopt} 
because it is still difficult. 
More precisely, we consider the relaxed problem given by
\begin{align}
\begin{array}{ll}
\text{maximize} 
& \quad
\displaystyle
\det
\left(
\sum_{j = 1}^{m}
w_{j} 
\begin{bmatrix}
\varphi_{1}(y_{j}) \\
\vdots \\
\varphi_{n}(y_{j}) \\
\end{bmatrix}
\begin{bmatrix}
\varphi_{1}(y_{j}) & 
\cdots & 
\varphi_{n}(y_{j}) 
\end{bmatrix}
\right) \\
\text{subject to}
& \quad 
\LARGE \mathstrut
0 \leq w_{j} \leq 1, \ w_{1} + \cdots + w_{m} = n.
\end{array}
\label{eq:equiv_appr_Dopt_relax}
\end{align}
In fact, the objective function of this problem is log-concave with respect to a vector $w = (w_{1}, \ldots, w_{m})^{T}$ 
(see e.g.~\cite{bib:BorweinLewis_ConvAnal_2006}). 
Therefore, in principle, we can obtain an optimal solution $w^{\ast} \in [0,1]^{m}$ of Problem~\eqref{eq:equiv_appr_Dopt_relax} 
by using a standard solver for convex optimization. 
Furthermore, as shown in Section \ref{sec:num_expr}, 
we can reduce the solution $w^{\ast}$ to a 0-1 vector, 
which becomes an approximate solution of Problem~\eqref{eq:equiv_appr_Dopt}. 
Taking these facts into account, 
in Section~\ref{sec:reform_by_SOCP}, 
we begin with reformulating Problem~\eqref{eq:equiv_appr_Dopt_relax} 
to solve it efficiently. 

\begin{rem}
In general, 
the matrix of Equation~\eqref{eq:approx_matrix_Phi} can be singular for certain choices of the set $\mathcal{X}_{n} = \{ x_{1}, \ldots, x_{n} \}$. 
Hence there can exist a set $\{ y_{1}, \ldots, y_{m} \}$ such that the optimal value of Problem~\eqref{eq:equiv_appr_Dopt} is zero. 
We need to avoid such sets because it is not desirable in view of our purpose. 
However, we leave this issue as a theme for future work. 
In the results of the numerical experiments in Section~\ref{sec:num_expr} below, 
such sets do not seem to appear. 
\end{rem}

	
\section{Reformulation by a second order cone programming (SOCP) problem}
\label{sec:reform_by_SOCP}

To solve Problem~\eqref{eq:equiv_appr_Dopt_relax} efficiently, 
we use a reformulation of Problem~\eqref{eq:equiv_appr_Dopt_relax} 
as a second order cone programming (SOCP) 
problem~\cite{bib:Sagnol_OptDesign_SOCP_2011,bib:Sagnol_etal_OptDesign_SOCP_2015}. 
In this section, 
we decribe this reformulation, which is proposed by Sagnol and Harman~\cite{bib:Sagnol_etal_OptDesign_SOCP_2015}. 
To this end, 
we consider the general form of the relaxed $D$-optimal design given by 
\begin{align}
\begin{array}{lll}
\text{maximize} 
& \quad
\displaystyle
\det
\left(
\sum_{j = 1}^{m}
w_{j} 
{a}_{j} {a}_{j}^{T}
\right) & \\
\text{subject to}
& \quad 0 \leq w_{j} \leq 1 & (j = 1,\ldots, m) \\ 
& \quad w_{1} + \cdots + w_{m} = n & 
\end{array}
\label{eq:gen_appr_Dopt_relax}
\end{align}
where ${a}_{j} \in \mathbf{R}^{\ell} \ (j = 1,\ldots, m)$ for $\ell \leq m$.  
If we set $\ell = n$ and ${a}_{j} = (\varphi_{1}(y_{j}), \ldots , \varphi_{n}(y_{j}))^{T}$, 
Problem~\eqref{eq:gen_appr_Dopt_relax} is reduced to Problem~\eqref{eq:equiv_appr_Dopt_relax}. 

We describe a reformulation of Problem~\eqref{eq:gen_appr_Dopt_relax} 
as a SOCP problem in the same manner as that in~\cite{bib:Sagnol_etal_OptDesign_SOCP_2015}. 
It consists of the following two steps. 

\begin{itemize}
\item[(1)]
Rewriting the determinant in Problem~\eqref{eq:gen_appr_Dopt_relax} 
as the optimal value of an optimization problem
for a fixed $w = (w_{1}, \ldots, w_{m})^{T}$. 

\item[(2)]
Adding the constraint of $w$ to the optimization problem and expressing it as a SOCP problem. 

\end{itemize}

\noindent
In the following, we show the details of (1) and (2) in Sections~\ref{sec:opt_det} and~\ref{sec:SOCP_form}, respectively. 
In fact, Problem~\eqref{eq:gen_appr_Dopt_relax} is a simplified case of the problem treated in~\cite{bib:Sagnol_etal_OptDesign_SOCP_2015}. 
Therefore, we just describe the results for the reformulation in these sections 
and show their proofs in Appendix~\ref{sec:proofs_SOC} for readers' convenience. 

\subsection{Optimization problems yielding a determinant}
\label{sec:opt_det}

Let the matrix $H \in \mathbf{R}^{\ell \times m}$ be given by 
\begin{align}
H = (\sqrt{w_{1}} a_{1}, \ldots, \sqrt{w_{m}} a_{m})
\label{eq:def_mat_H}
\end{align}
for $w = (w_{1}, \ldots, w_{m})^{T} \geq 0$. 
Then, the determinant in Problem~\eqref{eq:gen_appr_Dopt_relax} is written as $\det (HH^{T})$. 
This determinant is given by the optimal value of an optimization problem as shown by the following theorem. 

\begin{thm}[Special case of Theorem 4.2 in {\cite{bib:Sagnol_etal_OptDesign_SOCP_2015}}]
\label{thm:opt_expr_det}
Let $\ell$ and $m$ be integers with $\ell \leq m$ and let $H \in \mathbf{R}^{\ell \times m}$ 
be given by \eqref{eq:def_mat_H}. 
Furthermore, 
let $\mathrm{OPT}_{1}(H)$ be the optimal value of the following optimization problem:
\begin{align}
\begin{array}{cll}
\displaystyle
\mathop{\text{maximize}}_{Q \in \mathbf{R}^{m \times \ell}, \, S \in \mathbf{R}^{\ell \times \ell}} \quad  & \displaystyle (\det S)^{2} & \\
\text{subject to} \quad & S \text{ is lower triangular}, & \\
 & HQ = S, & \\
 & \| Q \, \mathbf{e}_{i} \| \leq 1 & (i = 1,\ldots, \ell).
\end{array}
\label{eq:opt_expr_det}
\end{align}
Then $\mathrm{OPT}_{1}(H) = \det(HH^{T})$. 
\end{thm}

\begin{thm}[Special case of Theorem 4.3 in {\cite{bib:Sagnol_etal_OptDesign_SOCP_2015}}]
\label{thm:SOCP_expr_det}
Let $\ell$ and $m$ be integers with $\ell \leq m$, 
let $a_{1}, \ldots, a_{m} \in \mathbf{R}^{\ell}$, and 
let $w = (w_{1}, \ldots, w_{m})^{T} \geq 0$. 
Furthermore, 
let $\mathrm{OPT}_{2}(\{ a_{i} \}, \, w)$ be the optimal value of the following optimization problem:
\begin{align}
\begin{array}{cll}
\displaystyle
\mathop{\text{maximize}}_{
\begin{subarray}{l}
Z = (z_{ij}), \, T = (t_{ij}) \in \mathbf{R}^{m \times \ell}, \\ 
G = (g_{ij}) \in \mathbf{R}^{\ell \times \ell}
\end{subarray}} \quad  & 
\displaystyle \prod_{j=1}^{\ell} g_{jj} & \\
\text{subject to} \quad & G \text{ is lower triangular}, & \\
 & (a_{1}, \ldots, a_{m}) \, Z = G, & \\
 & z_{ij}^{2} \leq t_{ij} w_{i} & (i = 1,\ldots, m, \, j=1,\ldots, \ell), \\
 & \displaystyle \sum_{i=1}^{m} t_{ij} \leq g_{jj} &  (j=1,\ldots, \ell). \\
\end{array}
\label{eq:opt_expr_SOCP}
\end{align}
Then $\mathrm{OPT}_{2}(\{ a_{i} \}, \, w) = \det \left( HH^{T} \right)$, where $H$ is given by \eqref{eq:def_mat_H}.
\end{thm}

\subsection{SOCP form of Problem~\eqref{eq:gen_appr_Dopt_relax}}
\label{sec:SOCP_form}

We show that Problem~\eqref{eq:gen_appr_Dopt_relax} can be expressed 
as a SOCP problem. 
To this end, 
we start with showing that the product of the variables 
in the objective function of Problem~\eqref{eq:opt_expr_SOCP} 
is represented as the optimal value of an optimization problem with first or second order constraints.

\begin{thm}[{\cite{bib:AlizadehGoldfarb_SOCP_2003,bib:Lobo_etal_SOCP_1998}}]
\label{thm:prod_SOCP}
Let $\ell \geq 2$ be an integer and let $h_{1}, \ldots , h_{\ell} \in \mathbf{R}$ be non-negative numbers. 
Furthermore, 
let $p$ be the integer with $2^{p-1} < \ell \leq 2^{p}$. 
We consider the optimization problem given by
\begin{align}
\begin{array}{cll}
\displaystyle
\mathop{\mathrm{maximize}} \quad  & u_{1} & \\
\text{subject to} \quad & u_{i}^{2} \leq u_{2i} u_{2i+1} & (i=1, \ldots, 2^{p-1}-1), \\
 & u_{i}^{2} \leq g_{2i - 2^{p} + 1} g_{2i - 2^{p} + 2} & (i=2^{p-1}, \ldots, 2^{p-1} + \ell/2 - 1), \\
 & u_{i}^{2} \leq u_{1}^{2} & (i=2^{p-1} + \ell/2, \ldots, 2^{p}-1), \\
 & u_{i} \geq 0 & (i = 1,\ldots, 2^{p}-1)
\end{array}
\label{eq:product_SOCP_even}
\end{align}
in the case that $\ell$ is even, and 
\begin{align}
\begin{array}{cll}
\displaystyle
\mathop{\mathrm{maximize}} \quad  & u_{1} & \\
\text{subject to} \quad & u_{i}^{2} \leq u_{2i} u_{2i+1} & (i=1, \ldots, 2^{p-1}-1), \\
 & u_{i}^{2} \leq g_{2i - 2^{p} + 1} g_{2i - 2^{p} + 2} & (i=2^{p-1}, \ldots, 2^{p-1} + (\ell-3)/2), \\
 & u_{i}^{2} \leq g_{2i - 2^{p} + 1} u_{1} & (i=2^{p-1} + (\ell-1)/2), \\ 
 & u_{i}^{2} \leq u_{1}^{2} & (i=2^{p-1} + (\ell+1)/2, \ldots, 2^{p}-1), \\
 & u_{i} \geq 0 & (i = 1,\ldots, 2^{p}-1) 
\end{array}
\label{eq:product_SOCP_odd}
\end{align}
in the case that $\ell$ is odd.
Then, its optimal value is equal to the geometric mean $\left( \prod_{j=1}^{\ell} g_{j} \right)^{1/\ell}$. 
\end{thm}

By Theorems~\ref{thm:SOCP_expr_det} and~\ref{thm:prod_SOCP}, 
we can show that the value $(\det (HH^{T}))^{1/\ell}$ is equal to the optimal value of Problem~\eqref{eq:opt_expr_SOCP} 
whose objective function is rewritten in the form of 
Problem~\eqref{eq:product_SOCP_even} or~\eqref{eq:product_SOCP_odd}. 
Consequently, 
regarding $w_{1}, \ldots, w_{m}$ as variables, 
we obtain the SOCP form of Problem~\eqref{eq:gen_appr_Dopt_relax}. 
For simplicity, we show only the case that $\ell$ is even as follows: 
\begin{align}
\begin{array}{cll}
\displaystyle
\mathop{\text{maximize}}_{
\begin{subarray}{l}
Z = (z_{ij}) \in \mathbf{R}^{m \times \ell}, \\ 
T = (t_{ij}) \in \mathbf{R}^{m \times \ell}, \\ 
\{ \tilde{g}_{1}, \ldots, \tilde{g}_{\ell} \} \subset \mathbf{R}, \\ 
\{u_{1}, \ldots, u_{2^{p}-1} \} \subset \mathbf{R}, \\
\{w_{1}, \ldots, w_{m} \} \subset \mathbf{R} 
\end{subarray}} & 
u_{1} & \\
\text{subject to} & (a_{1}, \ldots, a_{m}) \, Z = 
\begin{bmatrix}
\tilde{g}_{1} & 0 & \cdots & 0 \\
\ast & \tilde{g}_{2} & \ddots & \vdots \\
\ast & \ast & \ddots & 0 \\
\ast & \ast & \ast & \tilde{g}_{\ell} \\ 
\end{bmatrix}, & \\
 & z_{ij}^{2} \leq t_{ij} w_{i} & (i = 1,\ldots, m, \, j=1,\ldots, \ell), \\
 & \displaystyle \sum_{i=1}^{m} t_{ij} \leq \tilde{g}_{j} &  (j=1,\ldots, \ell). \\
 & u_{i}^{2} \leq u_{2i} u_{2i+1} & (i=1, \ldots, 2^{p-1}-1), \\
 & u_{i}^{2} \leq \tilde{g}_{2i - 2^{p} + 1} \tilde{g}_{2i - 2^{p} + 2} & (i=2^{p-1}, \ldots, 2^{p-1} + \ell/2 - 1), \\
 & u_{i}^{2} \leq u_{1}^{2} & (i=2^{p-1} + \ell/2, \ldots, 2^{p}-1), \\
 & u_{i} \geq 0 & (i = 1,\ldots, 2^{p}-1) \\
 & 0 \leq w_{j} \leq 1 & (j = 1,\ldots, m), \\ 
 & w_{1} + \cdots + w_{m} = n, & 
\end{array}
\label{eq:opt_expr_SOCP_pre_final}
\end{align}
whose optimal value is the $\ell$-th root of that of Problem~\eqref{eq:gen_appr_Dopt_relax}.
We can similarly deal with the case that $\ell$ is odd. 
Therefore, 
we can obtain the optimal solution 
$w^{\ast} = (w_{1}^{\ast}, \ldots, w_{m}^{\ast})^{T}$ of Problem~\eqref{eq:gen_appr_Dopt_relax} 
by solving Problem~\eqref{eq:opt_expr_SOCP_pre_final} 
with a SOCP solver.

\section{Algorithms and numerical experiments}
\label{sec:num_expr}

In this section, 
we propose algorithms for generating point sets and apply them to 
Examples~\ref{ex:kernel_Brownian}, \ref{ex:kernel_Spherical}, and \ref{ex:kernel_Gaussian} in Section~\ref{sec:prelim}. 

\subsection{Algorithms for generating point sets}

We propose two algorithms: Algorithms~\ref{alg:gen_points} and~\ref{alg:seq_gen_points}. 
In Algorithm~\ref{alg:gen_points}, 
we solve SOCP problem~\eqref{eq:opt_expr_SOCP_pre_final} with 
$a_{j} = (\varphi_{1}(y_{j}), \ldots, \varphi_{n}(y_{j}))^{T}$ and $\ell = n$
to obtain the optimal solution $w^{\ast} = (w_{1}^{\ast}, \ldots, w_{m}^{\ast})^{T}$. 
Then, we choose the ``local maxima'' of $w^{\ast}$, 
which mean the components $w_{j}^{\ast}$ satisfying $w_{j}^{\ast} \geq w_{k}^{\ast}$ 
for any ``neighbors'' $k \in \mathcal{N}_{j} := \{ k \mid y_{k} \text{ is a neighbor of } y_{j} \text{ in } \Omega \}$. 
We determine the neighbors of $y_{j}$
according to the geometric property of $\Omega$ and the arrangement of $y_{1}, \ldots , y_{m} \in \Omega$ 
as shown in Section~\ref{sec:num_methods}. 

We propose Algorithm~\ref{alg:seq_gen_points} to show that 
we can invent a sequential algorithm based on the SOCP formulation. 
To generate $n'$ points, 
the algorithm requires a sequence of positive integers $\{ n_{i} \}_{i=1}^{I}$ with $n_{1} + \cdots + n_{I} = n'$. 
Then, the algorithm consists of $I$ steps and new sampling points are added in each step. 
The integer $n_{i}$ indicates the number of the new sampling points added in the $i$-th step. 
To achieve this procedure, 
in the $i$-th step
we solve SOCP problem~\eqref{eq:opt_expr_SOCP_pre_final} with 
$n = n_{1} + \cdots + n_{i}$ and the weights $w_{j}$ fixed to $1$ for $j \in \mathcal{W}$, 
where $\mathcal{W}$ is the set of the indices corresponding to the points $y_{j}$ chosen in the previous steps. 
We can set the SOCP problems easily because we have only to add the linear constraints $w_{j} = 1 \ (j \in \mathcal{W})$. 
Algorithm~\ref{alg:seq_gen_points} can be regarded as a generalization of greedy algorithms that choose points one by one. 

\renewcommand{\algorithmicrequire}{\textbf{Input:}}
\renewcommand{\algorithmicensure}{\textbf{Output:}}
\begin{algorithm}[t]
\caption{Generation of sampling points}         
\label{alg:gen_points}                          
\begin{algorithmic}
\REQUIRE 
Region $\Omega \subset \mathbf{R}^{d}$, \
Orthonormal system $\{ \varphi_{1}, \ldots, \varphi_{n} \}$, \
Points $y_{1}, \ldots, y_{m} \in \Omega$, \ 
Neighborhood sets $\{ \mathcal{N}_{j} \}_{j=1}^{m}$, \ 
Integer $n$.
\ENSURE 
Points $x_{1}, \ldots, x_{n} \in \Omega$.
\FOR{$j = 1$ to $m$}
\STATE $a_{j} \leftarrow (\varphi_{1}(y_{j}), \ldots, \varphi_{n}(y_{j}))^{T}$
\ENDFOR
\STATE Solve SOCP problem~\eqref{eq:opt_expr_SOCP_pre_final} with $\ell = n$ to obtain the optimal solution $w^{\ast} = (w_{1}^{\ast}, \ldots, w_{m}^{\ast})^{T}$
\STATE $\mathcal{J} = \emptyset$
\FOR{$j = 1$ to $m$}
	\IF{$w_{j}^{\ast} \geq w_{k}^{\ast}$ for any $k \in \mathcal{N}_{j}$}
	\STATE $\mathcal{J} \leftarrow \mathcal{J} \cup \{ j \}$
	\ENDIF
\ENDFOR
\STATE Sort $\{ w_{j}^{\ast} \}_{j \in \mathcal{J}}$ in descending order: $w_{j_{1}}^{\ast} \geq w_{j_{2}}^{\ast} \geq \cdots \geq w_{j_{|\mathcal{J}|}}^{\ast}$
\IF{$|\mathcal{J}| \geq n$}
	\STATE $x_{k} \leftarrow y_{j_{k}} \ (k=1,\ldots, n)$
\ELSE
	\PRINT{``Error''}
\ENDIF
\end{algorithmic}
\end{algorithm}

\begin{algorithm}[t]
\caption{Sequential generation of sampling points}         
\label{alg:seq_gen_points}
\begin{algorithmic}
\REQUIRE 
Region $\Omega \subset \mathbf{R}^{d}$, \
Orthonormal system $\{ \varphi_{1}, \ldots, \varphi_{n'} \}$, \
Points $y_{1}, \ldots, y_{m} \in \Omega$, \ 
Neighborhood sets $\{ \mathcal{N}_{j} \}_{j=1}^{m}$, \ 
Sequence of positive integers $\{ n_{i} \}_{i=1}^{I}$ with $n_{1} + \cdots + n_{I} = n'$.
\ENSURE 
Points $x_{1}, \ldots, x_{n'} \in \Omega$.
\STATE $\mathcal{W} = \emptyset$
\FOR{$i = 1$ to $I$}
\STATE $n = n_{1} + \cdots + n_{i}$
\FOR{$j = 1$ to $m$}
\STATE $a_{j} \leftarrow (\varphi_{1}(y_{j}), \ldots, \varphi_{n}(y_{j}))^{T}$
\ENDFOR
\STATE Solve SOCP problem~\eqref{eq:opt_expr_SOCP_pre_final} with $\ell = n$ and the additional constraints $w_{j} = 1 \ (j \in \mathcal{W})$
to obtain the optimal solution $w^{\ast} = (w_{1}^{\ast}, \ldots, w_{m}^{\ast})^{T}$
\STATE $\mathcal{J} = \emptyset$
\FOR{$j = 1$ to $m$}
	\IF{$w_{j}^{\ast} \geq w_{k}^{\ast}$ for any $k \in \mathcal{N}_{j}$}
	\STATE $\mathcal{J} \leftarrow \mathcal{J} \cup \{ j \}$
	\ENDIF
\ENDFOR
\STATE Sort $\{ w_{j}^{\ast} \}_{j \in \mathcal{J}}$ in descending order: $w_{j_{1}}^{\ast} \geq w_{j_{2}}^{\ast} \geq \cdots \geq w_{j_{|\mathcal{J}|}}^{\ast}$
\IF{$|\mathcal{J}| \geq n$}
	\STATE $x_{k} \leftarrow y_{j_{k}} \ (k=1,\ldots, n)$
	\STATE $\mathcal{W} \leftarrow \mathcal{W} \cup \{ j_{1}, \ldots, j_{n} \}$
\ELSE
	\PRINT{``Error''}
\ENDIF
\ENDFOR
\end{algorithmic}
\end{algorithm}

\subsection{Methods of numerical experiments for Algorithms~\ref{alg:gen_points} and~\ref{alg:seq_gen_points}}
\label{sec:num_methods}

For numerical experiments, 
we take Examples~\ref{ex:kernel_Brownian}, \ref{ex:kernel_Spherical}, and~\ref{ex:kernel_Gaussian} 
in Section~\ref{sec:prelim}. 
For Example~\ref{ex:kernel_Spherical}, we set $\gamma = 0.1$. 
For Example~\ref{ex:kernel_Gaussian}, we set $\alpha = \varepsilon = 1$ and consider the four versions: 
\begin{description}
\setlength{\parskip}{0pt}
\setlength{\parskip}{0pt}
\item[3-1.] $d = 1$, $\Omega = [-1,1]$, 
\item[3-2.] $d = 2$, $\Omega = [-1,1]^{2}$, 
\item[3-3.] $d = 2$, $\Omega = \triangle := \{ (x_{1}, x_{2}) \in [-1,1]^{2} \mid x_{1} + x_{2} \geq 0 \}$,
\item[3-4.] $d = 2$, $\Omega = D := \{ (x_{1}, x_{2}) \in [-1,1]^{2} \mid x_{1}^{2}+x_{2}^{2} \leq 1 \}$.
\end{description}

We chose the candidate points $y_{1}, \ldots, y_{m} \in \Omega$ and the neighborhood sets $\mathcal{N}_{j}$ for these examples. 
For Example~\ref{ex:kernel_Brownian}, 
we took $m=250$, $y_{j} = (j-1)/(m-1) \ (j=1,\ldots, m)$ on $[0,1]$, and $\mathcal{N}_{j} = \{ j-1,  j+1 \} \cap \{1, \ldots, m \}$. 
For Example~\ref{ex:kernel_Spherical}, 
we took $k=25$, $m = (k-1)(k-2)+2 = 554$, and
and the angles 
$\theta_{p} = \pi (p-1)/(k-1) \ (p=1,\ldots, k)$ and 
$\phi_{q} = 2\pi (q-1)/(k-1) \ (q=1,\ldots, k-1)$ to generate the points
\begin{align}
y_{j} = (\sin \theta_{p_{j}} \cos \phi_{q_{j}}, \, \sin \theta_{p_{j}} \sin \phi_{q_{j}}, \,  \cos \theta_{p_{j}}) \in S^{2}, 
\label{eq:sphere_y_j}
\end{align}
where
\begin{align}
(p_{j}, q_{j}) =
\begin{cases}
(1,1) & (j=1), \\
(k,1) & (j=m), \\
\left( \left \lfloor (j-2)/(k-1) \right \rfloor + 2, \, \mathop{\mathrm{mod}}(j-2, k-1)+1 \right) & (j \neq 1,m).
\end{cases}
\notag
\end{align}
Here $\mathop{\mathrm{mod}}(a,b)$ denotes the remainder after division of $a$ by $b$. 
Furthermore, we defined $\mathcal{N}_{j}$ by 
\begin{align}
\mathcal{N}_{j} = 
\begin{cases}
\{ 2, \ldots, k \} & (j=1), \\
\{ m-(k-1), \ldots, m-1 \} & (j=m), \\
\{ j + l_{j} , j + r_{j}, j-(k-1), j + (k+1) \} & (j \neq 1, m), 
\end{cases}
\label{eq:sphere_N_j}
\end{align}
where 
\begin{align}
& l_{j} = 
\begin{cases}
-1 & (\mathop{\mathrm{mod}}(j-2, k-1) \neq 0), \\
+(k-2) & (\mathop{\mathrm{mod}}(j-2, k-1) = 0),
\end{cases} \notag \\
& r_{j} = 
\begin{cases}
+1 & (\mathop{\mathrm{mod}}(j-2, k-1) \neq k-2), \\
-(k-2) & (\mathop{\mathrm{mod}}(j-2, k-1) = k-2).
\end{cases}
\notag
\end{align}
For Example~\ref{ex:kernel_Gaussian}-1, 
we took $m=250$, $y_{j} = -1 + 2(j-1)/(m-1) \ (j=1,\ldots, m)$ on $[-1,1]$, and $\mathcal{N}_{j} = \{ j-1,  j+1 \} \cap \{1, \ldots, m \}$. 
For  Examples~\ref{ex:kernel_Gaussian}-2, \ref{ex:kernel_Gaussian}-3, \ref{ex:kernel_Gaussian}-4,  
first we took the set $\tilde{Y}_{k} := \{ (-1 + 2(p-1)/(k-1), -1 + 2(q-1)/(k-1) \mid p = 1,\ldots, k, \, q = 1,\ldots, k \} \subset [-1,1]^{2}$ 
for some $k$ and then let the intersection $\tilde{Y}_{k} \cap \Omega$ be the set $\{ y_{j} \}_{j=1}^{m}$. 
We chose $k=23, 32, 27$ to obtain $m = 529, 520, 529$ for 
Examples~\ref{ex:kernel_Gaussian}-2, \ref{ex:kernel_Gaussian}-3, \ref{ex:kernel_Gaussian}-4,
respectively. 

For the cases with $d \geq 2$ in Examples~\ref{ex:kernel_Spherical} and~\ref{ex:kernel_Gaussian}, 
we need to decide the order of the eigenfunctions of the kernels 
because there are several eigenfunctions corresponding to an eigenvalue. 
We chose the orders as follows:
\begin{align}
& \text{Example~\ref{ex:kernel_Spherical}} \ (d = 3): 
\notag \\
& Y_{0,1}, \ Y_{1,1}, Y_{1,2}, Y_{1,3}, \ \ldots, \ Y_{n,1}, \ldots, Y_{n,2n+1}, \ \ldots 
\label{eq:SpHarm_basis} \\
& \text{Example~\ref{ex:kernel_Gaussian}-2, \ref{ex:kernel_Gaussian}-3, \ref{ex:kernel_Gaussian}-4} \ (d = 2):
\notag \\
& \varphi_{1}\varphi_{1}, \ \varphi_{2}\varphi_{1}, \varphi_{1}\varphi_{2}, \ \ldots, \ 
\varphi_{n-1}\varphi_{1}, \ldots, \varphi_{1}\varphi_{n-1}, \ \ldots ,
\label{eq:Gauss_basis}
\end{align}
where 
$Y_{n, \ell}$ are the spherical harmonics that appear in~\eqref{eq:sphere_Mercer}, 
$\varphi_{i}$ are given in \eqref{eq:Gauss_eigen_val_func}, and 
$\varphi_{i}\varphi_{j} = \varphi_{i}(x_{1})\varphi_{j}(x_{2})$ for $(x_{1}, x_{2}) \in \mathbf{R}^{2}$. 
Under these settings, 
we applied Algorithms~\ref{alg:gen_points} and~\ref{alg:seq_gen_points} to these examples to generate sampling points. 
In applying Algorithm~\ref{alg:seq_gen_points}, 
we chose 
$\{ 4, 8, 12, 16, 20, 24 \}$ and 
$\{ 8, 12, 16, 20, 24, 28, 32, 35 \}$
as the sequence $\{ n_1, \ldots, n_{I} \}$ 
for the one-dimensional examples 
(Examples~\ref{ex:kernel_Brownian} and~\ref{ex:kernel_Gaussian}-1) and the others, respectively. 
Then, we computed the maximum values of the power functions 
and the condition numbers of the kernel matrix $\mathcal{K} = (K(x_{i}, x_{j}))_{ij}$. 

To implement Algorithms~\ref{alg:gen_points} and~\ref{alg:seq_gen_points}, 
we used MATLAB R2018b and the MOSEK optimization toolbox for MATLAB 8.1.0.56
provided by MOSEK ApS  in Denmark (\url{https://www.mosek.com/}, last accessed on 19 October 2018). 
The toolbox contains a solver for SOCP problems. 
All computation in this section was done with the double precision floating point numbers
on a computer with Intel Xeon $2.1$ GHz CPU and $31.9$ GB RAM. 
The programs used for the computation are available on the web page \cite{bib:Tanaka_SOCP_program_2018}. 

\subsection{Results}
\label{sec:num_results}

We show the results by Figures~\ref{fig:weights_powfuncs_Brown1D}--\ref{fig:sphere_time}. 
Figures~\ref{fig:weights_powfuncs_Brown1D}--\ref{fig:weights_Gauss2D}
display the numerical solutions 
for the weights $w^{\ast}$ and power functions for some representative cases. 
Figures~\ref{fig:sphere_n35}--\ref{fig:Gauss2D_n35}
show the generated points for the higher-dimensional kernels in the case $n=35$. 
One common feature of the computed weights is that 
most components of them are nearly zero except for those corresponding to the ``local maxima'' 
satisfying $w_{j}^{\ast} \geq w_{k}^{\ast}$ for any $j \in \mathcal{N}_{j}$. 
Therefore it is expected that Problem~\eqref{eq:equiv_appr_Dopt_relax} given by the convex relaxation 
can find good approximate solution of Problem~\eqref{eq:equiv_appr_Dopt}, 
although this phenomenon has not been proved theoretically and is not observed in some cases. 
Actually, we observed that it did not occur and Algorithm~\ref{alg:gen_points} failed 
in the case $n \geq 18$ for Example~\ref{ex:kernel_Gaussian}-1. 

Figures~\ref{fig:Brown1D_pf}--\ref{fig:Gauss2D_dsk_pf}
show the maximum values of the power functions 
given by Algorithms~\ref{alg:gen_points}, \ref{alg:seq_gen_points}, and 
the $P$-greedy algorithm. 
In addition, 
Figures~\ref{fig:Brown1D_cnd}--\ref{fig:Gauss2D_dsk_cnd} show the condition numbers of the kernel matrix $\mathcal{K} = (K(x_{i}, x_{j}))_{ij}$. 
For the one-dimensional examples, 
Examples~\ref{ex:kernel_Brownian} (Figures~\ref{fig:Brown1D_pf} and~\ref{fig:Brown1D_cnd})
and~\ref{ex:kernel_Gaussian}-1 (Figures~\ref{fig:Gauss1D_pf} and~\ref{fig:Gauss1D_cnd}), 
we can observe that Algorithm~\ref{alg:gen_points} outperforms the $P$-greedy algorithm for most of $n$, 
although the decay rates given by them seem to be similar. 
The performance of Algorithm~\ref{alg:seq_gen_points} is a bit worse. 
For the higher-dimensional examples, 
in most cases the $P$-greedy algorithm outperforms the others, 
although Algorithm~\ref{alg:seq_gen_points} compete with it. 
Algorithm~\ref{alg:gen_points} does not show stable performance, 
but it compete with the others for some $n$. 
We guess that Problem~\eqref{eq:equiv_appr_Dopt_relax} can be precisely solved in the one-dimensional case, 
whereas it becomes more difficult in the higher-dimensional cases. 
In such cases Algorithm~\ref{alg:seq_gen_points} performs better than Algorithm~\ref{alg:gen_points}. 

\begin{rem}
As $n$ gets large, 
it tends to be difficult to obtain appropriate points 
because the weights in $w^{\ast}$ 
often fails to have an apparent pattern for which $n$ ``local maxima'' can be found easily. 
For example, it has less than $n$ ``local maxima'' in some cases%
\footnote{
For example, 
we observed that 
Algorithm~\ref{alg:gen_points} failed
for $n=100$ in the case of Example~\ref{ex:kernel_Spherical}
because of this phenomenon.}. 
Even in the cases that we can find $n$ ``local maxima'', 
the chosen points tend to be ill-balanced. 
The saturation of the values of the power functions and condition numbers shown respectively by Figures~\ref{fig:Gauss1D_pf} and~\ref{fig:Gauss1D_cnd} 
is owing to this phenomenon. 
Resolution of this difficulty is one of the themes for future work. 
\end{rem}

Figures~\ref{fig:Brown_time} and~\ref{fig:sphere_time}
show the computation times of Algorithm~\ref{alg:gen_points} and 
the $P$-greedy algorithm applied to Examples~\ref{ex:kernel_Brownian} and~\ref{ex:kernel_Spherical}.
We present only these computation times
because those of Example~\ref{ex:kernel_Gaussian}-1 are similar to those of Example~\ref{ex:kernel_Brownian}, and those of 
Examples~\ref{ex:kernel_Gaussian}-2, \ref{ex:kernel_Gaussian}-3, and~\ref{ex:kernel_Gaussian}-4
are similar to those of Example~\ref{ex:kernel_Spherical}.
Recall that we used the same number $m=250$ (the number of the candidate points) 
for Examples~\ref{ex:kernel_Brownian} and~\ref{ex:kernel_Gaussian}-1, 
and used the similar numbers as $m$ for Examples~\ref{ex:kernel_Spherical}, \ref{ex:kernel_Gaussian}-2, \ref{ex:kernel_Gaussian}-3, and~\ref{ex:kernel_Gaussian}-4. 
In each figure, 
the left and right graphs show 
the times for making the SOCP instances of Algorithm~\ref{alg:gen_points} and
those for executing the SOCP optimizer and the $P$-greedy algorithm, respectively. 
Making an instance means constructing the structure of the MOSEK toolbox 
expressing the constraints of Problem~\eqref{eq:opt_expr_SOCP_pre_final} by a matrix and vectors. 
Its details are described in Appendix~\ref{sec:rem_impl_SOCP}. 
For the $P$-greedy algorithm, 
we plot the total time $t_{1} + \cdots + t_{n}$ for each $n$, 
where $t_{i}$ is the time for generating $i$-th point by the $P$-greedy algorithm. 
In addition, 
we omit the times for Algorithm~\ref{alg:seq_gen_points} 
because there were little differences between Algorithms~\ref{alg:gen_points} and~\ref{alg:seq_gen_points} for common numbers $n$. 
From these figures, 
we can observe that it took longer to make the SOCP instances than to solve them and 
the times for solving the instances are a bit longer than those for executing the $P$-greedy algorithm. 
Therefore 
we think that the efficiency of solving the SOCP problems is sufficient 
and expect that more efficient implementation of generating the instances
would make our algorithms competitive with the $P$-greedy algorithm.  

Finally, we show the results of investigation whether the points generated by the proposed algorithms are influenced by
\begin{enumerate}
\setlength{\parskip}{0pt}
\setlength{\parskip}{0pt}

\item
the parameter(s) and

\item
the order

\end{enumerate}
of the eigenfunctions in the Mercer expansion. 
For example, 
the eigenfunctions for the Gaussian kernel in Example~\ref{ex:kernel_Gaussian}
contain the parameter $\alpha$ that can take an arbitrary positive value.
It depends on the density function $\rho$ of the space $L^{2}(\Omega, \rho)$ 
to which the eigenfunctions belong. 
In addition, 
we used the order of the eigenfunctions given by \eqref{eq:SpHarm_basis} and \eqref{eq:Gauss_basis}
for Example~\ref{ex:kernel_Spherical} and 
Examples~\ref{ex:kernel_Gaussian}-2, \ref{ex:kernel_Gaussian}-3, \ref{ex:kernel_Gaussian}-4, respectively. 
We investigate whether these factors influence the generated points by numerical experiments. 
To this end, 
we take Example~\ref{ex:kernel_Gaussian}-2 with 
\begin{enumerate}
\setlength{\parskip}{0pt}
\setlength{\parskip}{0pt}

\item
$\alpha$ replaced by $2$ or

\item
the order of the eigenfunctions changed as
\begin{align}
& \varphi_{1}\varphi_{1}, \ \varphi_{1}\varphi_{2}, \varphi_{2}\varphi_{1}, \ \ldots, \ 
\varphi_{1}\varphi_{n-1}, \ldots, \varphi_{n-1}\varphi_{1}, \ \ldots .
\label{eq:Gauss_basis_changed}
\end{align}

\end{enumerate}
Then, we show the results of generating points by Algorithm~\ref{alg:gen_points} 
for $n=35$ by Figure~\ref{fig:Gauss2D_n35_another}.
We can observe that both factors influence the generated points. 

\begin{figure}[t]
\centering

\begin{minipage}[t]{0.48\linewidth}
\includegraphics[width = \linewidth]{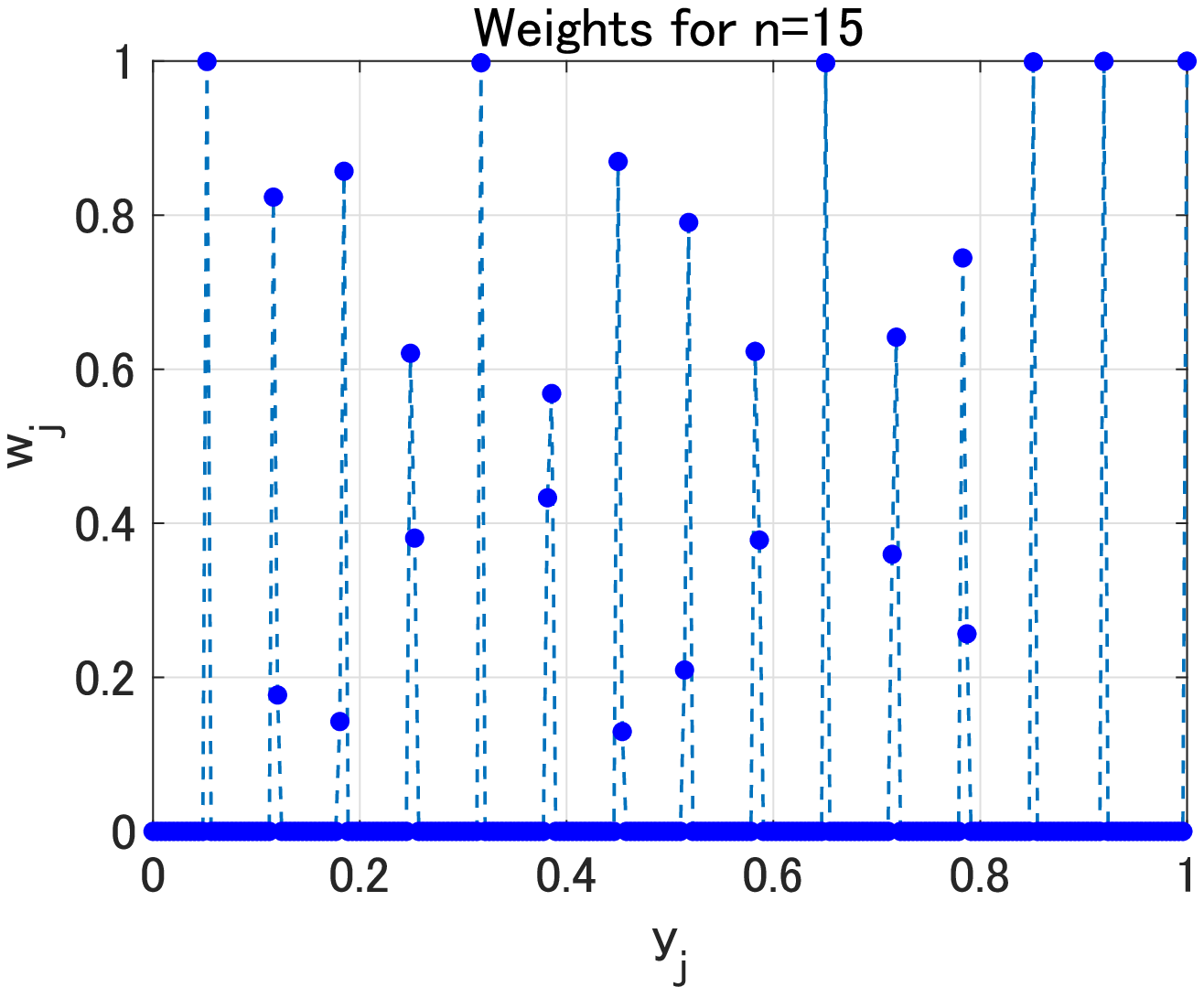}
\end{minipage}
\begin{minipage}[t]{0.48\linewidth}
\includegraphics[width = \linewidth]{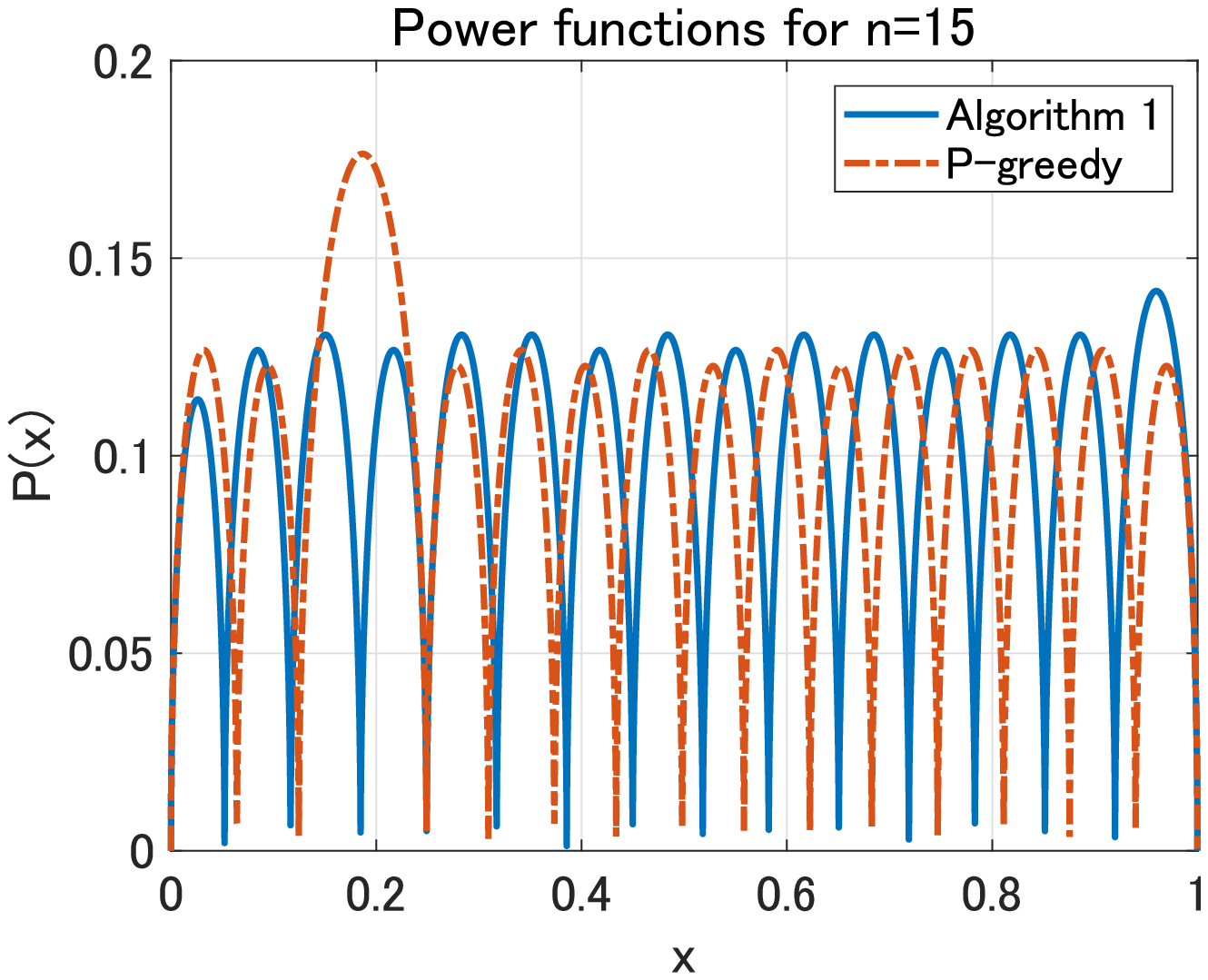}
\end{minipage}

\caption{Weights and power functions for the Brownian kernel on $[0,1]$ (Example~\ref{ex:kernel_Brownian}).
Left: the weights computed by Algorithm~\ref{alg:gen_points} for $n=15$, 
Right: the power functions given by Algorithm~\ref{alg:gen_points} and the $P$-greedy algorithm for $n = 15$. 
Their zeros are the points given by these algorithms. }
\label{fig:weights_powfuncs_Brown1D}

\end{figure}

\begin{figure}[t]
\centering

\begin{minipage}[t]{0.48\linewidth}
\includegraphics[width = \linewidth]{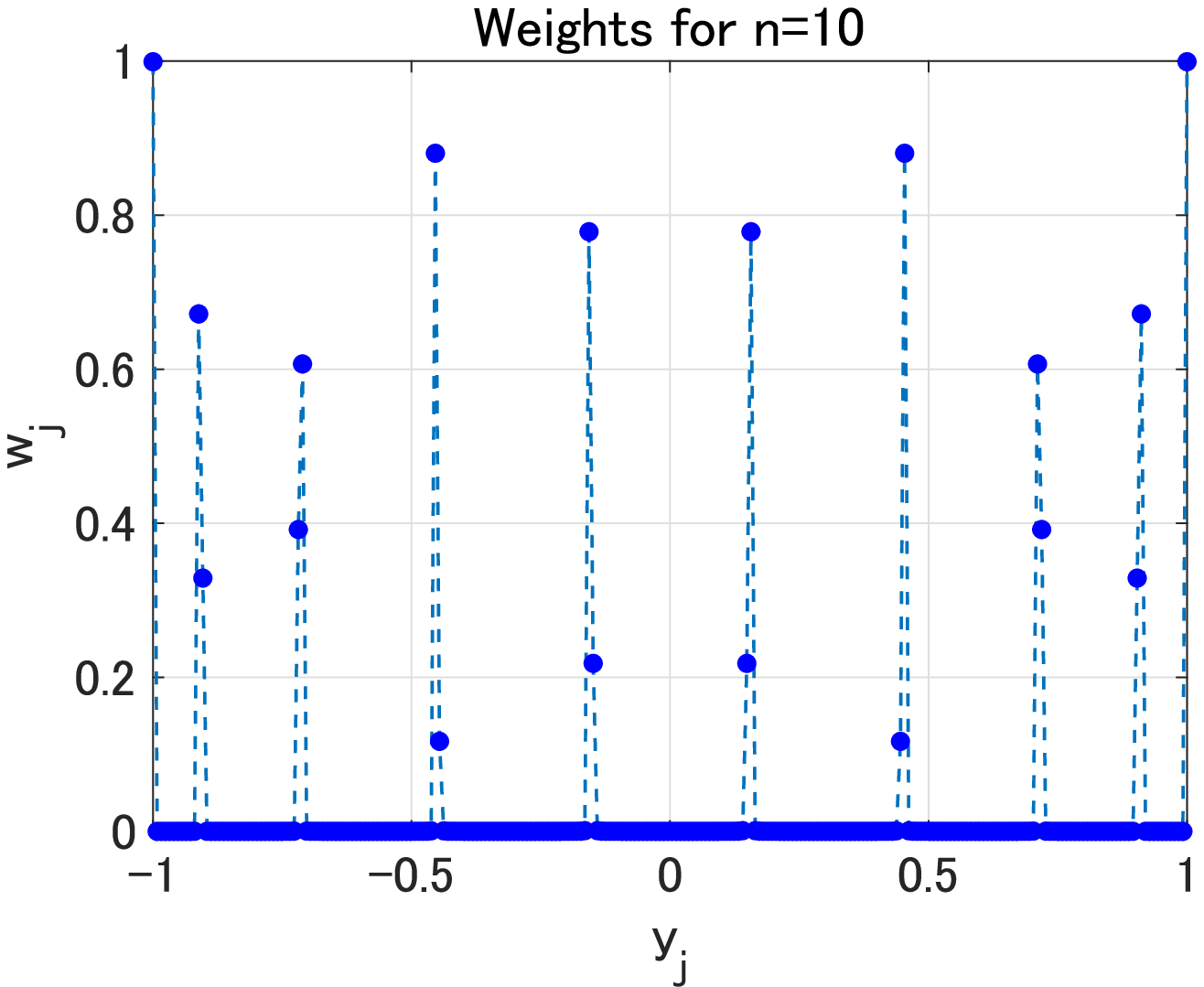}
\end{minipage}
\begin{minipage}[t]{0.48\linewidth}
\includegraphics[width = \linewidth]{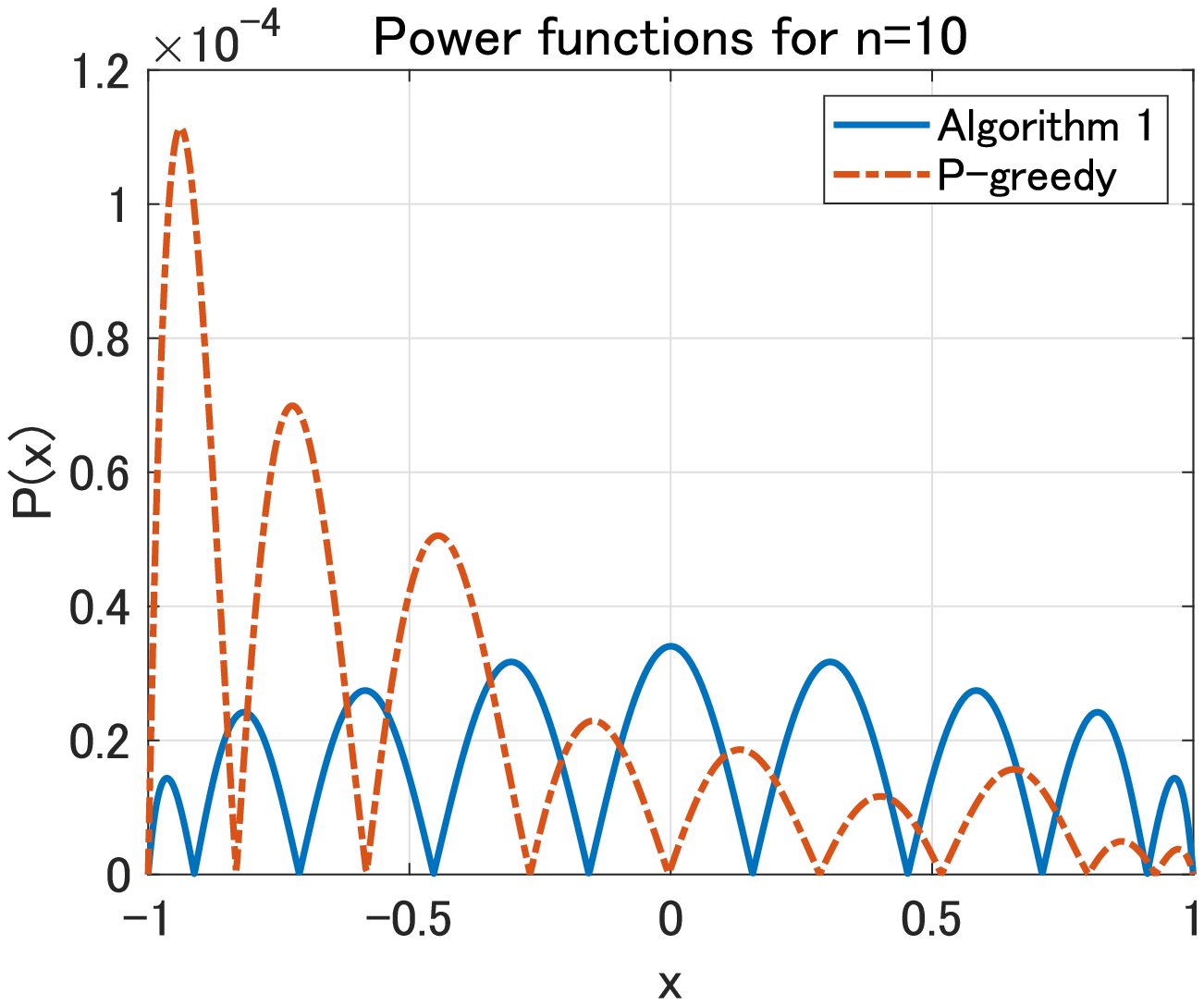}
\end{minipage}

\caption{Weights and power functions for the Gaussian kernel on $[-1,1]$ (Example~\ref{ex:kernel_Gaussian}-1).
Left: the weights computed by Algorithm~\ref{alg:gen_points} for $n=10$, 
Right: the power functions given by Algorithm~\ref{alg:gen_points} and the $P$-greedy algorithm for $n = 10$. 
Their zeros are the points given by these algorithms. }
\label{fig:weights_powfuncs_Gauss1D}

\end{figure}

\begin{figure}[t]

\centering
\begin{minipage}[t]{0.48\linewidth}
\includegraphics[width = \linewidth]{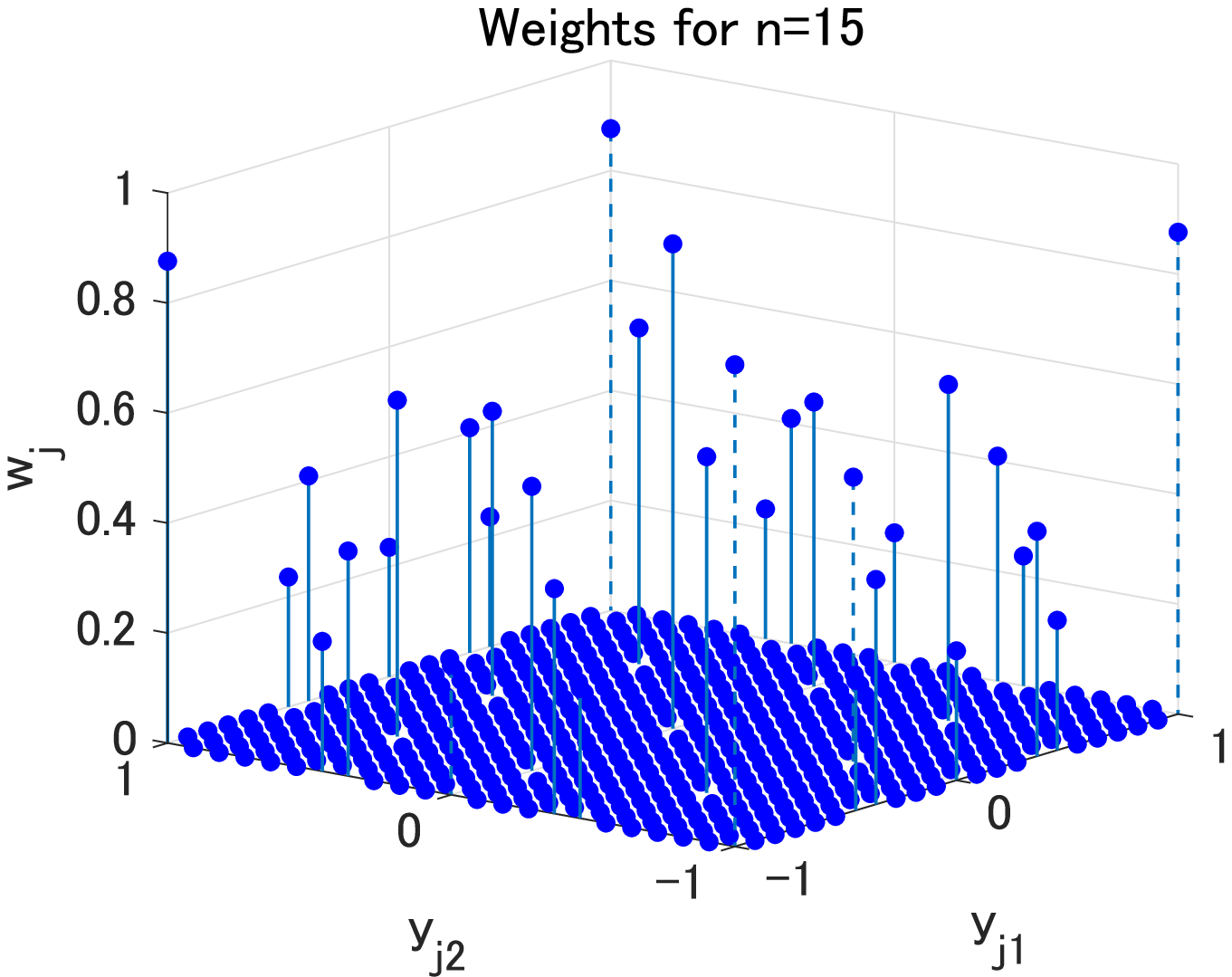}
\end{minipage}
\begin{minipage}[t]{0.48\linewidth}
\includegraphics[width = \linewidth]{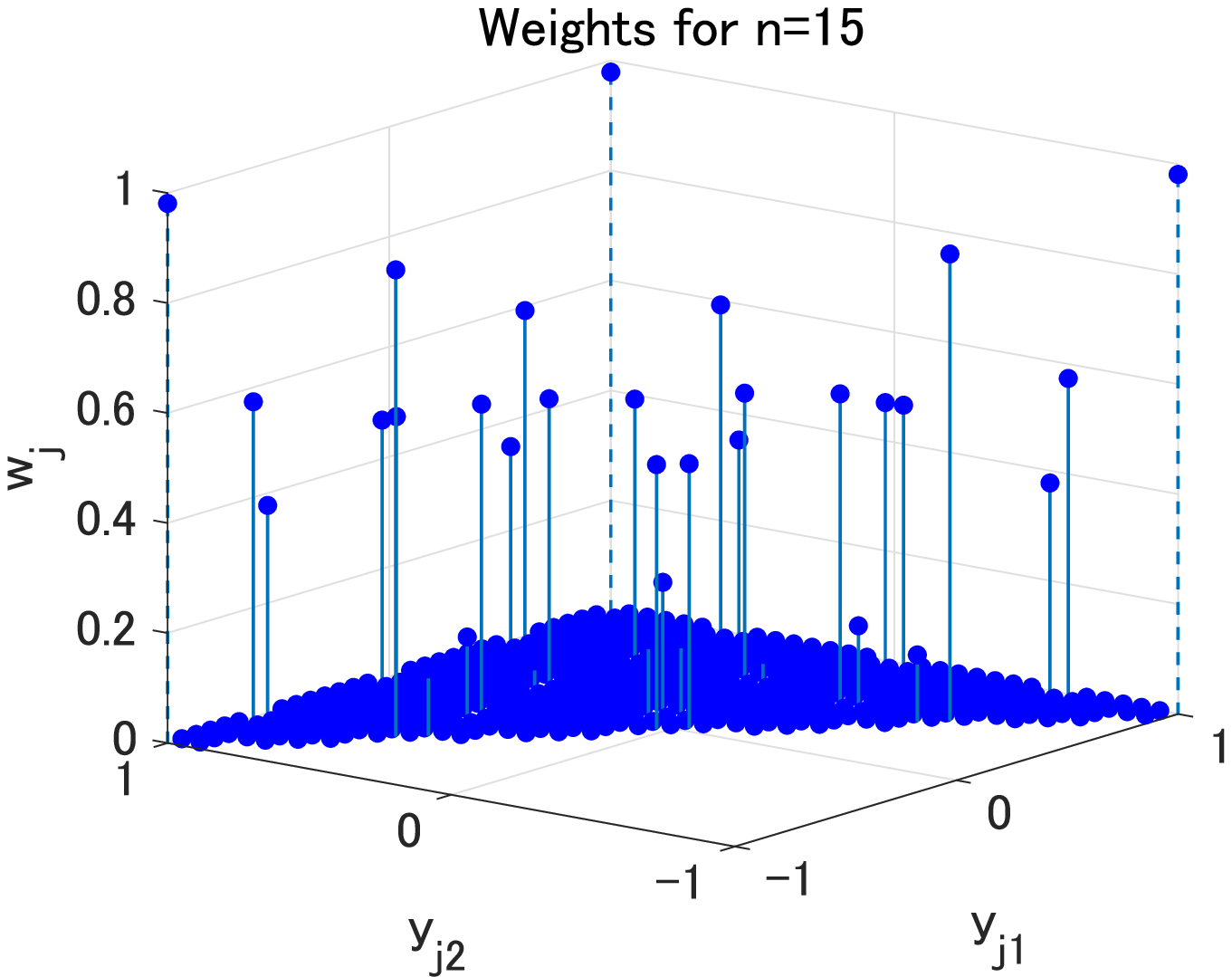}
\end{minipage}

\caption{Weights for the Gaussian kernel on $[-1,1]^{2}$ and $\triangle$ (Example~\ref{ex:kernel_Gaussian}-2, \ref{ex:kernel_Gaussian}-3).
Left: the weights on $[-1,1]^{2}$ given by Algorithm~\ref{alg:gen_points} for $n=15$, 
Right: the weights on $\triangle$ given by Algorithm~\ref{alg:gen_points} for $n = 15$. }
\label{fig:weights_Gauss2D}

\end{figure}

\begin{figure}[t]

\centering
\begin{minipage}[t]{0.32\linewidth}
\includegraphics[width = \linewidth]{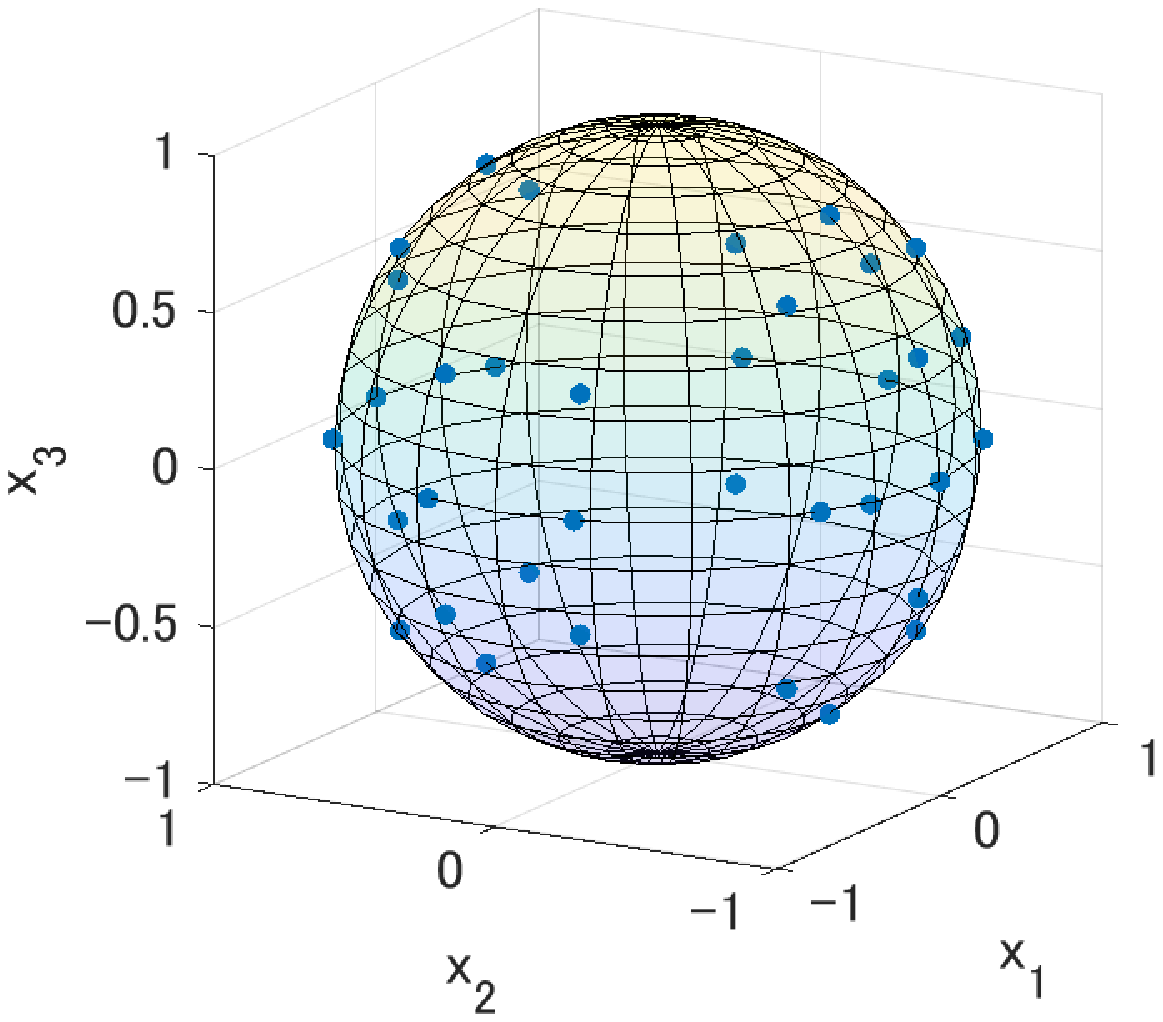}
\end{minipage}
\begin{minipage}[t]{0.32\linewidth}
\includegraphics[width = \linewidth]{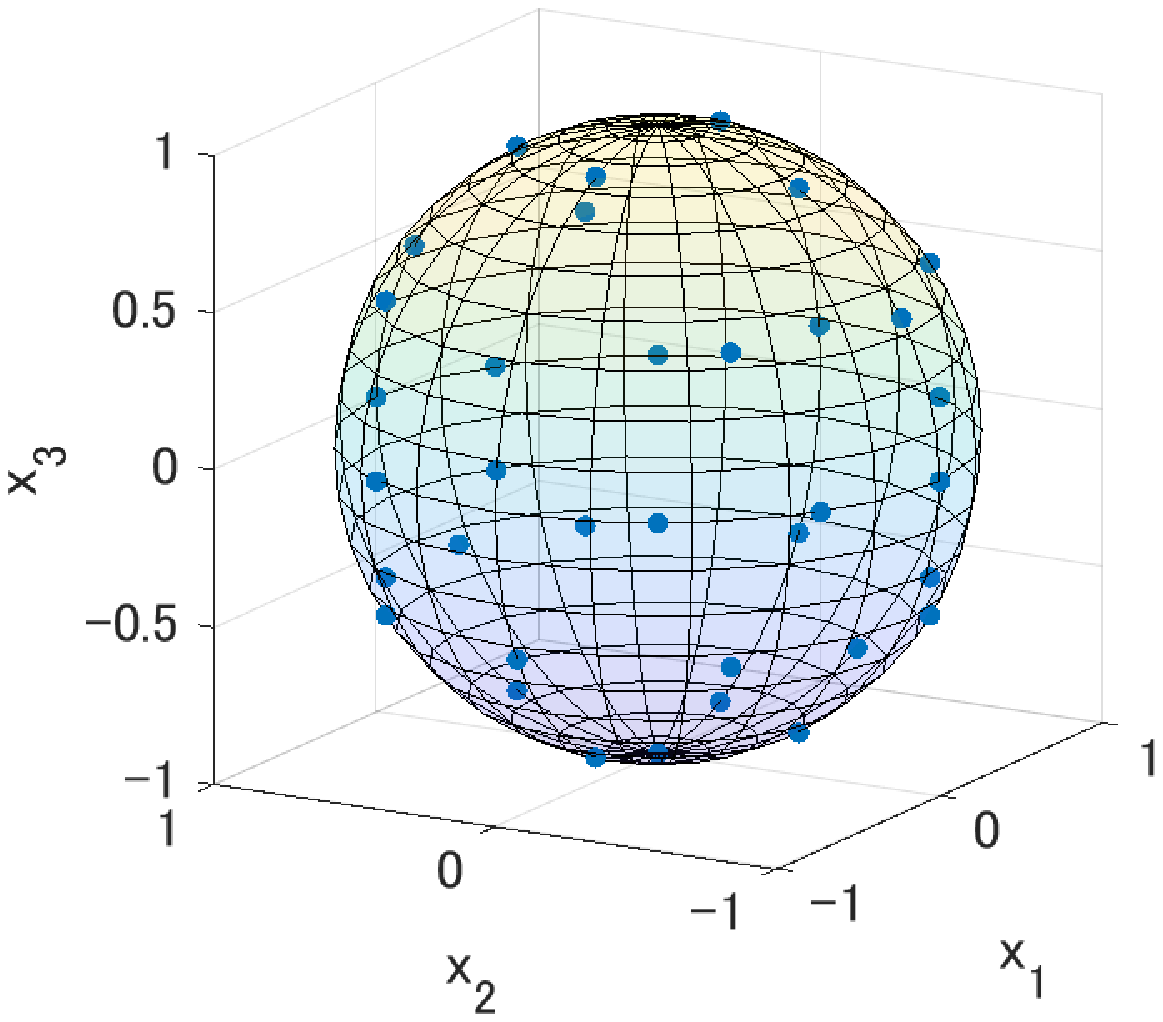}
\end{minipage}
\begin{minipage}[t]{0.32\linewidth}
\includegraphics[width = \linewidth]{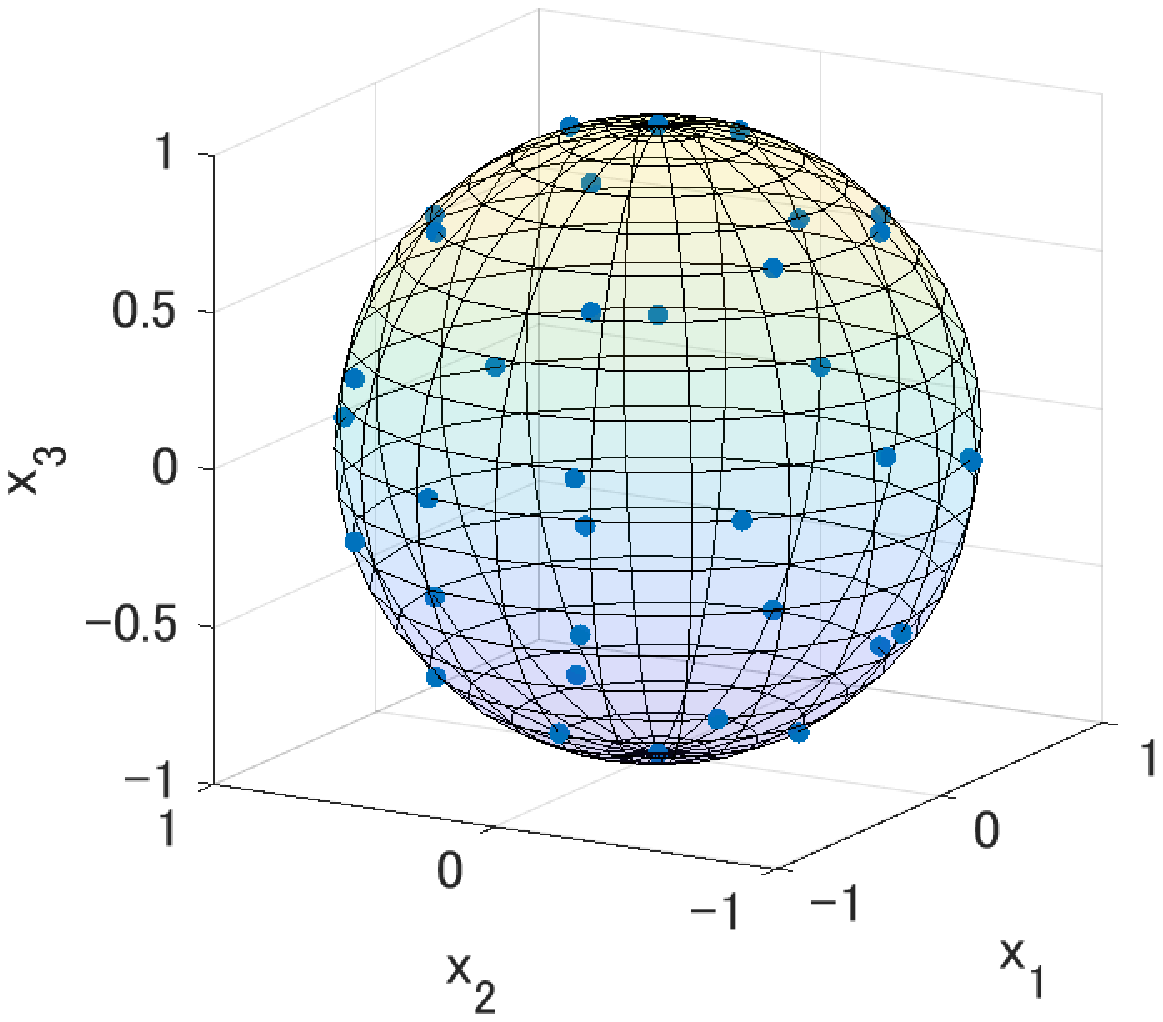}
\end{minipage}

\caption{Generated points on the sphere for the spherical inverse multiquadric kernels 
(Example~\ref{ex:kernel_Spherical}) in the case $n = 35$. 
Left: Algorithm~\ref{alg:gen_points}, 
Middle: Algorithm~\ref{alg:seq_gen_points}, 
Right: $P$-greedy algorithm.}
\label{fig:sphere_n35}

\end{figure}

\begin{figure}[t]

\centering
\begin{minipage}[t]{0.32\linewidth}
\includegraphics[width = \linewidth]{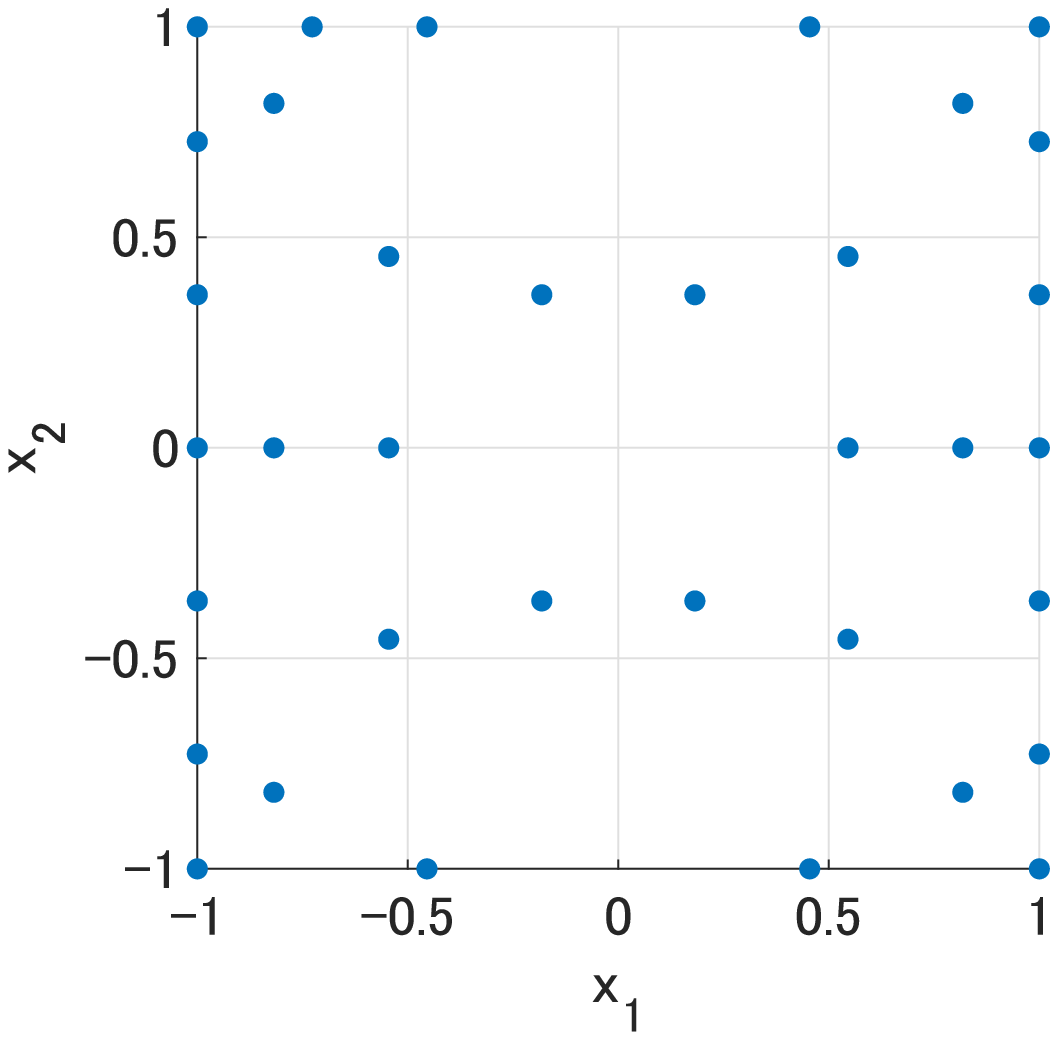}
\end{minipage}
\begin{minipage}[t]{0.32\linewidth}
\includegraphics[width = \linewidth]{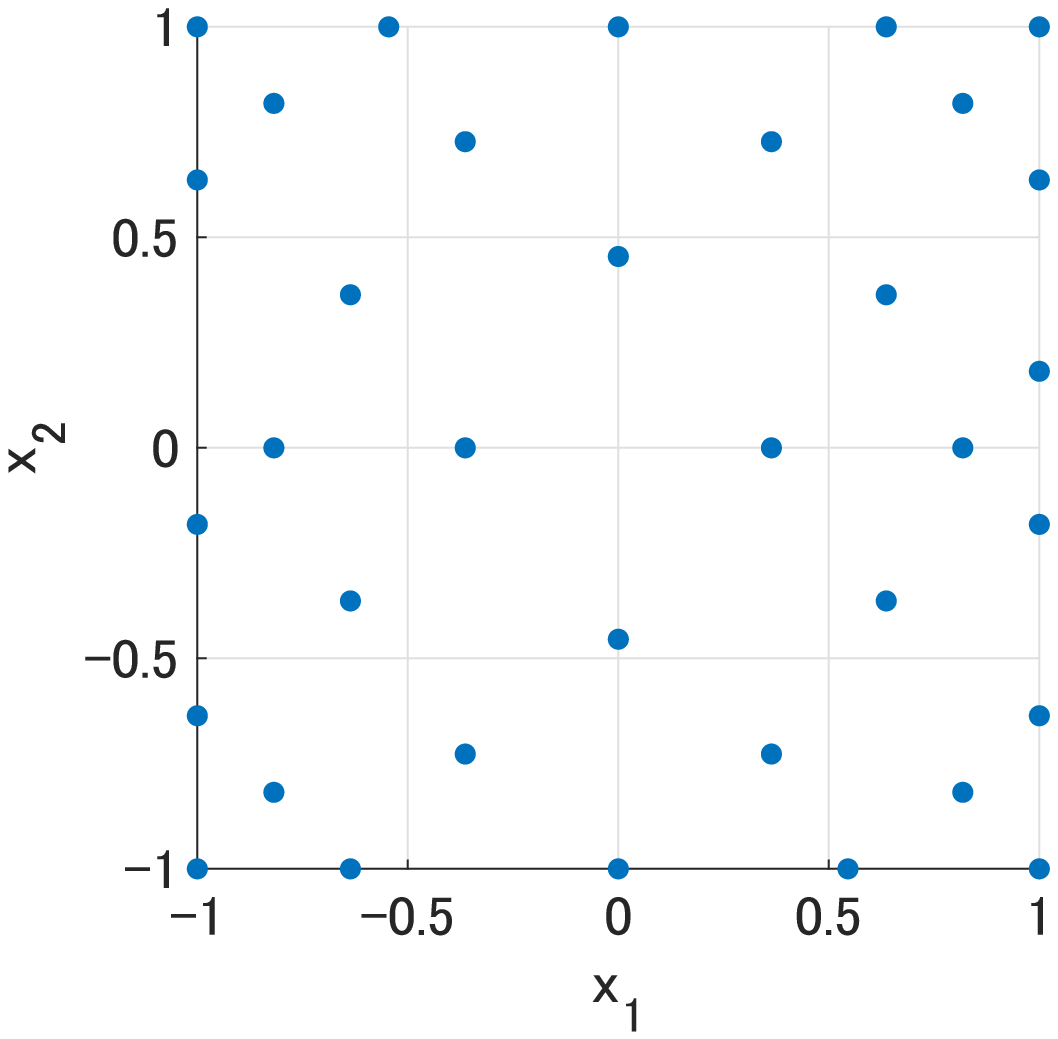}
\end{minipage}
\begin{minipage}[t]{0.32\linewidth}
\includegraphics[width = \linewidth]{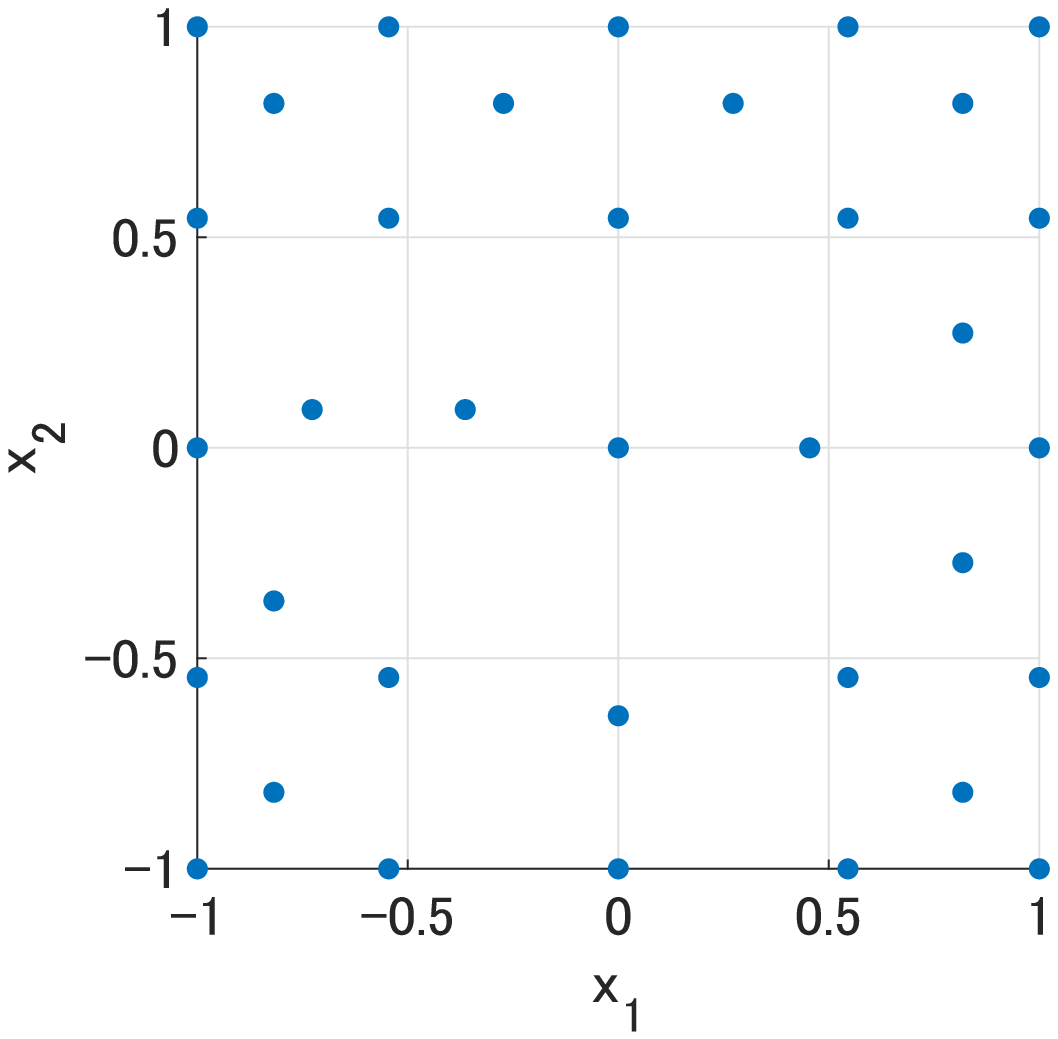}
\end{minipage}\\
(a) Generated points on $[-1,1]^{2}$ (Example~\ref{ex:kernel_Gaussian}-2).

\begin{minipage}[t]{0.32\linewidth}
\includegraphics[width = \linewidth]{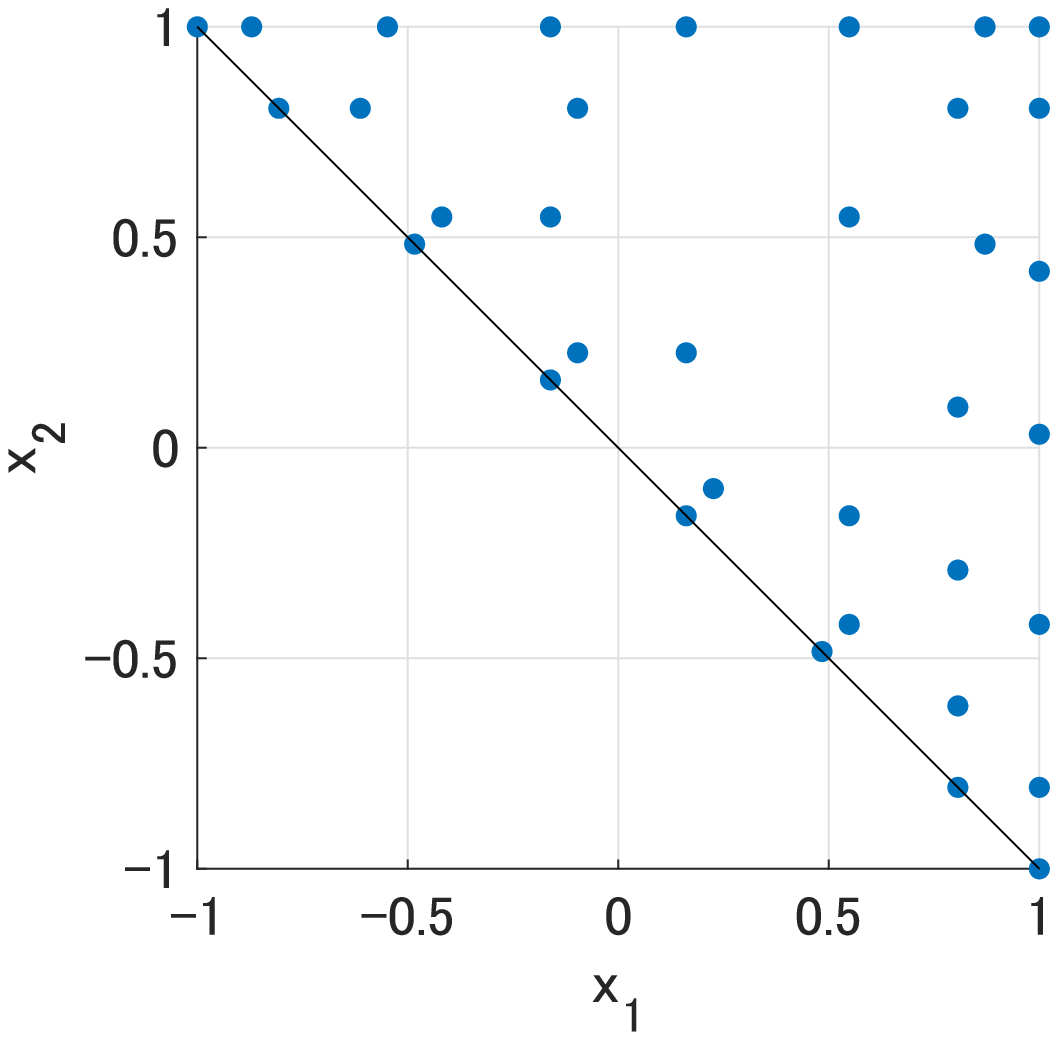}
\end{minipage}
\begin{minipage}[t]{0.32\linewidth}
\includegraphics[width = \linewidth]{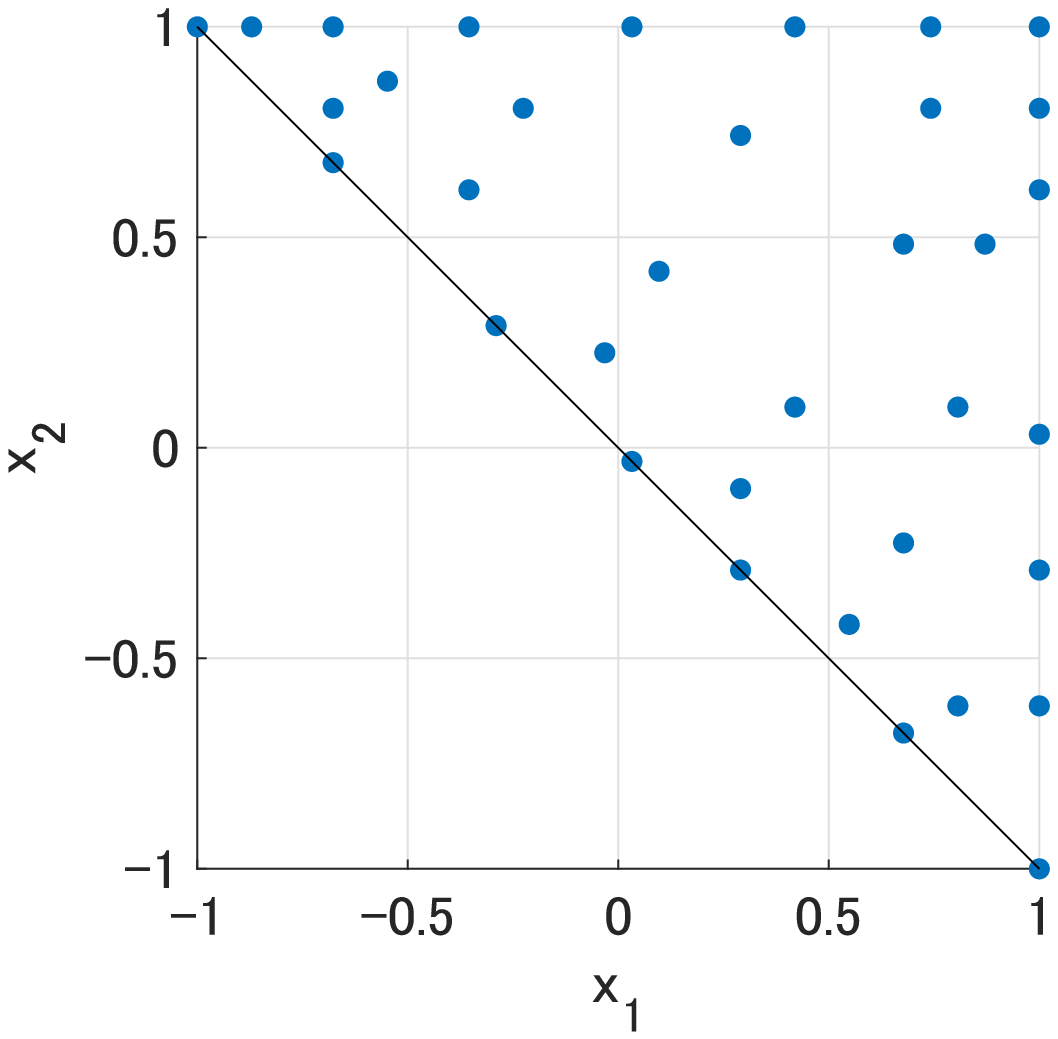}
\end{minipage}
\begin{minipage}[t]{0.32\linewidth}
\includegraphics[width = \linewidth]{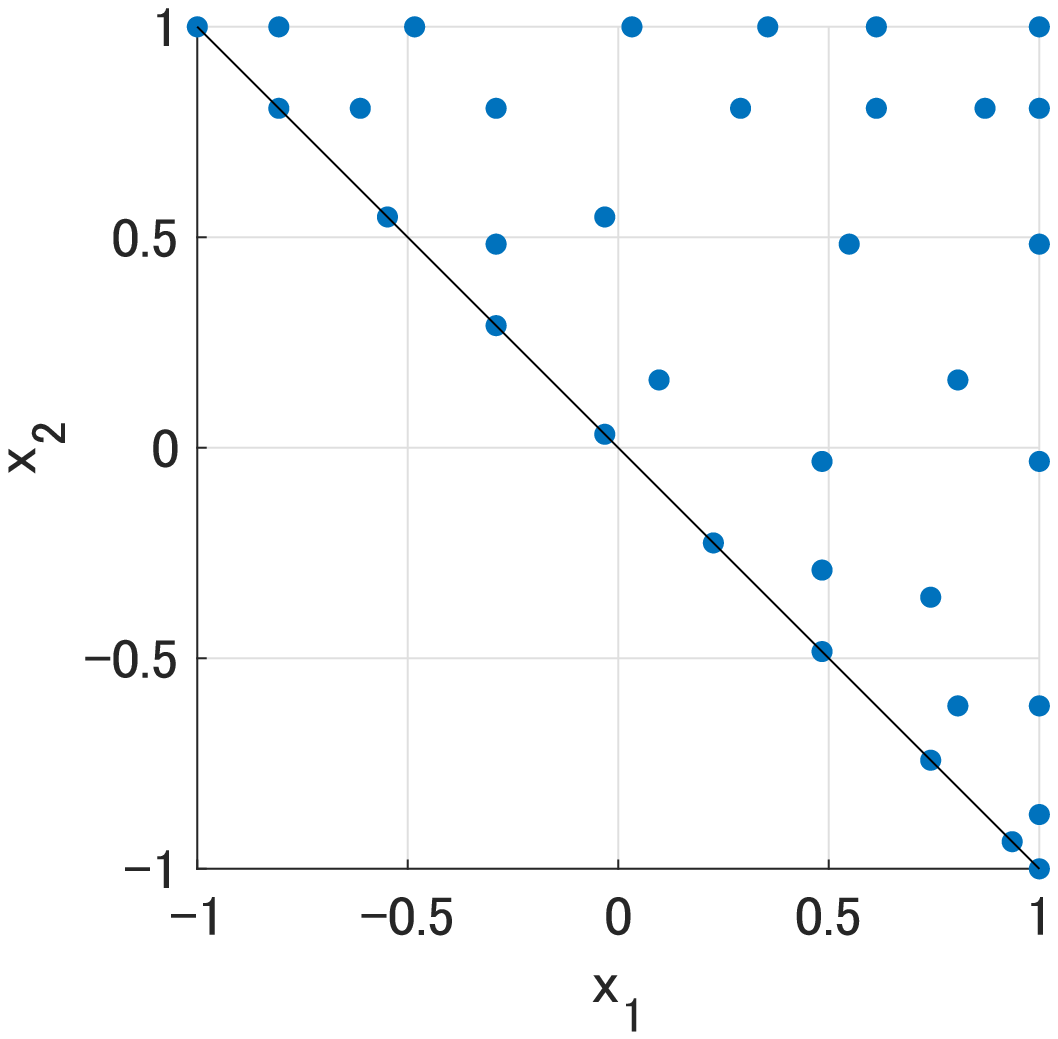}
\end{minipage}\\
(b) Generated points on $\triangle \subset \mathbf{R}^{2}$ (Example~\ref{ex:kernel_Gaussian}-3).

\begin{minipage}[t]{0.32\linewidth}
\includegraphics[width = \linewidth]{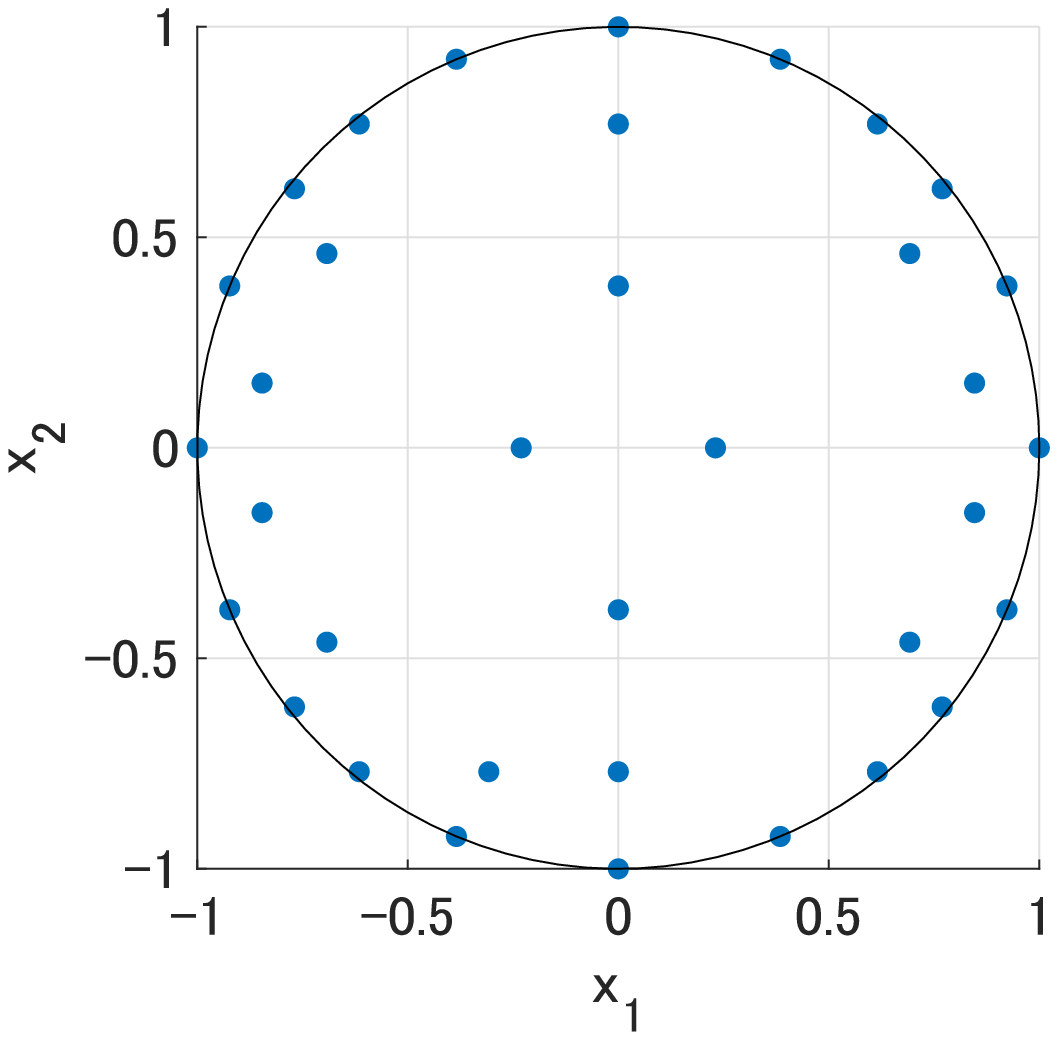}
\end{minipage}
\begin{minipage}[t]{0.32\linewidth}
\includegraphics[width = \linewidth]{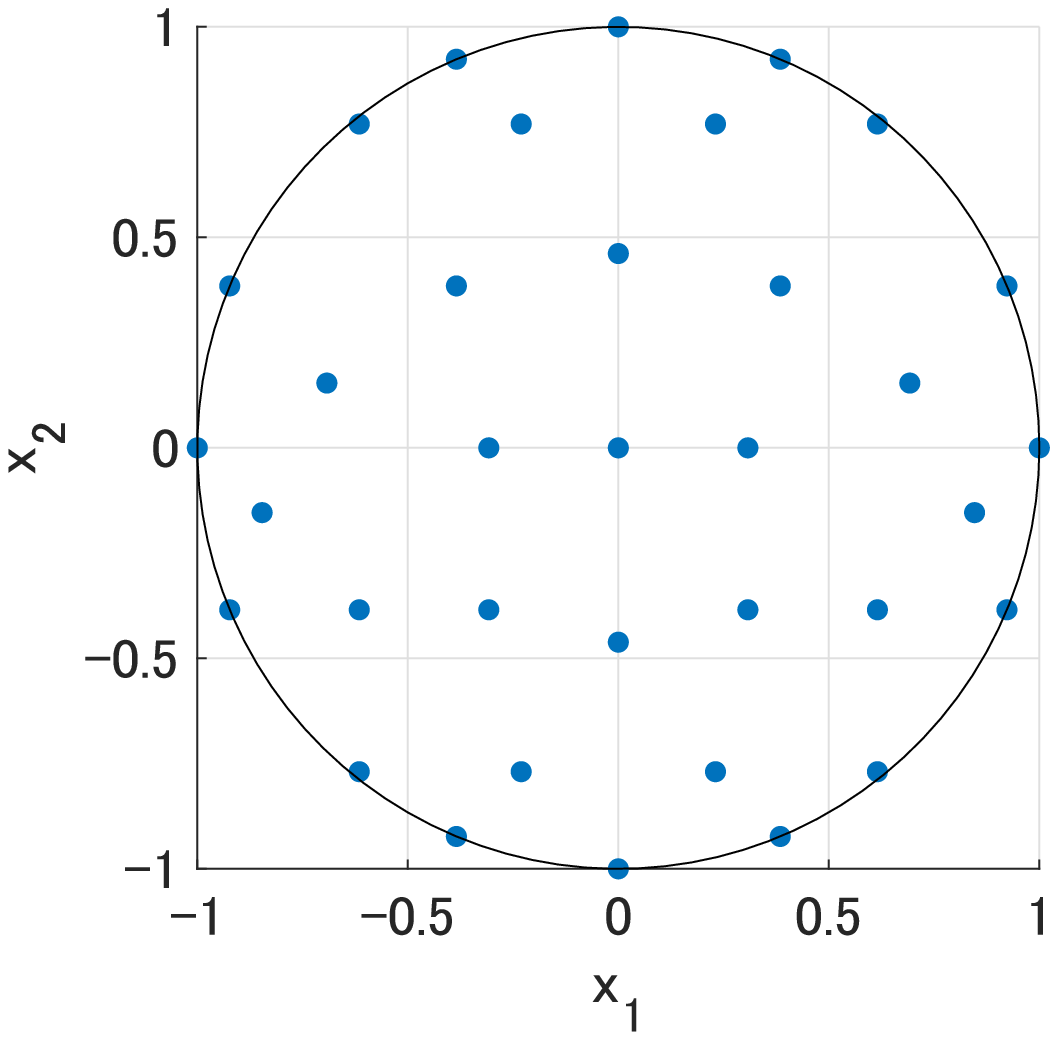}
\end{minipage}
\begin{minipage}[t]{0.32\linewidth}
\includegraphics[width = \linewidth]{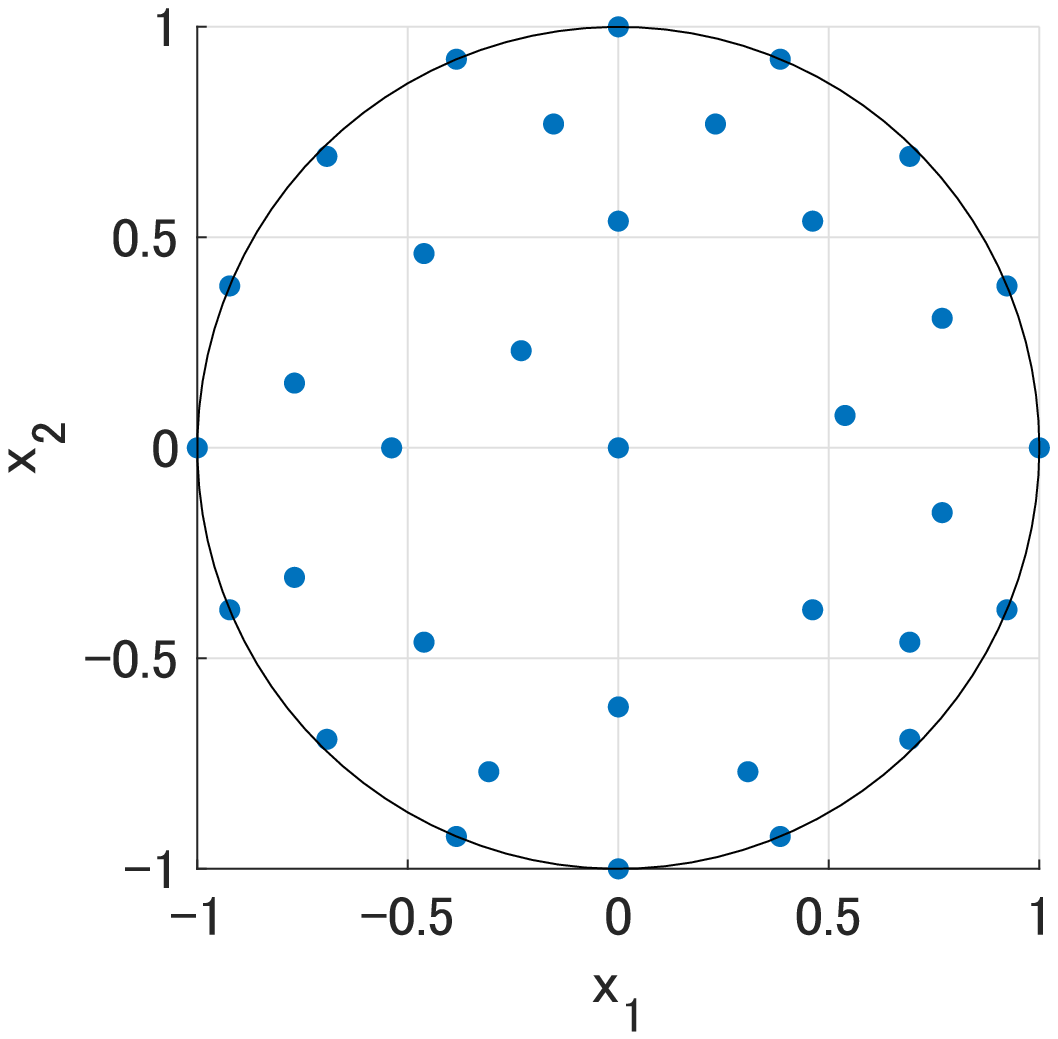}
\end{minipage}\\
(c) Generated points on $D \subset \mathbf{R}^{2}$ (Example~\ref{ex:kernel_Gaussian}-4).

\caption{Generated points for the Gaussian kernels (Example~\ref{ex:kernel_Gaussian}) in the case $n = 35$. 
Left: Algorithm~\ref{alg:gen_points}, 
Middle: Algorithm~\ref{alg:seq_gen_points}, 
Right: $P$-greedy algorithm.}
\label{fig:Gauss2D_n35}

\end{figure}

\begin{figure}[t]
\centering

\begin{minipage}[t]{0.48\linewidth}
\includegraphics[width = \linewidth]{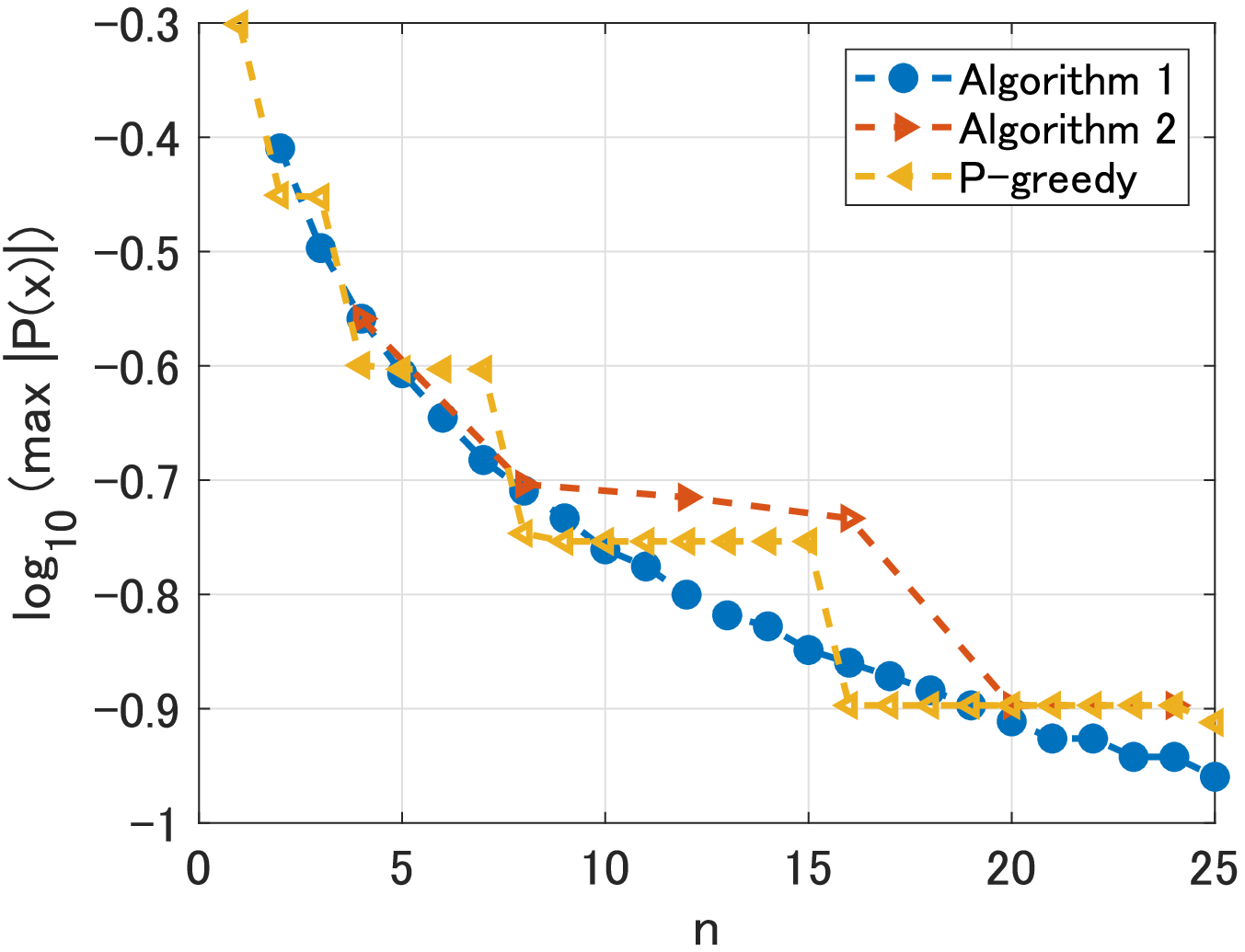}
\caption{Maximum values of the power functions for Brownian kernel on $[0,1]$ (Example~\ref{ex:kernel_Brownian}).}
\label{fig:Brown1D_pf}
\end{minipage}
\quad
\begin{minipage}[t]{0.48\linewidth}
\includegraphics[width = \linewidth]{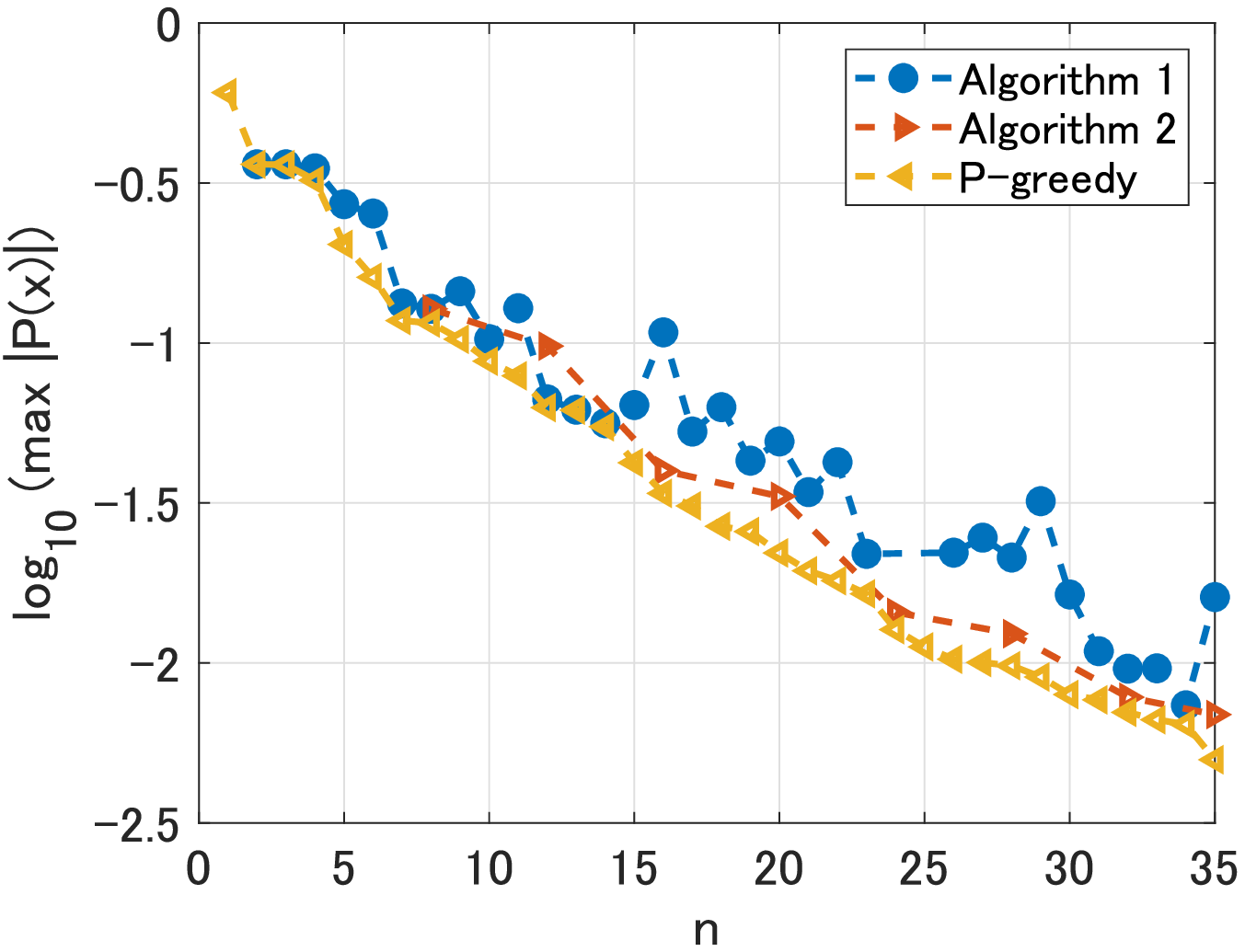}
\caption{Maximum values of the power functions 
for the spherical inverse multiquadric kernel on $S^{2} \subset \mathbf{R}^{3}$ (Example~\ref{ex:kernel_Spherical}).}
\label{fig:sphere_pf}
\end{minipage}

\begin{minipage}[t]{0.48\linewidth}
\includegraphics[width = \linewidth]{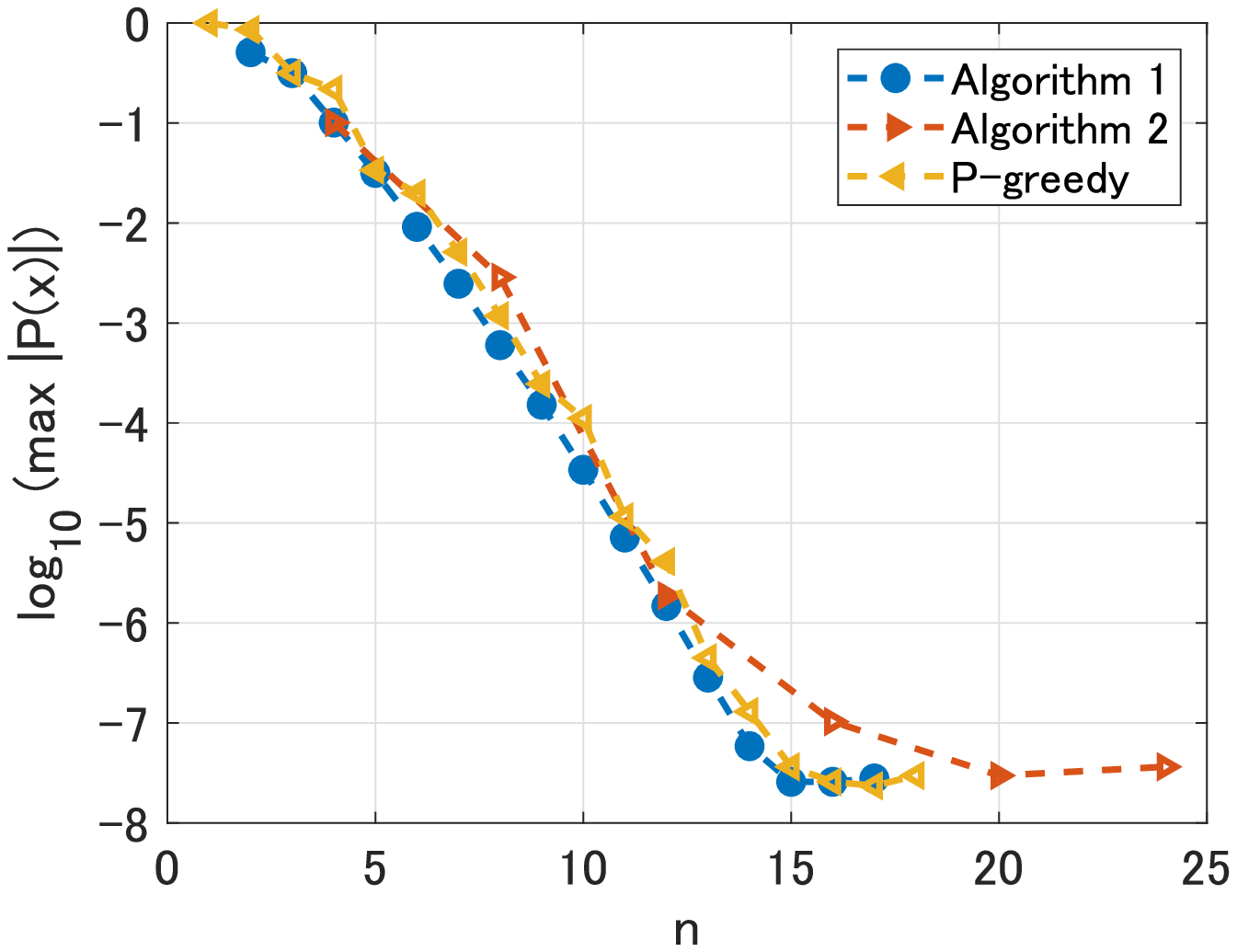}
\caption{Maximum values of the power functions for the Gaussian kernel on $[-1,1]$ (Example~\ref{ex:kernel_Gaussian}-1).}
\label{fig:Gauss1D_pf}
\end{minipage}
\quad
\begin{minipage}[t]{0.48\linewidth}
\includegraphics[width = \linewidth]{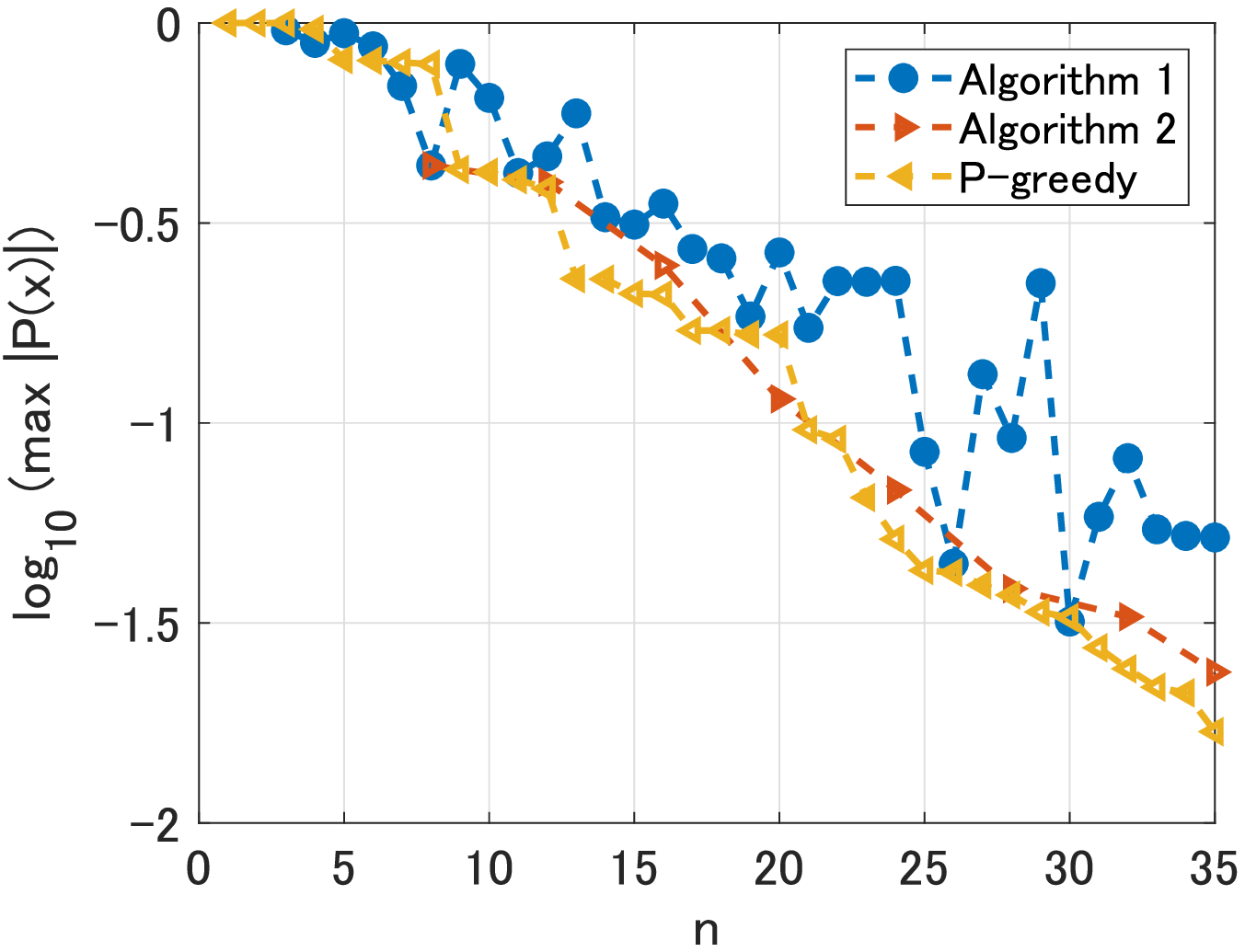}
\caption{Maximum values of the power functions for the Gaussian kernel on $[-1,1]^2 \subset \mathbf{R}^{2}$ (Example~\ref{ex:kernel_Gaussian}-2).}
\label{fig:Gauss2D_sqr_pf}
\end{minipage}
\end{figure}

\begin{figure}[t]
\centering
\begin{minipage}[t]{0.48\linewidth}
\includegraphics[width = \linewidth]{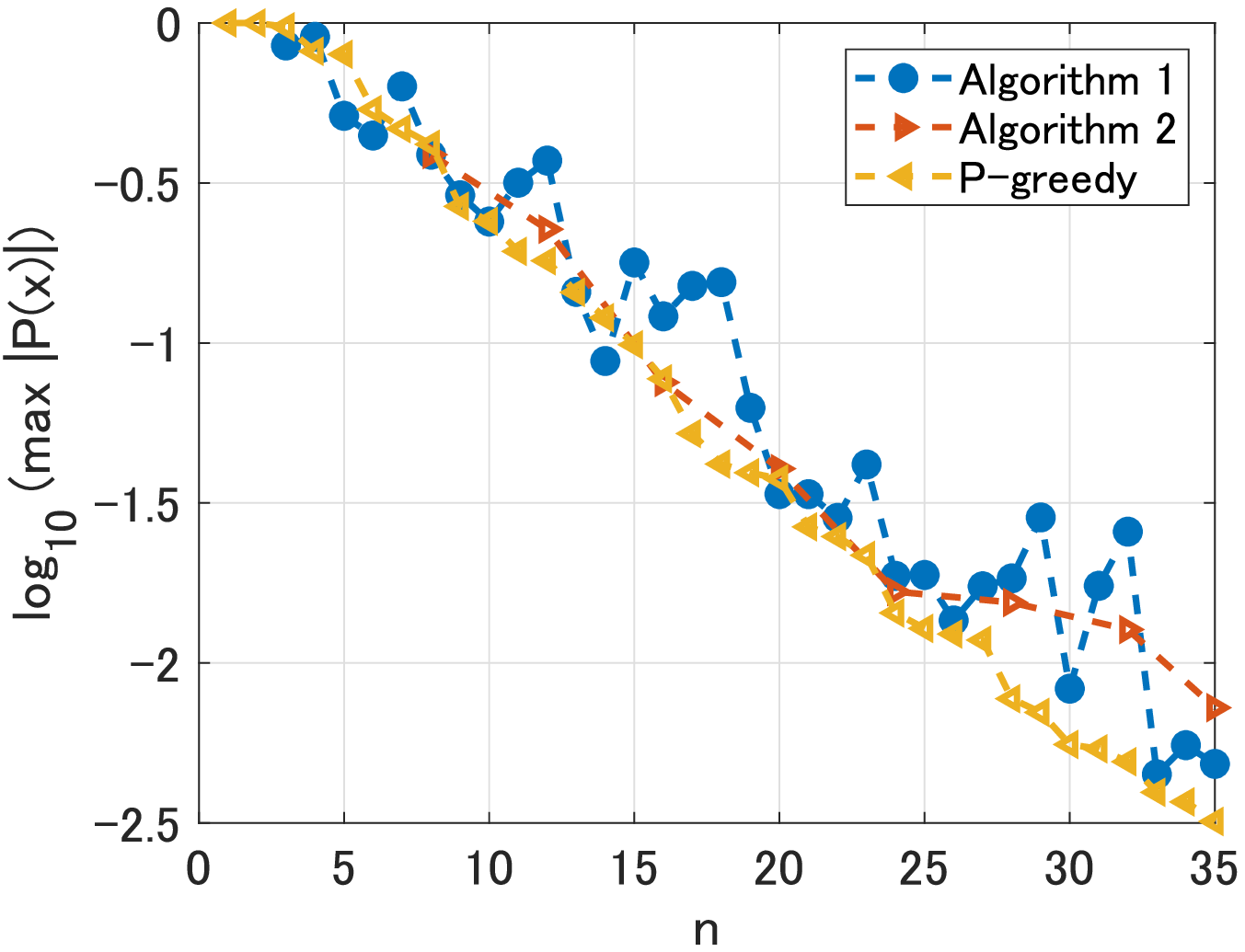}
\caption{Maximum values of the power functions for the Gaussian kernel on $\triangle \subset \mathbf{R}^{2}$ (Example~\ref{ex:kernel_Gaussian}-3).}
\label{fig:Gauss2D_tri_pf}
\end{minipage}
\quad
\begin{minipage}[t]{0.48\linewidth}
\includegraphics[width = \linewidth]{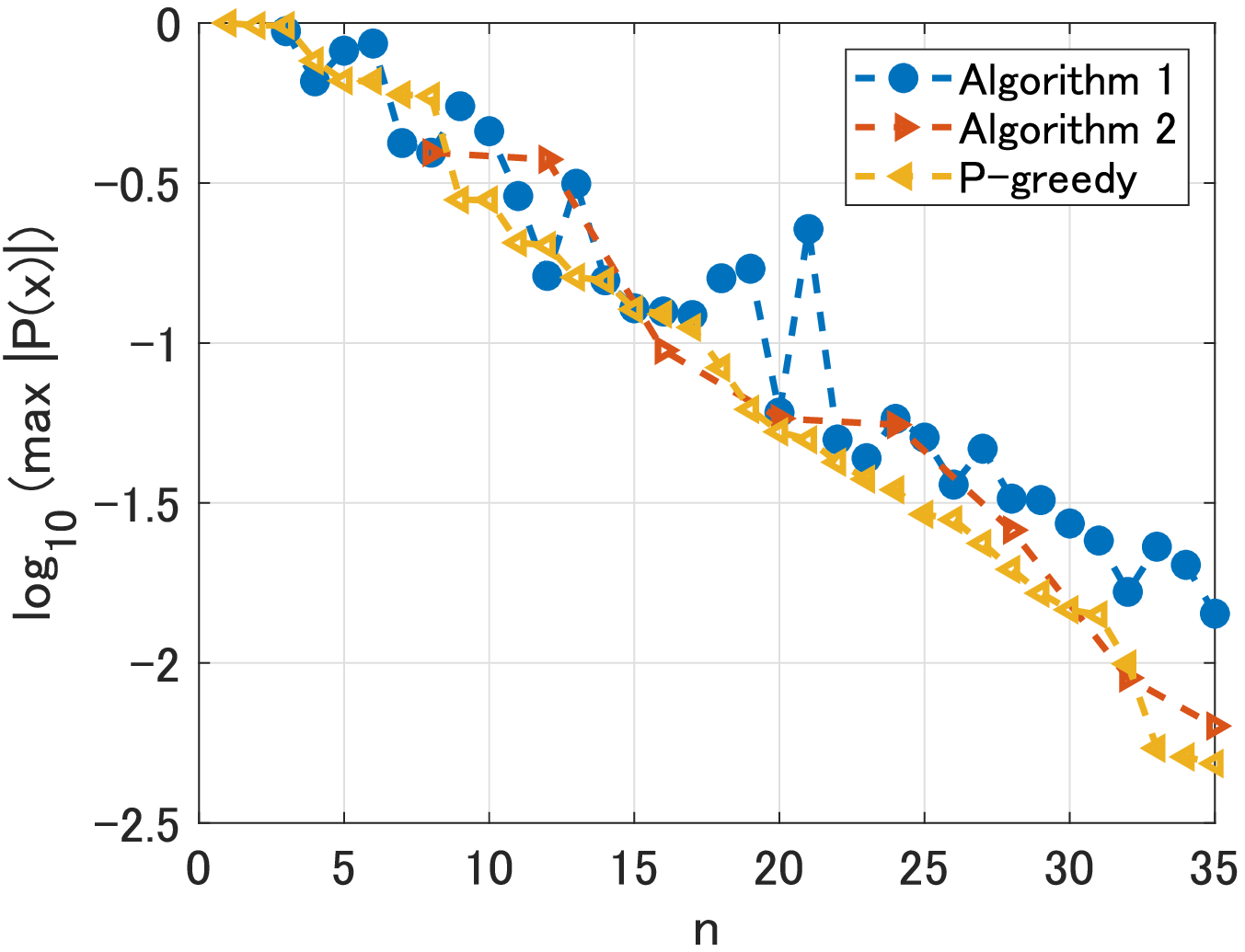}
\caption{Maximum values of the power functions for the Gaussian kernel on $D \subset \mathbf{R}^{2}$ (Example~\ref{ex:kernel_Gaussian}-4).}
\label{fig:Gauss2D_dsk_pf}
\end{minipage}
\end{figure}

\begin{figure}[t]

\centering
\begin{minipage}[t]{0.48\linewidth}
\includegraphics[width = \linewidth]{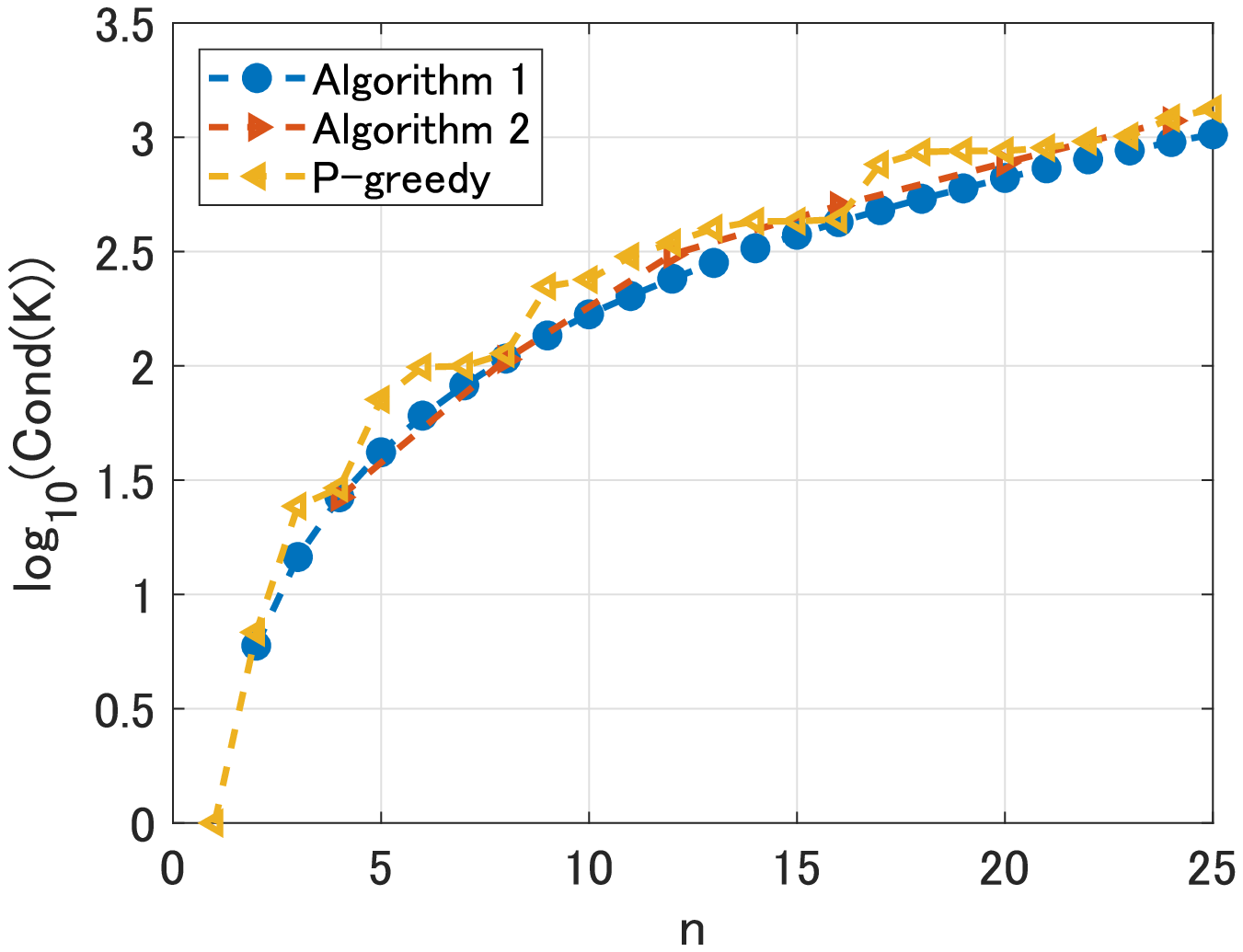}
\caption{Condition numbers of the kernel matrices for Brownian kernel on $[0,1]$ (Example~\ref{ex:kernel_Brownian}).}
\label{fig:Brown1D_cnd}
\end{minipage}
\quad
\begin{minipage}[t]{0.48\linewidth}
\includegraphics[width = \linewidth]{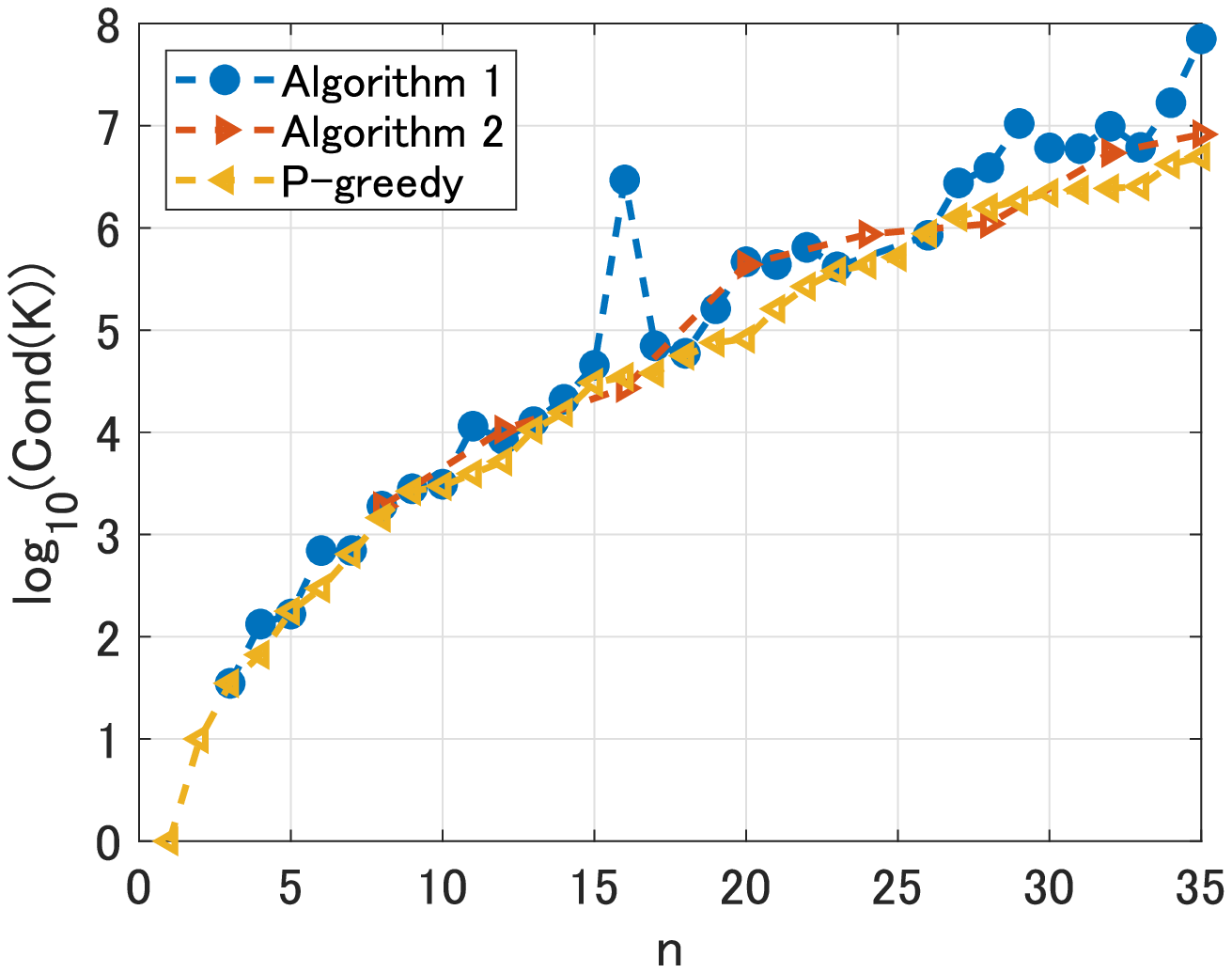}
\caption{Condition numbers of the kernel matrices 
for the spherical inverse multiquadric kernel on $S^{2} \subset \mathbf{R}^{3}$ (Example~\ref{ex:kernel_Spherical}).}
\label{fig:sphere_cnd}
\end{minipage}

\begin{minipage}[t]{0.48\linewidth}
\includegraphics[width = \linewidth]{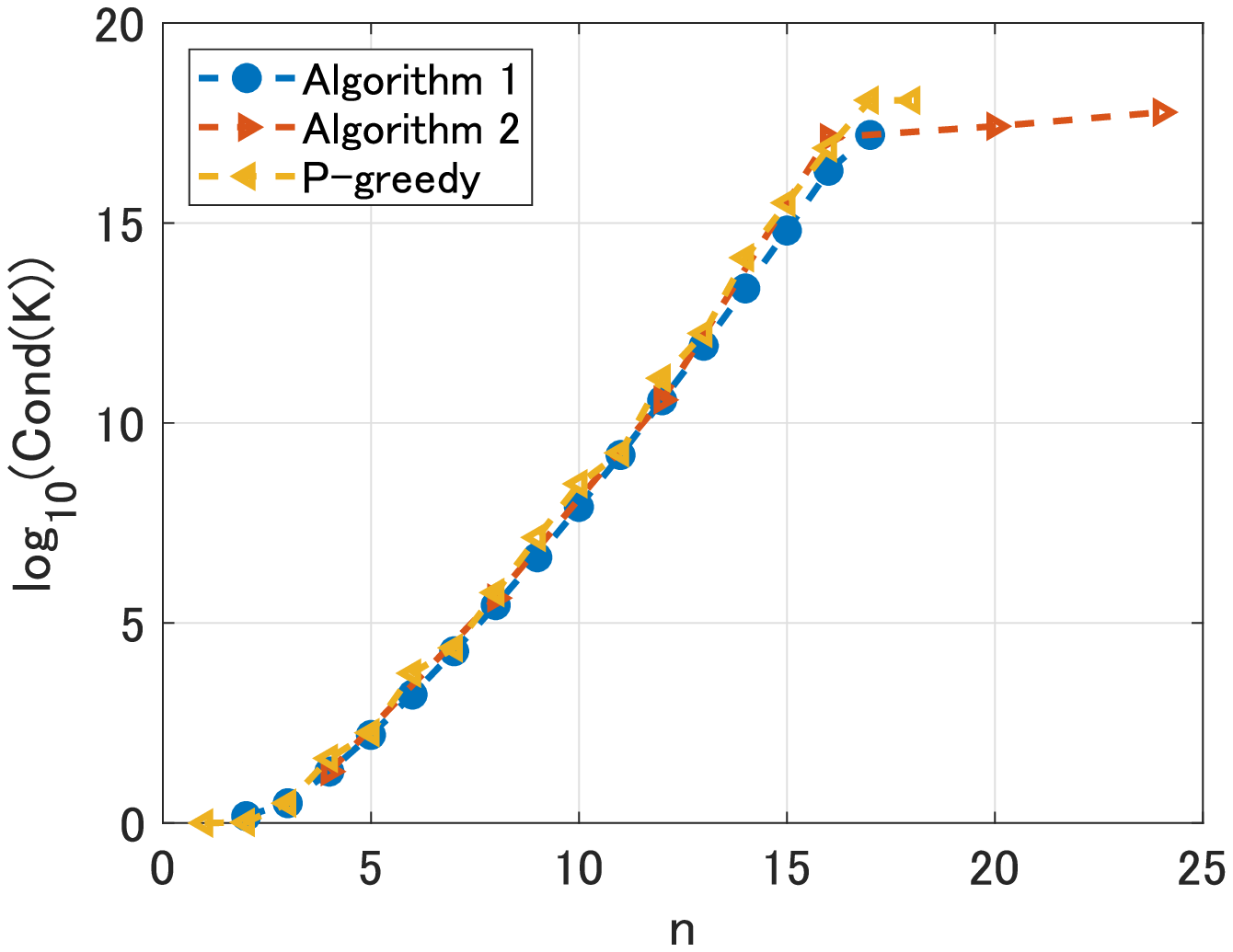}
\caption{Condition numbers of the kernel matrices for the Gaussian kernel on $[-1,1]$ (Example~\ref{ex:kernel_Gaussian}-1).}
\label{fig:Gauss1D_cnd}
\end{minipage}
\quad
\begin{minipage}[t]{0.48\linewidth}
\includegraphics[width = \linewidth]{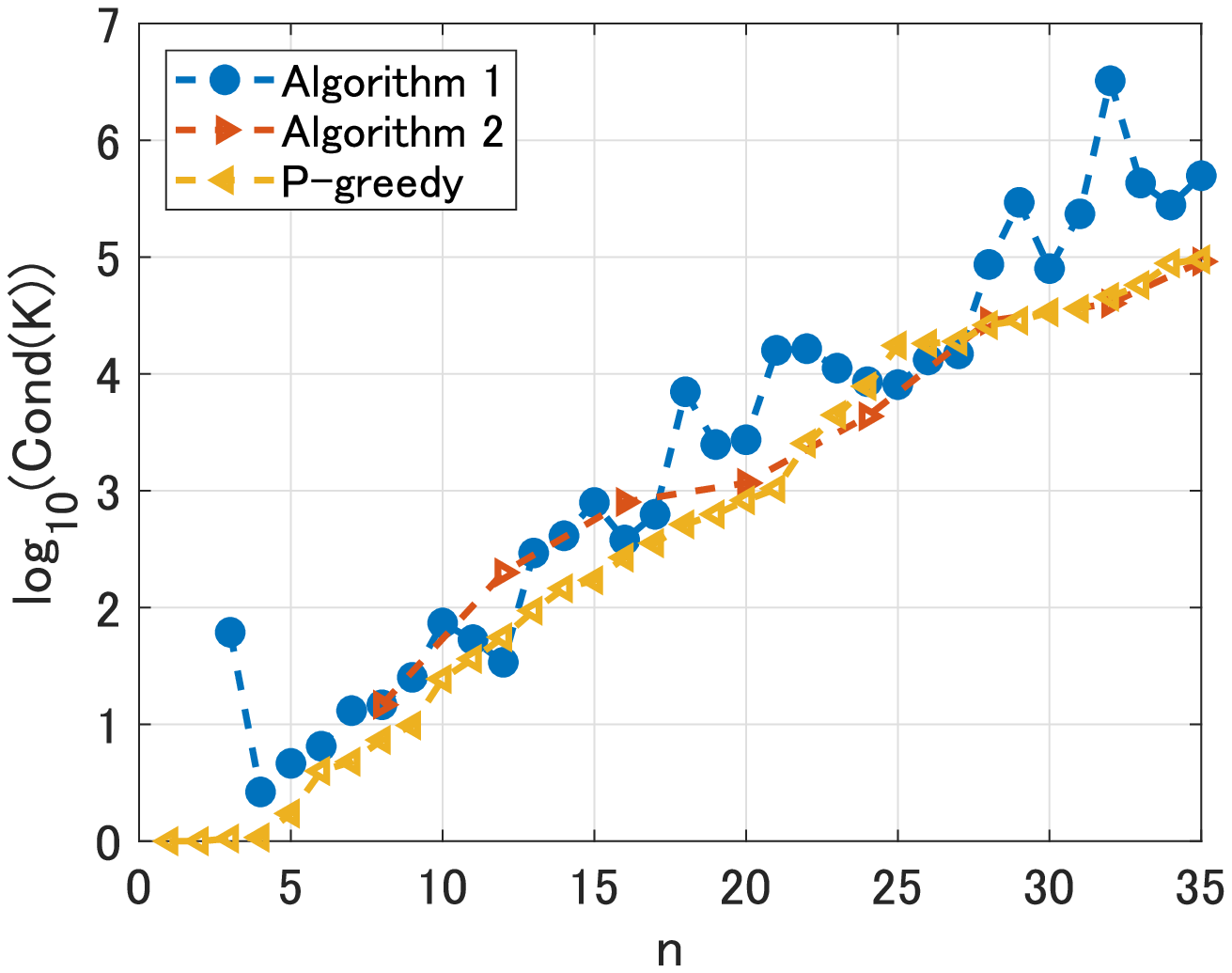}
\caption{Condition numbers of the kernel matrices for the Gaussian kernel on $[-1,1]^2 \subset \mathbf{R}^{2}$ (Example~\ref{ex:kernel_Gaussian}-2).}
\label{fig:Gauss2D_sqr_cnd}
\end{minipage}

\end{figure}

\begin{figure}[t]

\centering
\begin{minipage}[t]{0.48\linewidth}
\includegraphics[width = \linewidth]{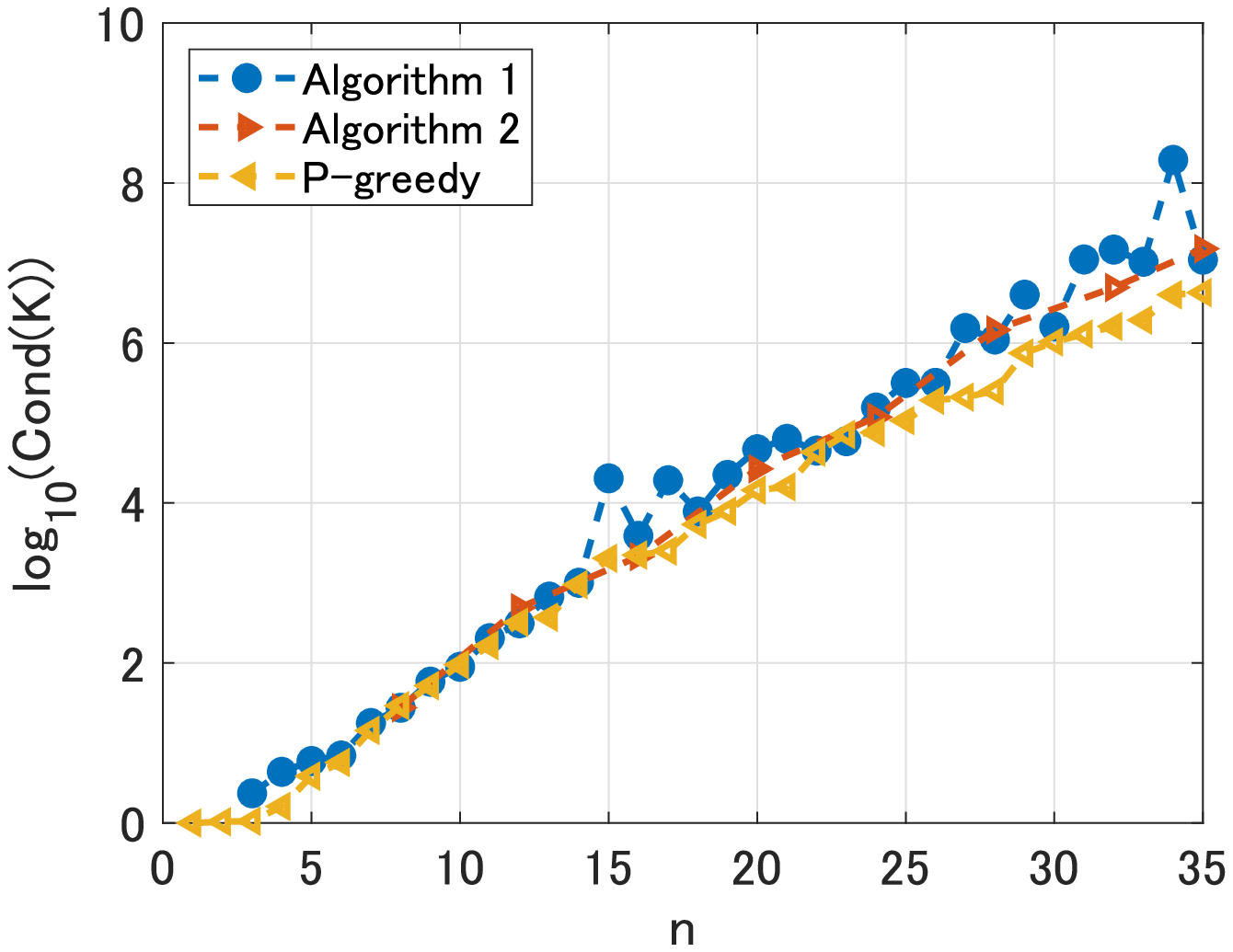}
\caption{Condition numbers of the kernel matrices for the Gaussian kernel on $\triangle \subset \mathbf{R}^{2}$ (Example~\ref{ex:kernel_Gaussian}-3).}
\label{fig:Gauss2D_tri_cnd}
\end{minipage}
\quad
\begin{minipage}[t]{0.48\linewidth}
\includegraphics[width = \linewidth]{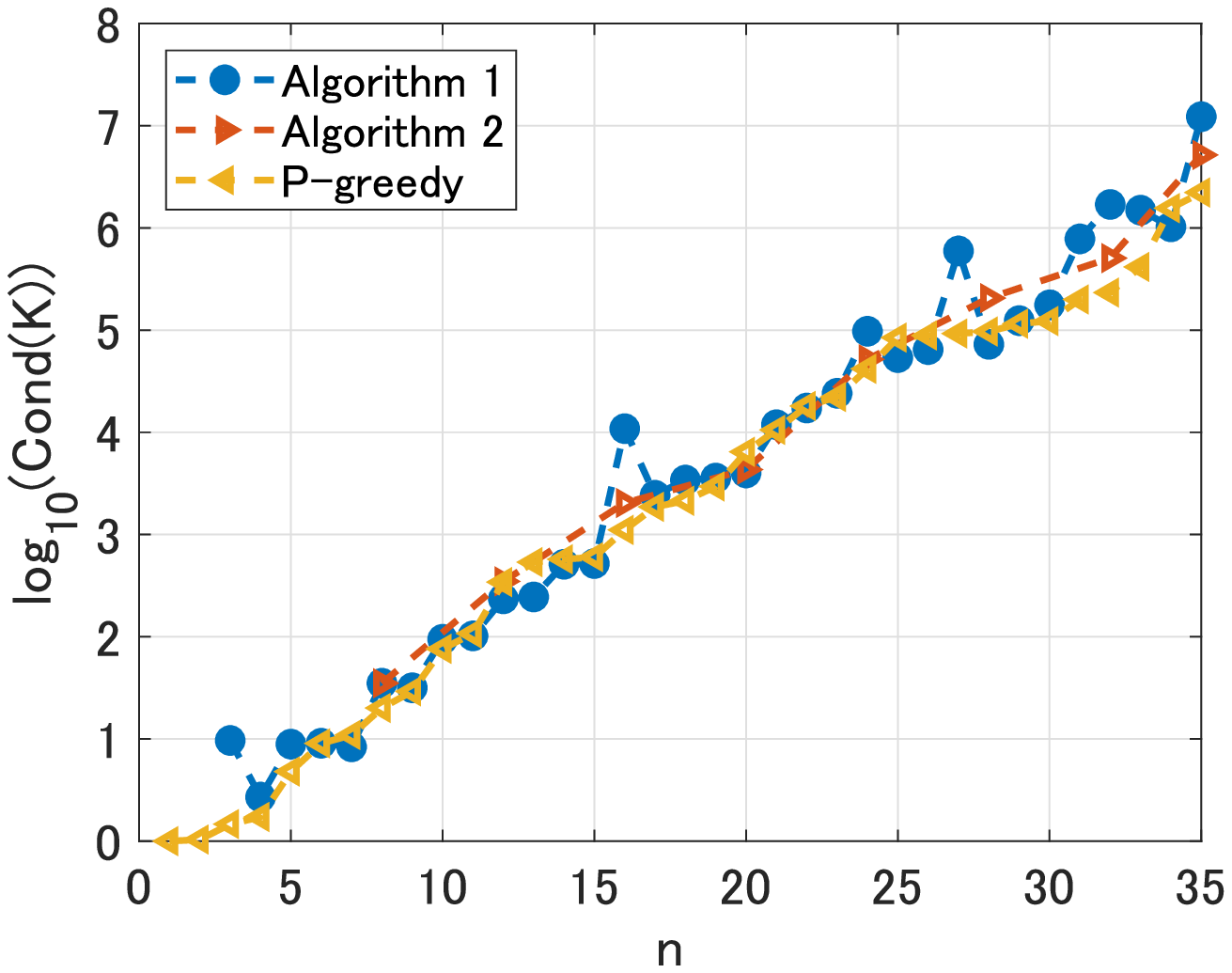}
\caption{Condition numbers of the kernel matrices for the Gaussian kernel on $D \subset \mathbf{R}^{2}$ (Example~\ref{ex:kernel_Gaussian}-4).}
\label{fig:Gauss2D_dsk_cnd}
\end{minipage}
\end{figure}

\begin{figure}[t]

\centering
\begin{minipage}[t]{0.48\linewidth}
\includegraphics[width = \linewidth]{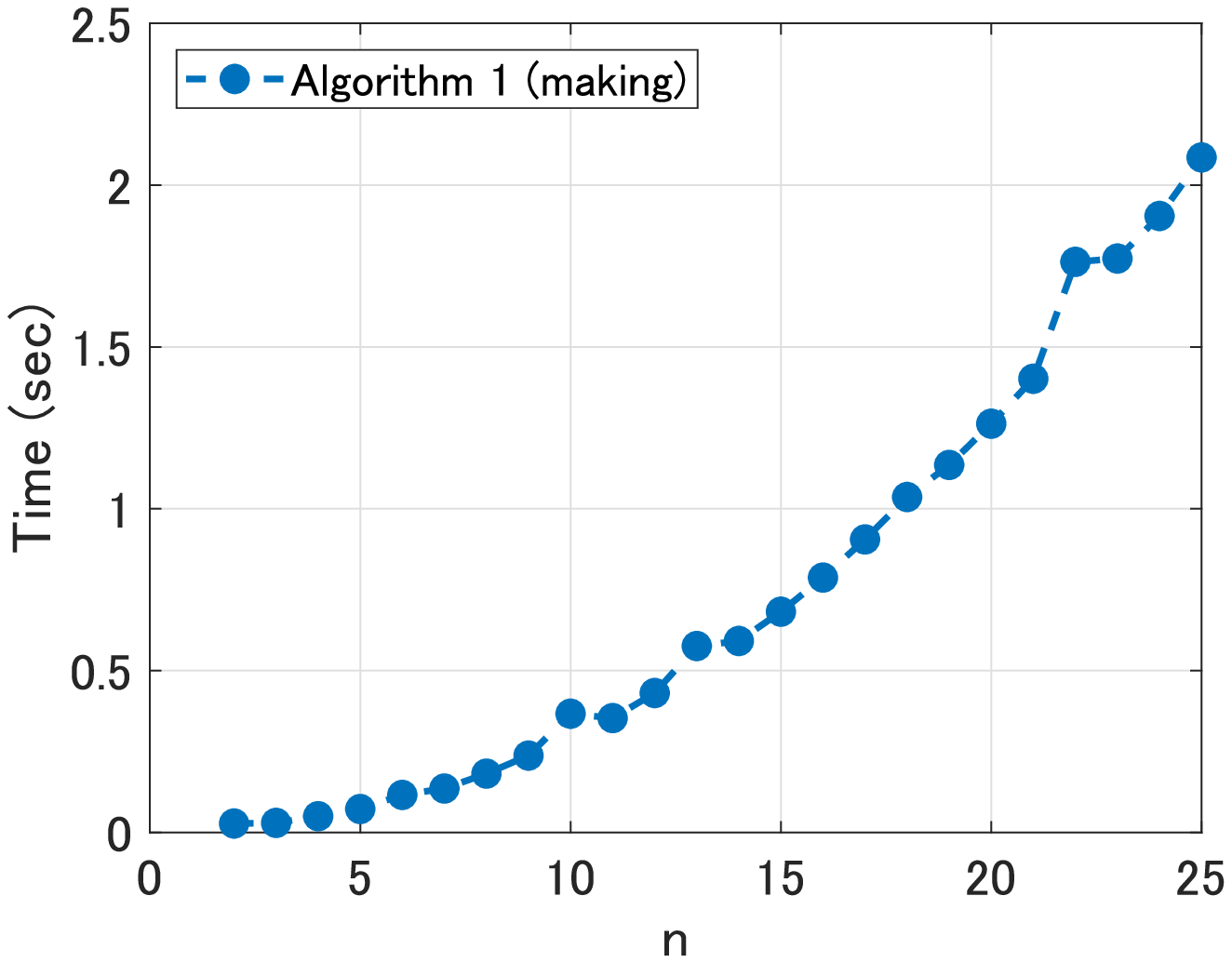}
\end{minipage}
\begin{minipage}[t]{0.48\linewidth}
\includegraphics[width = \linewidth]{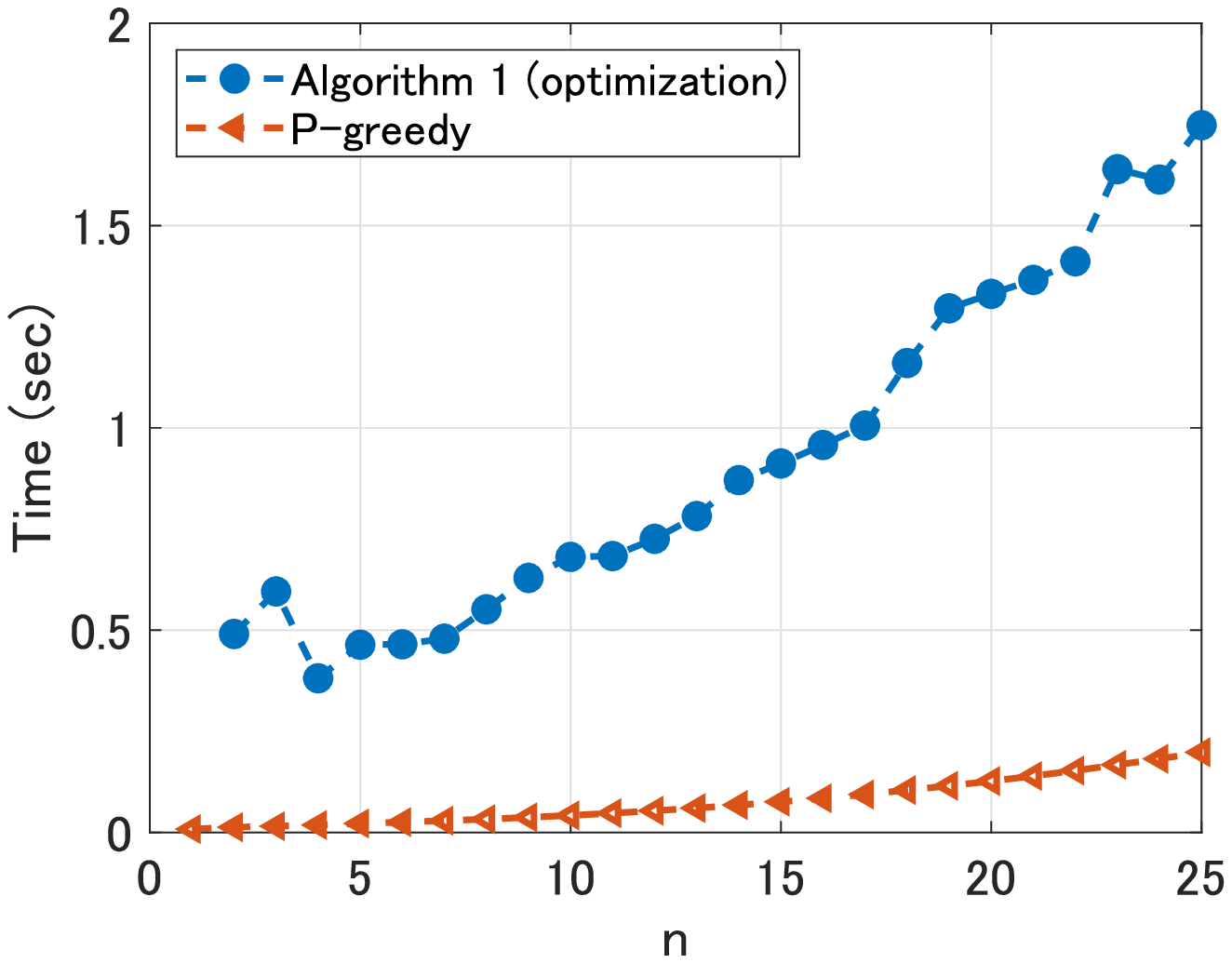}
\end{minipage}

\caption{Computation times for Brownian kernel on $[0,1]$ (Example~\ref{ex:kernel_Brownian}).
Left: the times for making the SOCP instances of Algorithm~\ref{alg:gen_points}, 
Right: the times for executing the SOCP optimizer and the $P$-greedy algorithm.}
\label{fig:Brown_time}

\begin{minipage}[t]{0.48\linewidth}
\includegraphics[width = \linewidth]{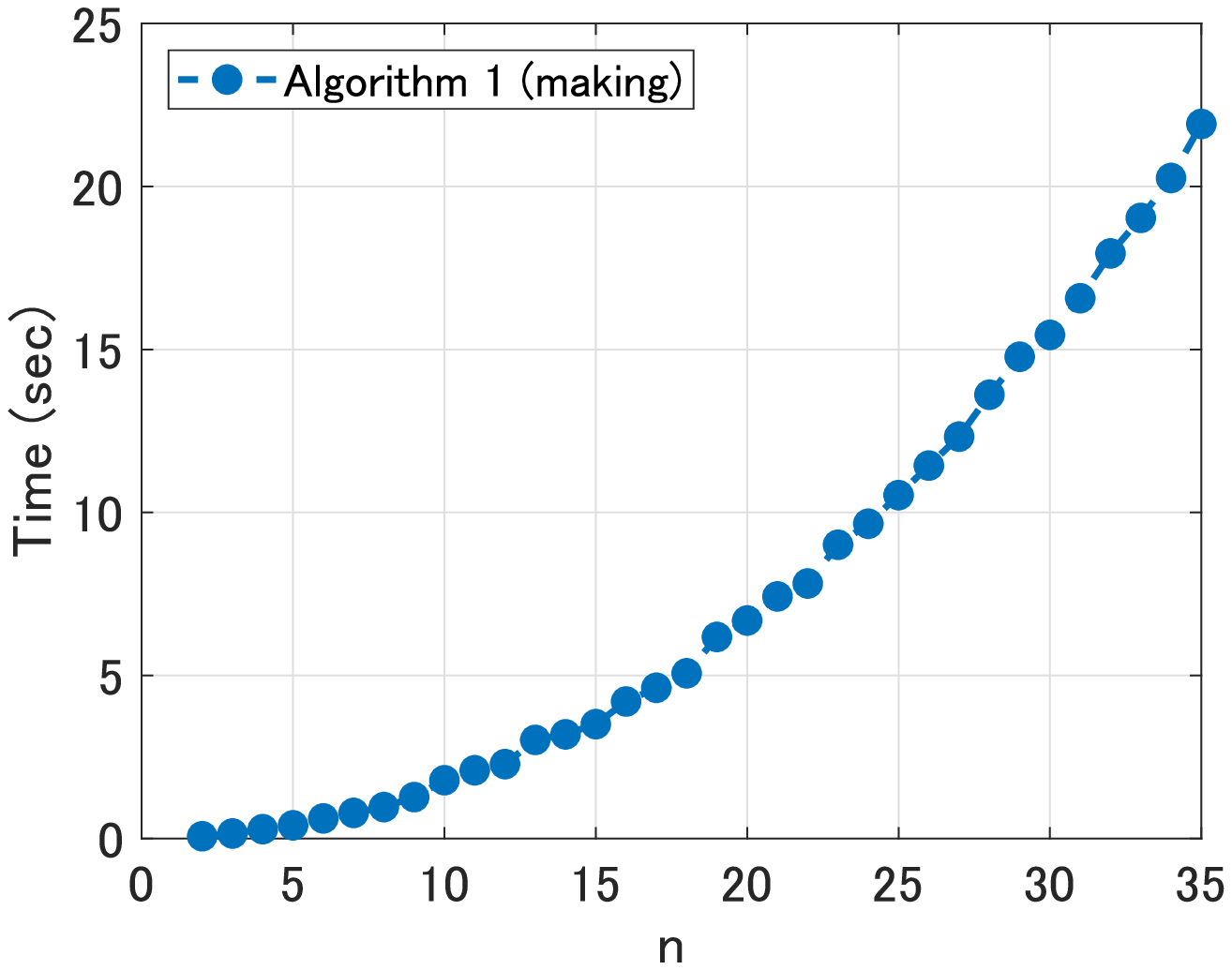}
\end{minipage}
\begin{minipage}[t]{0.48\linewidth}
\includegraphics[width = \linewidth]{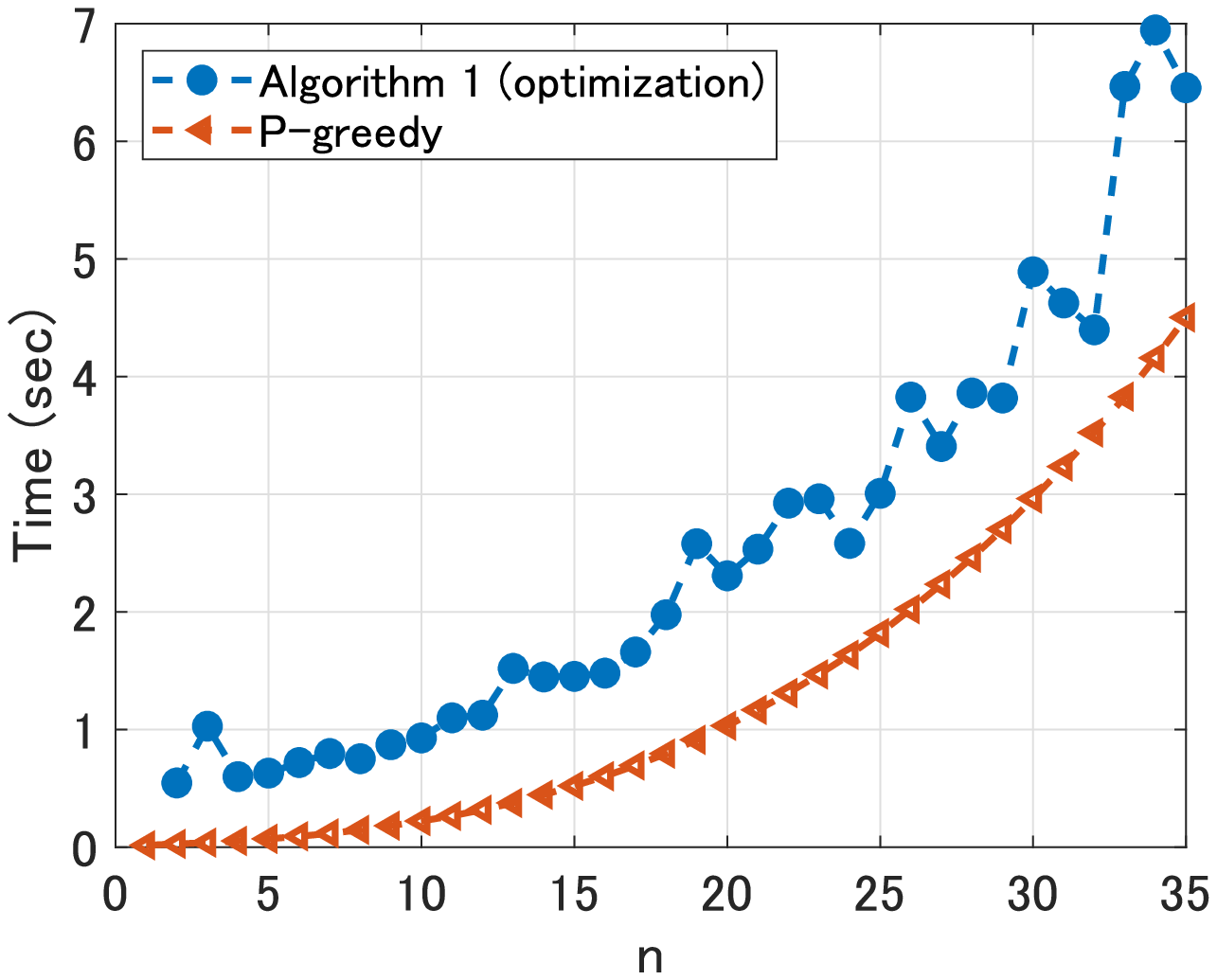}
\end{minipage}

\caption{Computation times for the spherical inverse multiquadric kernel on $S^{2}$ (Example~\ref{ex:kernel_Spherical}).
Left: the times for making the SOCP instances of Algorithm~\ref{alg:gen_points}, 
Right: the times for executing the SOCP optimizer and the $P$-greedy algorithm.}
\label{fig:sphere_time}
\end{figure}

\begin{figure}[t]

\centering
\begin{minipage}[t]{0.32\linewidth}
\includegraphics[width = \linewidth]{fig_points_Gauss2D_sqr_n35.eps}
\end{minipage}
\begin{minipage}[t]{0.32\linewidth}
\includegraphics[width = \linewidth]{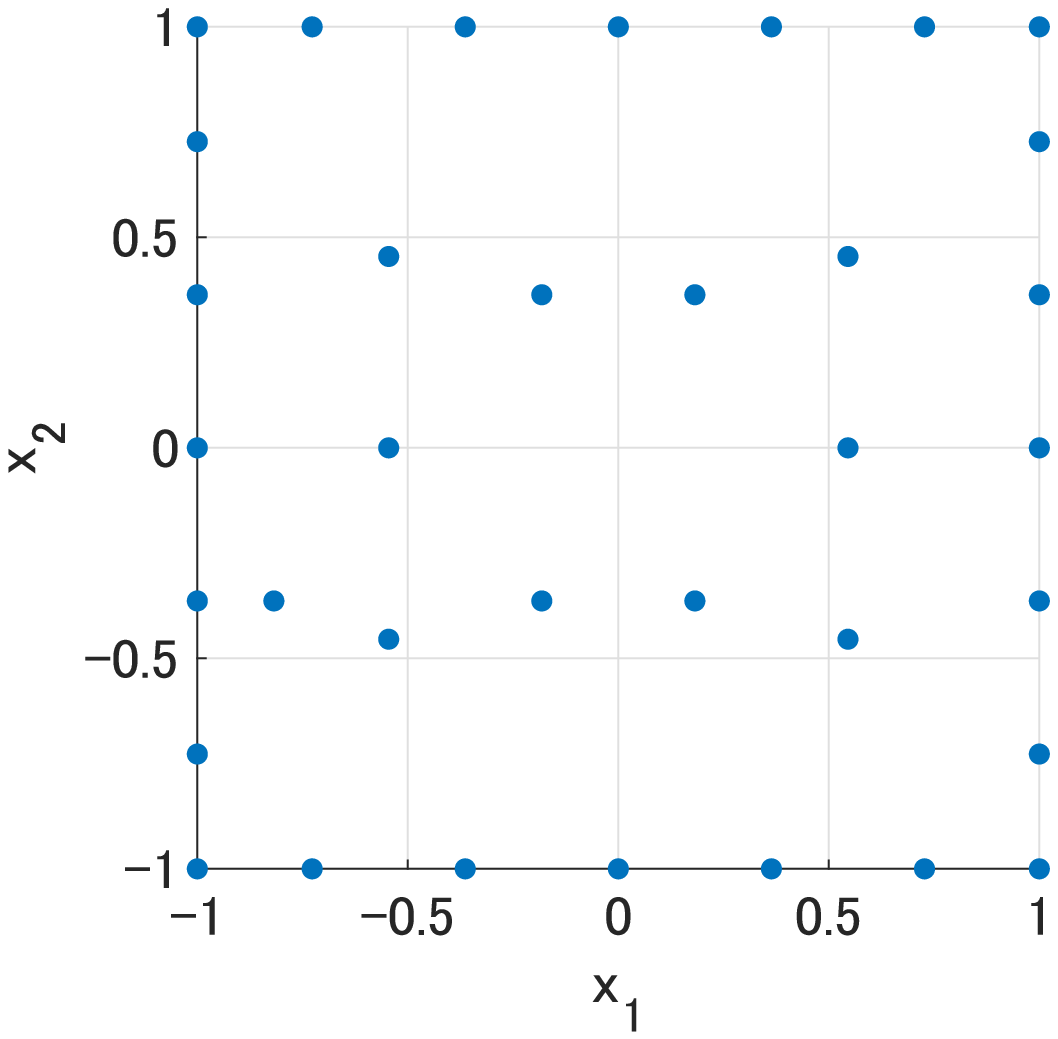}
\end{minipage}
\begin{minipage}[t]{0.32\linewidth}
\includegraphics[width = \linewidth]{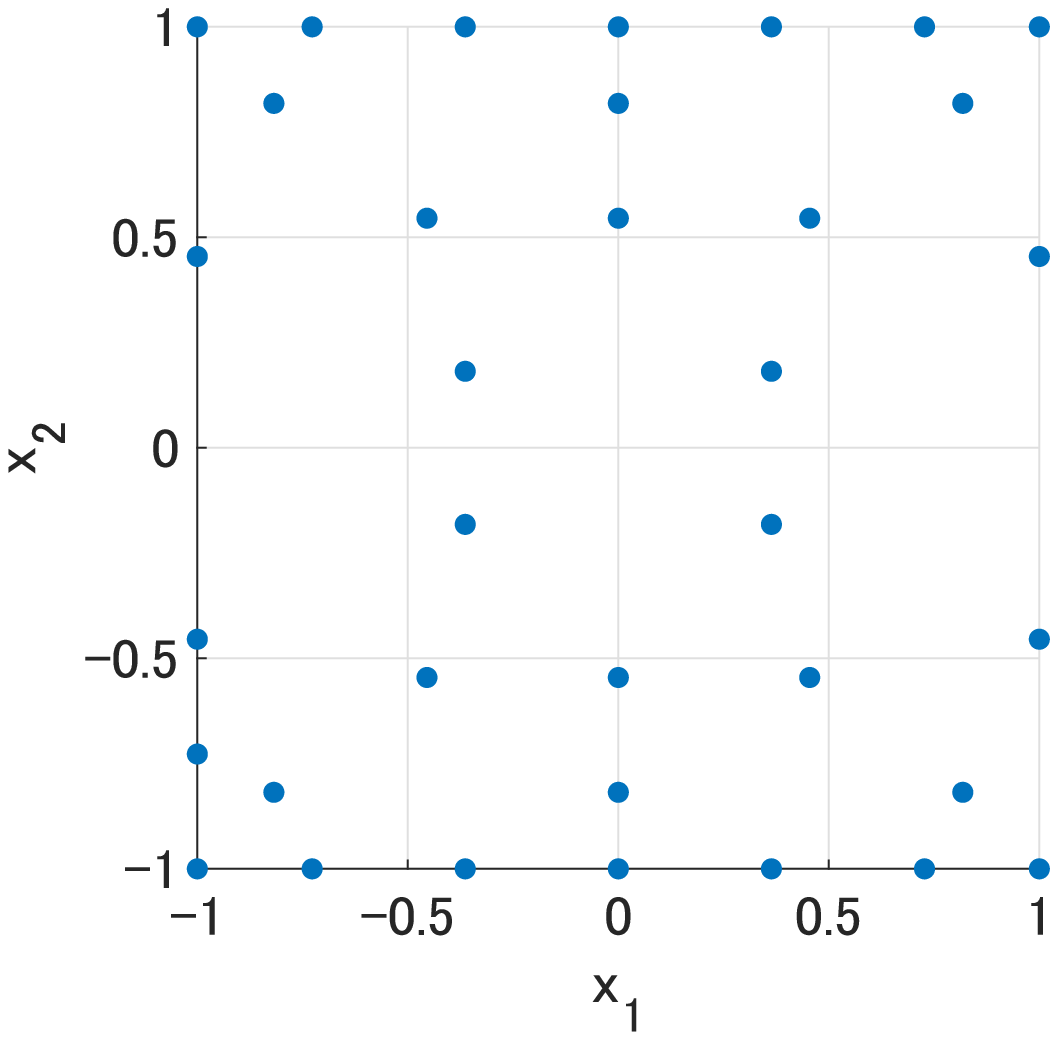}
\end{minipage}

\caption{Generated points on $[-1,1]^{2}$ for the Gaussian kernels (Example~\ref{ex:kernel_Gaussian}-2) in the case $n = 35$. 
Left: the same picture as the left one of Figure~\ref{fig:Gauss2D_n35} (a) (for comparison), 
Middle: the points for $\varepsilon = 1$, $\alpha = 2$, and the order of the eigenfunctions given by~\eqref{eq:Gauss_basis}, 
Right: the points for $\varepsilon = 1$, $\alpha = 1$, and the order of the eigenfunctions given by~\eqref{eq:Gauss_basis_changed}.}
\label{fig:Gauss2D_n35_another}

\end{figure}

\section{Concluding remarks}
\label{sec:concl}

We proposed Algorithms~\ref{alg:gen_points} and~\ref{alg:seq_gen_points}
generating point sets $\{ x_{1}, \ldots x_{n} \}$ 
by using the second order cone programming (SOCP) problem in~\eqref{eq:opt_expr_SOCP_pre_final}
for interpolation in reproducing kernel Hilbert spaces (RKHSs). 
Algorithm~\ref{alg:gen_points} is not a sequential algorithm 
in that it generates a point set by solving the SOCP problem at one sitting for a given number $n$. 
On the other hand,  
Algorithm~\ref{alg:seq_gen_points} is a sequential one 
that generates $n_{i}$ points in the $i$-th step 
for any sequence of positive integers $n_{1}, \ldots, n_{L}$. 
The SOCP problem is equivalent to 
the relaxed $D$-optimal experimental design problem in~\eqref{eq:equiv_appr_Dopt_relax}
derived from the maximization problem of 
the kernel matrix $\mathcal{K} = (K(x_{i}, x_{j}))_{ij}$, 
where $K$ is the kernel of the RKHS. 
Therefore, 
we can regard Algorithm~\ref{alg:gen_points} as an algorithm yielding approximate Fekete points. 
In the results of the numerical experiments, 
we observed that the proposed algorithms compete with the $P$-greedy algorithm in several cases, 
although they are a bit time-consuming and Algorithm~\ref{alg:gen_points} often yields worse results.
From the observation, we expect that the approximate Fekete points will 
also provide nearly optimal interpolation in RKHSs. 
To show this, 
further improvement of the algorithms and their theoretical analysis will be necessary.

\section*{Funding information}

K. Tanaka is supported by the grant-in-aid of 
Japan Society of the Promotion of Science with 
KAKENHI Grant Number 17K14241.

\section*{Acknowledgements}

The author thanks Takanori Maehara for his valuable comment 
about the Matlab programs used in the numerical experiments.
Thanks to the comment, 
much faster execution of the proposed algorithms has been realized 
than that in the initially submitted version of this article. 
Furthermore, 
the author gives thanks to the anonymous reviewers for their valuable suggestions about this article.



\appendix

\section{Proofs of the theorems for the SOC-representablity of the $D$-optimal design}
\label{sec:proofs_SOC}

\subsection{Proof of Theorem~\ref{thm:opt_expr_det}}

\begin{proof}
Let $\tilde{Q} \in \mathbf{R}^{m \times m}$ 
and $\tilde{R} \in \mathbf{R}^{m \times \ell}$
be the matrices that give the QR-decomposition $H^{T} = \tilde{Q} \tilde{R}$. 
They have the following partitions: 
\begin{align}
\tilde{Q} = 
\begin{bmatrix}
Q_{\ast} & \hat{Q}
\end{bmatrix}, 
\quad
\tilde{R} = 
\begin{bmatrix}
S_{\ast}^{T} \\
O^{T}
\end{bmatrix}, 
\end{align}
where $Q_{\ast} \in \mathbf{R}^{m \times \ell}$, $\hat{Q} \in \mathbf{R}^{m \times (m-\ell)}$, and 
$S_{\ast} \in \mathbf{R}^{\ell \times \ell}$ is lower triangular.  
Then, $(Q, S) = (Q_{\ast}, S_{\ast})$ is a feasible solution of Problem~\eqref{eq:opt_expr_det}. 
In fact, 
\begin{align}
HQ_{\ast} 
= 
\tilde{R}^{T} \tilde{Q}^{T} {Q}_{\ast}
= 
\begin{bmatrix}
S_{\ast} & O
\end{bmatrix}
\begin{bmatrix}
Q_{\ast}^{T} \\ \hat{Q}^{T}
\end{bmatrix}
{Q}_{\ast}
= 
\begin{bmatrix}
S_{\ast} & O
\end{bmatrix}
\begin{bmatrix}
I_{\ell} \\ O^{T}
\end{bmatrix}
= S_{\ast},
\end{align}
and $\| Q_{\ast} \mathbf{e}_{i} \| = 1$ because $\tilde{Q}$ is an orthogonal matrix.  
Furthermore, because
\begin{align}
HH^{T} = \tilde{R}^{T} \tilde{Q}^{T} \tilde{Q} \tilde{R} = \tilde{R}^{T} \tilde{R} = {S}_{\ast}^{T} {S}_{\ast}, 
\end{align}
we have $\det (HH^{T}) = \det({S}_{\ast}^{T} {S}_{\ast}) = (\det {S}_{\ast})^{2}$. 

Therefore, what remains is to show that $(\det S)^{2} \leq \det (HH^{T})$ 
for any feasible solution $(Q, S)$ of Problem~\eqref{eq:opt_expr_det}. 
This inequality holds true because we have
\begin{align}
(\det S)^{2} = (\det (HQ))^{2} \leq \det(HH^{T}) \det(Q^{T}Q) \leq \det(HH^{T}). 
\label{eq:det_ineqs}
\end{align}
The first inequality in~\eqref{eq:det_ineqs} is the Cauchy-Schwarz inequality for determinants {\cite[Exercise 12.15]{bib:AbadirMagnus_MatAlg}} (see Proposition~\ref{prop:CS_mat} in Appendix~\ref{sec:CS_mat}). 
Moreover, the second inequality in~\eqref{eq:det_ineqs} is derived by the Hadamard inequality {\cite[Exercise 12.26]{bib:AbadirMagnus_MatAlg}}
and the last constraint of Problem~\eqref{eq:opt_expr_det}
as follows:
\begin{align}
\det(Q^{T}Q) \leq \prod_{i=1}^{\ell} \left( \text{The } (i,i) \text{ component of } Q^{T}Q \right) = \prod_{i=1}^{\ell} \| Q \mathbf{e}_{i} \|^{2} \leq 1. 
\end{align}
\end{proof}

\subsection{Proof of Theorem~\ref{thm:SOCP_expr_det}}

\begin{proof}
We show that every feasible solution $(Z, T, G) = ((z_{ij}), (t_{ij}), (g_{ij}))$ of Problem~\eqref{eq:opt_expr_SOCP} 
yields a feasible solution $(Q, S) = ((q_{ij}), (s_{ij}))$ of Problem~\eqref{eq:opt_expr_det} in which $g_{jj} = s_{jj}^{2}$ for all $j = 1,\ldots, \ell$, 
and vice versa. 
If we do these, we have 
\(
\mathrm{OPT}_{2}(\{ a_{i} \}, \, w)
= \det J
= (\det S)^{2}
= \mathrm{OPT}_{1}(H)
\), 
from which the conclusion follows.

First, suppose that $(Z, T, G) = ((z_{ij}), (t_{ij}), (g_{ij}))$ is a feasible solution of Problem~\eqref{eq:opt_expr_SOCP}. 
Then, we define $(Q, S) = ((q_{ij}), (s_{ij}))$ by
\begin{align}
q_{ij} = 
\begin{cases}
\displaystyle \frac{z_{ij}}{\sqrt{w_{i}} \sqrt{g_{jj}}} & (w_{i} > 0, \ g_{jj} > 0), \\
0 & (\text{otherwise}), 
\end{cases}
\qquad 
s_{ij} = 
\begin{cases}
\displaystyle \frac{g_{ij}}{\sqrt{g_{jj}}} & (g_{jj} > 0), \\
0 & (\text{otherwise}).
\end{cases}
\notag
\end{align}
Then, clearly $S$ is lower triangular. 
For $j$ with $g_{jj} > 0$, we have
\begin{align}
(HQ) \, \mathbf{e}_{j} 
& = \sum_{i=1}^{m} \sqrt{w_{i}} \, q_{ij} \, a_{i} 
= \sum_{i: \, w_{i} > 0} \sqrt{w_{i}} \, q_{ij} \, a_{i} 
= \frac{1}{\sqrt{g_{jj}}} \sum_{i: \, w_{i} > 0} z_{ij} \, a_{i} \notag \\
& = \frac{1}{\sqrt{g_{jj}}} \sum_{i=1}^{m} z_{ij} \, a_{i} 
= \frac{1}{\sqrt{g_{jj}}} G \, \mathbf{e}_{j}
= S \, \mathbf{e}_{j}. \notag
\end{align}
For the fourth equality, 
we used the fact that $w_{i} = 0 \Rightarrow z_{ij} = 0$, 
which follows from the constraint $z_{ij}^{2} \leq t_{ij} w_{i}$ in Problem~\eqref{eq:opt_expr_SOCP}. 
For $j$ with $g_{jj} = 0$, we have $(HQ) \, \mathbf{e}_{j} = \boldsymbol{0} = S \, \mathbf{e}_{j}$. 
Hence $HQ = S$ holds true. 
Furthermore, for $j$ with $g_{jj} > 0$, we have
\begin{align}
\| Q \, \mathbf{e}_{j} \|^{2} 
= \sum_{i=1}^{m} q_{ij}^{2} 
= \sum_{i: \, w_{i} > 0} \frac{1}{g_{jj}} \frac{z_{ij}^{2}}{w_{i}} 
\leq \frac{1}{g_{jj}} \sum_{i: \, w_{i} > 0} t_{ij} \leq 1,
\notag
\end{align}
and for $j$ with $g_{jj} > 0$, we have $\| Q \, \mathbf{e}_{j} \|^{2} = 0 \leq 1$. 
Consequently, $(Q, S)$ is a feasible solution of Problem~\eqref{eq:opt_expr_det} with $g_{jj} = s_{jj}^{2}$. 

Next, suppose that $(Q, S) = ((q_{ij}), (s_{ij}))$ is a feasible solution of Problem~\eqref{eq:opt_expr_det}. 
Then, we define $(Z, T, G) = ((z_{ij}), (t_{ij}), (g_{ij}))$ by
\begin{align}
z_{ij} = \sqrt{w_{i}} \, s_{jj} \, q_{ij}, \qquad
t_{ij} = s_{jj}^{2} \, q_{ij}^{2}, \qquad
g_{ij} = s_{jj} \, s_{ij}.
\notag 
\end{align}
Then, clearly $G$ is lower triangular. 
Furthermore, we have
\begin{align}
(a_{1}, \ldots, a_{m}) \, Z \, \mathbf{e}_{j}
= \sum_{i=1}^{m} a_{i} \, z_{ij} 
= s_{jj} \sum_{i=1}^{m} \sqrt{w_{i}} \, a_{i} \, q_{ij} 
= s_{jj} HQ \, \mathbf{e}_{j} 
= s_{jj} S \, \mathbf{e}_{j} 
= G \, \mathbf{e}_{j}, 
\notag
\end{align}
\(
z_{ij}^{2} = {w_{i}} \, s_{jj}^{2} \, q_{ij}^{2} = {w_{i}} \, t_{ij}
\),
and 
\begin{align}
\sum_{i=1}^{m} t_{ij} = s_{jj}^{2} \sum_{i=1}^{m} q_{ij}^{2} \leq s_{jj}^{2} = g_{jj}. 
\notag
\end{align}
Therefore, 
$(Z, T, G)$ is a feasible solution of Problem~\eqref{eq:opt_expr_SOCP} with $s_{jj}^{2} = g_{jj}$.  
\end{proof}

\subsection{Proof of Theorem~\ref{thm:prod_SOCP}}

\begin{proof}
We show the assertion only in the case that $\ell$ is even, because the other case can be shown similarly. 
If $g_{j} = 0$ for some $j$, we can deduce that $u_{1} = 0$. 
Therefore, it suffices to consider the case that $g_{j} > 0$ for $j = 1,\ldots, \ell$. 

First, we show that $u_{1} \leq \left(\prod_{j=1}^{\ell} g_{j} \right)^{1/\ell}$ for any feasible solution $\{ u_{i} \}$. 
If $u_{1} = 0$, the conclusion is trivial. 
Then, we assume that $u_{1} > 0$. 
From the constraints with $1 \leq i \leq 2^{p-1}-1$ in~\eqref{eq:product_SOCP_even}, we have
\begin{gather*}
u_{1}^{2^{p}} \leq u_{2}^{2^{p-1}} u_{3}^{2^{p-1}}, \\
u_{2}^{2^{p-1}} \leq u_{4}^{2^{p-2}}  u_{5}^{2^{p-2}} , \quad 
u_{3}^{2^{p-1}} \leq u_{6}^{2^{p-2}}  u_{7}^{2^{p-2}} , \\
\iddots \qquad \qquad \qquad \qquad \vdots \qquad \qquad \qquad \qquad \ddots \\
u_{2^{p-2}}^{2^{2}} \leq u_{2^{p-1}}^{2}  u_{2^{p-1}+1}^{2} , \quad 
u_{2^{p-2}+1}^{2^{2}} \leq u_{2^{p-1}+2}^{2}  u_{2^{p-1}+3}^{2} , \quad \ldots \quad 
u_{2^{p-1} - 1}^{2^{2}} \leq u_{2^{p}-2}^{2}  u_{2^{p}-1}^{2}. 
\end{gather*}
Then, taking the product of all these inequalities, we have 
\begin{align}
u_{1}^{2^{p}} \leq \prod_{i=2^{p-1}}^{2^{p}-1} u_{i}^{2}. 
\label{eq:prod_SOCP_key1}
\end{align}
Similarly, taking the product of all the constraints with $2^{p-1} \leq i \leq 2^{p} - 1$, we have
\begin{align}
\prod_{i=2^{p-1}}^{2^{p}-1} u_{i}^{2} \leq \left(\prod_{j=1}^{\ell} g_{j} \right) u_{1}^{2^{p} - \ell}.
\label{eq:prod_SOCP_key2}
\end{align}
Therefore, it follows from \eqref{eq:prod_SOCP_key1} and \eqref{eq:prod_SOCP_key2} that 
\begin{align}
u_{1}^{2^{p}} \leq \left(\prod_{j=1}^{\ell} g_{j} \right) u_{1}^{2^{p} - \ell}
\iff
u_{1} \leq \left(\prod_{j=1}^{\ell} g_{j} \right)^{1/\ell}.
\notag 
\end{align}

Next, we show that there exists a feasible solution $\{ u_{i} \}$ with 
\(
u_{1} = \left(\prod_{j=1}^{\ell} g_{j} \right)^{1/\ell}.
\)
Set $u_{1}$ by this equality and set the other variables by 
\begin{align}
\begin{array}{ll}
 u_{i} = u_{1} & (i=2^{p-1} + \ell/2, \ldots, 2^{p}-1), \\ 
 u_{i} = \sqrt{g_{2i - 2^{p} + 1} g_{2i - 2^{p} + 2}} & (i=2^{p-1}, \ldots, 2^{p-1} + \ell/2 - 1), \\
 u_{i} = \sqrt{u_{2i} u_{2i+1}} & (i=2, \ldots, 2^{p-1}-1). 
\end{array}
\notag
\end{align}
Then, in a similar manner to that for deriving \eqref{eq:prod_SOCP_key1} and \eqref{eq:prod_SOCP_key2}, 
we have
\begin{align}
u_{2}^{2^{p-1}} u_{3}^{2^{p-1}} 
= 
\left(\prod_{j=1}^{\ell} g_{j} \right) u_{1}^{2^{p} - \ell}
= 
u_{1}^{\ell} u_{1}^{2^{p} - \ell}
=
u_{1}^{2^{p}},
\end{align}
which guarantees the solution $\{ u_{i} \}$ is feasible. 
\end{proof}

\section{Cauchy-Schwarz inequality for determinants}
\label{sec:CS_mat}

\begin{lem}[{\cite[Exercise 12.7]{bib:AbadirMagnus_MatAlg}}]
\label{lem:det_sum_ineq}
Let $\ell$ be a non-negative integer and
let $C, D \in \mathbf{R}^{\ell \times \ell}$ be $\ell \times \ell$ 
positive semi-definite matrices.
Then, the following inequality holds true: 
\begin{align}
\det(C+D) \geq \det C + \det D. \notag
\end{align}
\end{lem}

\begin{proof}
If $C$ and $D$ are singular, the inequality is trivial. 
Therefore, we assume that $C$ is non-singular without loss of generality.
Let $S: = C^{-1/2} D C^{-1/2}$. 
Then, because
\begin{align}
\det(C+D) - \det C - \det D
& =
\det(C^{1/2} (I + S) C^{1/2}) - \det C - \det (C^{1/2} S C^{1/2}) \notag \\
& = 
\det C \, (\det(I + S) - 1 - \det S), 
\notag
\end{align}
it suffices to show that $\det(I + S) - 1 - \det S \geq 0$ for any $S \in \mathbf{R}^{\ell \times \ell}$ with $S \succeq O$. 
By letting $\lambda_{1}, \ldots, \lambda_{\ell} \geq 0$ be the eigenvalues of $S$, this inequality is reduced to
\begin{align}
\prod_{i=1}^{\ell} (1 + \lambda_{i}) - 1 - \prod_{i=1}^{\ell} \lambda_{i} \geq 0, 
\notag
\end{align}
which can be proved by expanding the product $\prod_{i=1}^{\ell} (1 + \lambda_{i})$. 
\end{proof}

\begin{prop}[Cauchy-Schwarz inequality for determinants {\cite[Exercise 12.15]{bib:AbadirMagnus_MatAlg}}]
\label{prop:CS_mat}
Let $\ell$ and $m$ be non-negative integers with $\ell \leq m$ and
let $A, B \in \mathbf{R}^{m \times \ell}$ be $m \times \ell$ matrices. 
Then, the following inequality holds true:
\begin{align}
(\det (A^{T} B))^{2} \leq \det (A^{T} A) \, \det (B^{T} B). 
\label{eq:CS_ineq_mat}
\end{align}
\end{prop}

\begin{proof}
If $\det (A^{T} B) = 0$, the inequality is trivial. 
Therefore, we assume that $\det (A^{T} B) \neq 0$. 
Then, $A$ and $B$ have full column rank, 
which implies that $A^{T} A$ and $B^{T} B$ are non-singular. 
Let $C, D \in \mathbf{R}^{\ell \times \ell}$ be given by 
\begin{align}
C &:= B^{T} A \, (A^{T} A)^{-1} A^{T} B, \notag \\
D &:= B^{T} B - C = B^{T} (I - A \, (A^{T} A)^{-1} A^{T}) B. \notag
\end{align}
Then, clearly $C \succ O$ holds true. 
Furthermore, because $P := I - A \, (A^{T} A)^{-1} A^{T}$
satisfies $P^{2} = P$, we have $P \succeq O$ and hence $D \succeq O$. 
Then, using Lemma~\ref{lem:det_sum_ineq}, we have 
\begin{align}
\det (B^{T} B) = \det(C+D) \geq \det C + \det D \geq \det C 
= (\det (A^{T} B))^{2} \, (\det A^{T} A)^{-1},
\notag
\end{align}
which is equivalent to~\eqref{eq:CS_ineq_mat}. 
\end{proof}

\section{Remarks on implementation of SOCP \eqref{eq:opt_expr_SOCP_pre_final} by MOSEK}
\label{sec:rem_impl_SOCP}

\subsection{MOSEK}

As a software to solve SOCP problems, 
we use the MOSEK Optimization Toolbox for MATLAB  8.1.0.56, 
which is provided by MOSEK ApS in Denmark (\url{https://www.mosek.com/}, last accessed on 19 October 2018).

\subsubsection{Expression of rotated second order cones}

Recall that $p = \lceil \log_{2} \ell \rceil$. 
In SOCP~\eqref{eq:opt_expr_SOCP_pre_final} in which $\ell$ is even, 
we need to consider the rotated second order cone constraints 
\begin{align}
& z_{ij}^{2} \leq t_{ij} w_{i} \qquad (i=1,\ldots,m, j=1,\ldots, \ell), \label{eq:orig_cone_even_1} \\
& u_{i}^{2} \leq u_{2i} u_{2i+1} \qquad (i=1,\ldots, 2^{p-1}-1), \label{eq:orig_cone_even_2} \\
& u_{i}^{2} \leq \tilde{g}_{2i-2^{p}+1} \tilde{g}_{2i-2^{p}+2} \qquad (i = 2^{p-1}, \ldots, 2^{p-1} + \ell/2 - 1). \label{eq:orig_cone_even_3} 
\end{align}
Furthermore, 
in the counterpart in which $\ell$ is odd, 
we need to consider the constraints 
\begin{align}
& z_{ij}^{2} \leq t_{ij} w_{i} \qquad (i=1,\ldots,m, j=1,\ldots, \ell), \label{eq:orig_cone_odd_1} \\
& u_{i}^{2} \leq u_{2i} u_{2i+1} \qquad (i=1,\ldots, 2^{p-1}-1), \label{eq:orig_cone_odd_2} \\
& u_{i}^{2} \leq \tilde{g}_{2i-2^{p}+1} \tilde{g}_{2i-2^{p}+2} \qquad (i = 2^{p-1}, \ldots, 2^{p-1} + (\ell-3)/2), \label{eq:orig_cone_odd_3} \\
& u_{i}^{2} \leq \tilde{g}_{2i-2^{p}+1} u_{1} \qquad (i = 2^{p-1} + (\ell-1)/2). \label{eq:orig_cone_odd_4} 
\end{align}

However, MOSEK permits only the expression $\eta_{1}^{2} + \cdots + \eta_{N}^{2} \leq 2 \xi \zeta$ for a rotated second order cone, 
where we just need the case that $N=1$. 
Therefore, we employ the following variable transformations: 
\begin{align}
& w_{i} = 2 \hat{w}_{i} \quad (i = 1,\ldots, m), \\
& \tilde{g}_{i} = \sqrt{2} \hat{g}_{i} \quad (i = 1, \ldots, \ell-1), \\
& \tilde{g}_{\ell} = 
\begin{cases}
\sqrt{2} \hat{g}_{\ell} & (\ell \text{ is even}), \\
2 \hat{g}_{\ell} & (\ell \text{ is odd}). \\ 
\end{cases} 
\end{align}
We do not change $u_{i}$ because of the following reason. 
If we change the constraints $u_{i}^{2} \leq u_{2i} u_{2i+1}$ to $u_{i}^{2} \leq 2 u_{2i} u_{2i+1}$, 
they just change the optimal value by $2^{(p-1)2^{p-1}}$ times%
\footnote{Therefore the optimal value of SOCP~\eqref{eq:opt_expr_SOCP_pre_final} is 
$2^{(p-1)2^{p-1}} (\text{The optimal value of Problem~\eqref{eq:gen_appr_Dopt_relax}})^{1/\ell}$.}. 
Consequently, 
we deal with the SOCP problem
\begin{align}
\begin{array}{cll}
\displaystyle
\mathop{\text{maximize}}_{
\begin{subarray}{l}
\{u_{1}, \ldots, u_{2^{p}-1} \} \subset \mathbf{R}, \\
\{ \hat{g}_{1}, \ldots, \hat{g}_{\ell} \} \subset \mathbf{R}, \\ 
Z = (z_{ij}) \in \mathbf{R}^{m \times \ell}, \\ 
T = (t_{ij}) \in \mathbf{R}^{m \times \ell}, \\ 
\{\hat{w}_{1}, \ldots, \hat{w}_{m} \} \subset \mathbf{R} 
\end{subarray}} & 
u_{1} & \\
\text{subject to} & (a_{1}, \ldots, a_{m}) \, Z = 
\begin{bmatrix}
\sqrt{2} \hat{g}_{1} & 0 & \cdots & 0 \\
\ast & \sqrt{2} \hat{g}_{2} & \ddots & \vdots \\
\ast & \ast & \ddots & 0 \\
\ast & \ast & \ast & \sqrt{2} \hat{g}_{\ell} \\ 
\end{bmatrix}, & \\
 & \hat{w}_{1} + \cdots + \hat{w}_{m} = n/2, & \\
 & \displaystyle \sum_{i=1}^{m} t_{ij} \leq \sqrt{2} \hat{g}_{j} &  (j=1,\ldots, \ell), \\
 & u_{i}^{2} \leq u_{1}^{2} & (i=2^{p-1} + \ell/2, \ldots, 2^{p}-1), \\
 \hline
 & u_{i} \geq 0 & (i = 1,\ldots, 2^{p}-1), \\
 & 0 \leq \hat{w}_{j} \leq 1/2 & (j = 1,\ldots, m), \\
 \hline
 & u_{i}^{2} \leq 2 u_{2i} u_{2i+1} & (i=1, \ldots, 2^{p-1}-1), \\
 & u_{i}^{2} \leq 2 \hat{g}_{2i - 2^{p} + 1} \hat{g}_{2i - 2^{p} + 2} & (i=2^{p-1}, \ldots, 2^{p-1} + \ell/2 - 1), \\
 & z_{ij}^{2} \leq 2 t_{ij} \hat{w}_{i} & (i = 1,\ldots, m, \, j=1,\ldots, \ell), 
\end{array}
\label{eq:opt_expr_SOCP_even}
\end{align}
in the case that $\ell$ is even, and 
\begin{align}
\begin{array}{cll}
\displaystyle
\mathop{\text{maximize}}_{
\begin{subarray}{l}
\{u_{1}, \ldots, u_{2^{p}-1} \} \subset \mathbf{R}, \\
\{ \hat{g}_{1}, \ldots, \hat{g}_{\ell} \} \subset \mathbf{R}, \\ 
Z = (z_{ij}) \in \mathbf{R}^{m \times \ell}, \\ 
T = (t_{ij}) \in \mathbf{R}^{m \times \ell}, \\ 
\{\hat{w}_{1}, \ldots, \hat{w}_{m} \} \subset \mathbf{R} 
\end{subarray}} & 
u_{1} & \\
\text{subject to} & (a_{1}, \ldots, a_{m}) \, Z = 
\begin{bmatrix}
\sqrt{2} \hat{g}_{1} & 0 & \cdots & 0 \\
\ast & \sqrt{2} \hat{g}_{2} & \ddots & \vdots \\
\ast & \ast & \ddots & 0 \\
\ast & \ast & \ast & 2 \hat{g}_{\ell} \\ 
\end{bmatrix}, & \\
 & \hat{w}_{1} + \cdots + \hat{w}_{m} = n/2, & \\
 & \displaystyle \sum_{i=1}^{m} t_{ij} \leq \sqrt{2} \hat{g}_{j} &  (j=1,\ldots, \ell-1), \\
 & \displaystyle \sum_{i=1}^{m} t_{i\ell} \leq 2 \hat{g}_{\ell},  &  \\
 & u_{i}^{2} \leq u_{1}^{2} & (i=2^{p-1} + (\ell+1)/2, \ldots, 2^{p}-1), \\
 \hline
 & u_{i} \geq 0 & (i = 1,\ldots, 2^{p}-1), \\
 & 0 \leq \hat{w}_{j} \leq 1/2 & (j = 1,\ldots, m), \\
 \hline
 & u_{i}^{2} \leq 2 u_{2i} u_{2i+1} & (i=1, \ldots, 2^{p-1}-1), \\
 & u_{i}^{2} \leq 2 \hat{g}_{2i - 2^{p} + 1} \hat{g}_{2i - 2^{p} + 2} & (i=2^{p-1}, \ldots, 2^{p-1} + (\ell-3)/2), \\
 & u_{i}^{2} \leq 2 \hat{g}_{2i - 2^{p} + 1} u_{1} & (i=2^{p-1} + (\ell-1)/2), \\
 & z_{ij}^{2} \leq 2 t_{ij} \hat{w}_{i} & (i = 1,\ldots, m, \, j=1,\ldots, \ell), 
\end{array}
\label{eq:opt_expr_SOCP_odd}
\end{align}
in the case that $\ell$ is odd. 

\subsubsection{Restriction about cone constraints}

MOSEK does not permit that a variable belongs to plural cones. 
However, we need to consider the cone constraints in~\eqref{eq:opt_expr_SOCP_even} and~\eqref{eq:opt_expr_SOCP_odd}
in which 
each $u_{i}$ appears $2$ times and
each $\hat{w}_{i}$ appears $\ell$ times. 
Therefore we prepare the copies of these variables:
\begin{align*}
u_{i} \ \leftrightarrow \ u_{i1}, u_{i2}, 
\qquad
\hat{w}_{i} \ \leftrightarrow \ \hat{w}_{i1}, \hat{w}_{i2}, \ldots, \hat{w}_{i \ell}, 
\end{align*}
and set the linear constraints between these variables: 
\begin{align}
& u_{i1} = u_{i2}, \label{eq:u_equalities} \\
& \hat{w}_{i1} = \hat{w}_{i2} = \cdots = \hat{w}_{i \ell}. \label{eq:w_equalities}
\end{align}
Then, we introduce a vector 
\begin{align}
v = 
(& u_{11}, u_{12}, \ u_{21}, u_{22}, \ \ldots, \ u_{(2^{p}-1), 1}, u_{(2^{p}-1), 2}; \notag \\
& \hat{g}_{1}, \hat{g}_{2}, \ldots, \hat{g}_{\ell}; \notag \\
& z_{11}, \ldots, z_{m1}, \ z_{12}, \ldots, z_{m2}, \ \ldots, \ z_{1\ell}, \ldots, z_{m\ell}; \notag \\ 
& t_{11}, \ldots, t_{m1}, \ t_{12}, \ldots, t_{m2}, \ \ldots, \ t_{1\ell}, \ldots, t_{m\ell}; \notag \\ 
& \hat{w}_{11}, \ldots, \hat{w}_{m1}, \ \hat{w}_{12}, \ldots, \hat{w}_{m2}, \ \ldots, \ \hat{w}_{1\ell}, \ldots, \hat{w}_{m\ell})^{T}. 
\end{align}

\subsection{Linear constraints}

To describe the linear constraints in~\eqref{eq:opt_expr_SOCP_even} and~\eqref{eq:opt_expr_SOCP_odd} in terms of the vector $v$, 
we provide a matrix $A$ and vectors $b_{\mathrm{l}}$ and $b_{\mathrm{u}}$ 
that express the constraints as $b_{\mathrm{l}} \leq Av \leq b_{\mathrm{u}}$. 

We begin with the case that $\ell$ is even. 
First, we express the constraints in~\eqref{eq:u_equalities} and~\eqref{eq:w_equalities} by
\begin{align}
\underbrace{
\begin{bmatrix}
1 & -1 & 0 & 0 & \cdots &  0 & 0 & 0_{\ell + 3m\ell}^{T} \\
0 & 0 & 1 & -1 & \cdots &  0 & 0 & 0_{\ell + 3m\ell}^{T} \\
\vdots & & & & \ddots & & \vdots & \vdots \\
0 & 0 & 0 & 0 & \cdots & 1 & -1 & 0_{\ell + 3m\ell}^{T} \\
\end{bmatrix}}_{A_{1}}
v = 0_{2^{p} - 1}, 
\end{align}
and 
\begin{align}
\underbrace{
\begin{bmatrix}
O_{m, 2(2^{p}-1) + \ell + 2m\ell} & I_{m} & -I_{m} & O_{m,m} & \cdots &  O_{m,m} & \\
O_{m, 2(2^{p}-1) + \ell + 2m\ell} & O_{m,m} & I_{m} & -I_{m} & & O_{m,m} & \\
\vdots & & & \ddots & & & \\
O_{m, 2(2^{p}-1) + \ell + 2m\ell} & O_{m,m} & \cdots & O_{m,m} & I_{m} & -I_{m} & \\
\end{bmatrix}}_{A_{2}}
v = 0_{m (\ell-1)}, 
\end{align}
respectively. 
Next, we consider the constraint
\begin{align*}
(a_{1}, \ldots, a_{m}) \, Z = 
\begin{bmatrix}
\sqrt{2} \hat{g}_{1} & 0 & \cdots & 0 \\
\ast & \sqrt{2} \hat{g}_{2} & \ddots & \vdots \\
\ast & \ast & \ddots & 0 \\
\ast & \ast & \ast & \sqrt{2} \hat{g}_{\ell} \\ 
\end{bmatrix}.
\end{align*}
By using $v$ and the expression 
\begin{align*}
(a_{1}, \ldots, a_{m}) = 
\begin{bmatrix}
\boldsymbol{\alpha}_{1} &
\cdots & 
\boldsymbol{\alpha}_{\ell} 
\end{bmatrix}^{T} \qquad 
(\boldsymbol{\alpha}_{j} \in \mathbf{R}^{m}), 
\end{align*}
we express the linear constraint as follows: 
\begin{align}
\underbrace{
\left[
\begin{array}{c|cccccc|cccc|c}
{0}_{2(2^{p}-1)}^{T} & -\sqrt{2} & 0 & 0 & \cdots & 0 & 0 & \boldsymbol{\alpha}_{1}^{T} & {0}_{m}^{T} & \cdots & {0}_{m}^{T} & {0}_{2m\ell}^{T}  \\
{0}_{2(2^{p}-1)}^{T} & 0 & 0 & 0 & \cdots & 0 & 0 & {0}_{m}^{T} & \boldsymbol{\alpha}_{1}^{T}  & & {0}_{m}^{T} & {0}_{2m\ell}^{T} \\
\vdots &  &  & \vdots & & \vdots & \vdots & &  & \ddots & \vdots & \vdots \\
{0}_{2(2^{p}-1)}^{T} & 0 & 0 & 0 & \cdots & 0 & 0 & {0}_{m}^{T} & \cdots & {0}_{m}^{T} & \boldsymbol{\alpha}_{1}^{T} & {0}_{2m\ell}^{T} \\
\hline
{0}_{2(2^{p}-1)}^{T} & 0 & -\sqrt{2} & 0 & \cdots & 0 & 0 & {0}_{m}^{T} & \boldsymbol{\alpha}_{2}^{T}  & & {0}_{m}^{T} & {0}_{2m\ell}^{T} \\
\vdots &  &  & \vdots & &  \vdots & \vdots & &  & \ddots & \vdots & \vdots \\
{0}_{2(2^{p}-1)}^{T} & 0 & 0 & 0 & \cdots & 0 & 0 & {0}_{m}^{T} & \cdots & {0}_{m}^{T} & \boldsymbol{\alpha}_{2}^{T} & {0}_{2m\ell}^{T} \\
\hline
\vdots &  &  & \vdots & &  \vdots & \vdots & & &  & \vdots & \vdots \\
\hline
{0}_{2(2^{p}-1)}^{T} & 0 & 0 & 0 & \cdots & 0 & -\sqrt{2} & {0}_{m}^{T} & \cdots & {0}_{m}^{T} & \boldsymbol{\alpha}_{\ell}^{T} & {0}_{2m\ell}^{T} \\
\end{array}
\right]}_{A_{3}}
v = 0_{\ell (\ell+1)/2}.
\end{align}
Next, we express the constraint $\hat{w}_{1\ell} + \cdots + \hat{w}_{m\ell} = n/2$ by
\begin{align}
\underbrace{
\begin{bmatrix}
0_{2(2^{p}-1) + \ell + 3m(\ell - 1)}^{T} & 1_{m}^{T}
\end{bmatrix}}_{A_{4}}
v = \frac{n}{2}.
\end{align}
Next, we express the inequality constraints 
\begin{align}
\sum_{i=1}^{m} t_{ij} \leq \sqrt{2} \hat{g}_{j} \quad (j = 1,\ldots , \ell)
\notag
\end{align}
by 
\begin{align}
\underbrace{
\left[
\begin{array}{c|cccc|c|cccc|c}
0_{2(2^{p}-1)}^{T} & -\sqrt{2} & 0 & \cdots & 0 & 0_{m\ell}^{T} & 1_{m}^{T} & 0_{m}^{T} & \cdots & 0_{m}^{T} & 0_{m\ell}^{T} \\
0_{2(2^{p}-1)}^{T} & 0 & -\sqrt{2} & \cdots & 0 & 0_{m\ell}^{T} & 0_{m}^{T} & 1_{m}^{T} & \cdots & 0_{m}^{T} & 0_{m\ell}^{T} \\
\vdots & \vdots & & \ddots & \vdots & \vdots & \vdots & & \ddots & \vdots & \vdots \\
0_{2(2^{p}-1)}^{T} & 0 & \cdots & 0 & -\sqrt{2} & 0_{m\ell}^{T} & 0_{m}^{T} & \cdots & 0_{m}^{T} & 1_{m}^{T} & 0_{m\ell}^{T} \\
\end{array}
\right]}_{A_{5}}
v \leq 0_{\ell}.
\end{align}
Finally, we consider the constraints 
\begin{align}
u_{i}^{2} \leq u_{1}^{2} \qquad (i = 2^{p-1} + \ell/2, \ldots, 2^{p}-1),
\end{align}
which are equivalent to the linear constraints $u_{i} \leq u_{1}$ owing to the non-negative constraints $u_{i} \geq 0$. 
These constraints are expressed by
\begin{align}
\underbrace{
\left[
\begin{array}{cc|c|ccccccc|c}
-1 & 0 & 0_{2^{p} + \ell - 4}^{T} & 1 & 0 & 0 & 0 & \cdots & 0 & 0 & 0_{\ell + 3m\ell}^{T} \\
-1 & 0 & 0_{2^{p} + \ell - 4}^{T} & 0 & 0 & 1 & 0 & \cdots & 0 & 0 & 0_{\ell + 3m\ell}^{T} \\ 
\vdots & \vdots & \vdots & \vdots & \vdots & & & \ddots & \vdots & \vdots & \vdots \\ 
-1 & 0 & 0_{2^{p} + \ell - 4}^{T} & 0 & 0 & 0 & 0 & \cdots & 1 & 0 & 0_{\ell + 3m\ell}^{T} \\
\end{array}
\right]}_{A_{6}}
v \leq 0_{2^{p-1}-\ell/2}. 
\end{align}
Consequently, by setting
\begin{align}
A & = 
\begin{bmatrix}
A_{1}^{T} & A_{2}^{T} & A_{3}^{T} & A_{4}^{T} & A_{5}^{T} & A_{6}^{T}
\end{bmatrix}^{T}, \\
b_{\mathrm{l}} 
& = 
\left( 
0_{2^{p} - 1}^{T}, \ 0_{m (\ell-1)}^{T}, \ 0_{\ell (\ell+1)/2}^{T}, \ n/2, \ -\infty \cdot 1_{\ell}^{T} , \ -\infty \cdot 1_{2^{p-1}-\ell/2}
\right)^{T}, \\
b_{\mathrm{u}} 
& = 
\left( 
0_{2^{p} - 1}^{T}, \ 0_{m (\ell-1)}^{T}, \ 0_{\ell (\ell+1)/2}^{T}, \  n/2, \ 0_{\ell}^{T} , \ 0_{2^{p-1}-\ell/2}^{T}
\right)^{T}, 
\end{align}
we can express the linear constraints as $b_{\mathrm{l}} \leq A v \leq  b_{\mathrm{u}}$.

In the case that $\ell$ is odd, 
we need to modify $A_{3}$, $A_{5}$, $A_{6}$, $b_{\mathrm{l}}$, and $b_{\mathrm{u}}$. 
We replace the component ``$-\sqrt{2}$'' in the last rows of $A_{3}$ and $A_{5}$ by ``$-2$''. 
Furthermore, 
We modify $A_{6}$, $b_{\mathrm{l}}$, and $b_{\mathrm{u}}$ as 
\begin{align}
A_{6} & =
\left.
\left[
\begin{array}{cc|c|ccccccc|c}
-1 & 0 & 0_{2^{p} + \ell - 3}^{T} & 1 & 0 & 0 & 0 & \cdots & 0 & 0 & 0_{\ell + 3m\ell}^{T} \\
-1 & 0 & 0_{2^{p} + \ell - 3}^{T} & 0 & 0 & 1 & 0 & \cdots & 0 & 0 & 0_{\ell + 3m\ell}^{T} \\ 
\vdots & \vdots & \vdots & \vdots & \vdots & & & \ddots & \vdots & \vdots & \vdots \\ 
-1 & 0 & 0_{2^{p} + \ell - 3}^{T} & 0 & 0 & 0 & 0 & \cdots & 1 & 0 & 0_{\ell + 3m\ell}^{T} \\
\end{array}
\right] 
\ \right\} \ \text{\footnotesize $2^{p-1} - (\ell+1)/2$}, \\
b_{\mathrm{l}} 
& = 
\left( 
0_{2^{p} - 1}^{T}, \ 0_{m (\ell-1)}^{T}, \ 0_{\ell (\ell+1)/2}^{T}, \ n/2, \ -\infty \cdot 1_{\ell}^{T} , \ -\infty \cdot 1_{2^{p-1}-(\ell+1)/2}
\right)^{T}, \\
b_{\mathrm{u}} 
& = 
\left( 
0_{2^{p} - 1}^{T}, \ 0_{m (\ell-1)}^{T}, \ 0_{\ell (\ell+1)/2}^{T}, \  n/2, \ 0_{\ell}^{T} , \ 0_{2^{p-1}-(\ell+1)/2}^{T}
\right)^{T}, 
\end{align}
respectively.

\subsection{Constraints for the ranges of the variables}

Besides the constraints $u_{i} \geq 0$ and $0 \leq \hat{w}_{j} \leq 1/2$ 
in~\eqref{eq:opt_expr_SOCP_even} and~\eqref{eq:opt_expr_SOCP_odd}, 
we can add $\hat{g}_{i} \geq 0$ and $t_{ij} \geq 0$ 
because any solution that does not satisfy these constraints is useless. 
Therefore by using the vectors given by
\begin{align}
\tilde{b}_{\mathrm{l}} 
& = 
\left(
0_{2(2^{p}-1) + \ell}^{T}, \ -\infty \cdot 1_{m\ell}^{T}, \ 0_{2m\ell}^{T} 
\right)^{T}, \\
\tilde{b}_{\mathrm{u}} 
& = 
\left(
\infty \cdot 1_{2(2^{p}-1) + \ell + 2m\ell}^{T}, \ (1/2) \cdot 1_{m\ell}^{T} 
\right)^{T},
\end{align}
we can express the constraints for the ranges of the variables as 
$\tilde{b}_{\mathrm{l}} \leq v \leq \tilde{b}_{\mathrm{u}}$.

\subsection{Cone constraints}

To express the cone constraints in~\eqref{eq:opt_expr_SOCP_even} and~\eqref{eq:opt_expr_SOCP_odd} by MOSEK, 
we need to specify the components of the vector $v$ 
that correspond to the constraints. 
For the constraints 
\begin{align}
u_{i}^{2} \leq 2 u_{2i} u_{2i+1} \qquad (i = 1, \ldots, 2^{p-1} - 1), \notag
\end{align}
we take the components $v_{2i}$, $v_{2(2i)-1}$, and $v_{2(2i+1)-1}$ for $u_{i}$, $u_{2i}$, and $u_{2i+1}$, respectively. 
For the constraints 
\begin{align}
u_{i}^{2} \leq 2 \hat{g}_{2i - 2^{p} + 1} \hat{g}_{2i - 2^{p} + 2} \qquad 
\begin{cases}
(i = 2^{p-1} , \ldots, 2^{p-1} + \ell/2 + 1) & (\ell \text{ is even}), \\
(i = 2^{p-1} , \ldots, 2^{p-1} + (\ell-3)/2) & (\ell \text{ is odd}), 
\end{cases}
\notag
\end{align}
we take the components $v_{2i}$, $v_{2^{p} + 2i - 1}$, and $v_{2^{p} + 2i}$ 
for $u_{i}$, $\hat{g}_{2i - 2^{p} + 1}$, and $\hat{g}_{2i - 2^{p} + 2}$, respectively. 
In the case that $\ell$ is odd, we take the components
$v_{2^{p} + \ell - 1}$, $v_{2(2^{p} -1) + \ell}$, and $v_{1}$
for the constraint $u_{2^{p-1} + (\ell-1)/2}^{2} \leq 2 \hat{g}_{\ell} u_{1}$. 
Finally, 
for the constraints
\begin{align}
z_{ij}^{2} \leq 2 t_{ij} \hat{w}_{i} \qquad (i=1,\ldots, m, \ j=1,\ldots, \ell), 
\notag
\end{align}
we take the components $v_{2(2^{p}-1) + \ell + ij}$, $v_{2(2^{p}-1) + \ell + ij + m\ell}$, and $v_{2(2^{p}-1) + \ell + ij + 2m\ell}$.

\end{document}